\documentclass[a4paper,10pt,reqno]{amsart}

\title[Poincar\'e profiles of Lie groups and a coarse geometric dichotomy]{Poincar\'e profiles of Lie groups and a coarse geometric dichotomy}
\thanks{The first author was supported by a Titchmarsh Research Fellowship at the University of Oxford and a Heilbronn Research Fellowship at the University of Bristol.  
  The second author was supported in part by EPSRC grant EP/P010245/1.}
\author{David Hume}
\address{School of Mathematics, University of Bristol, Bristol, BS8 1TX.}
\email{david.hume@bristol.ac.uk}
\author{John M. Mackay}
\address{School of Mathematics, University of Bristol, Bristol, BS8 1TX.}
\email{john.mackay@bristol.ac.uk}
\author{Romain Tessera}
\address{Université de Paris, Sorbonne Université, CNRS, Institut de Mathématiques
de Jussieu-Paris Rive Gauche, F-75013 Paris, France.}
\email{romatessera@gmail.com}
\date{\today}

\usepackage{amsmath,amsthm,amsfonts,graphicx,amssymb}
\usepackage{xcolor}
\usepackage[shortlabels]{enumitem}
\usepackage{tikz}
\usepackage{pgfplots}
\pgfplotsset{compat = newest}

 \usepackage{hyperref}

\numberwithin{equation}{section}
\newtheorem{theorem}[equation]{Theorem}
\newtheorem{proposition}[equation]{Proposition}
\newtheorem{corollary}[equation]{Corollary}
\newtheorem{lemma}[equation]{Lemma}

\theoremstyle{definition}
\newtheorem{fact}[equation]{Fact}
\newtheorem{example}[equation]{Example}
\newtheorem{examples}[equation]{Examples}
\newtheorem{question}[equation]{Question}

\newtheorem{conjecture}[equation]{Conjecture}
\newtheorem{definition}[equation]{Definition}
\newtheorem{remark}[equation]{Remark}

\newtheorem*{theorem*}{Theorem}
\newtheorem*{assump*}{Standing assumption}
\newtheorem*{remark*}{Remark}
\newtheorem*{claim*}{Claim}

\newtheoremstyle{citing}
  {3pt}
  {3pt}
  {\itshape}
  {}
  {\bfseries}
  {}
  {.5em}
  {\thmnote{#3}}

\theoremstyle{citing}

\DeclareMathOperator{\diam}{diam}

\DeclareMathOperator{\asdim}{asdim}
\DeclareMathOperator{\Am}{Am}
\DeclareMathOperator{\GL}{GL}

\DeclareMathOperator{\PSL}{PSL}
\DeclareMathOperator{\SSL}{SL}

\DeclareMathOperator{\cut}{cut}

\DeclareMathOperator{\Aut}{Aut}
\DeclareMathOperator{\Aff}{Aff}

\DeclareMathOperator{\Hom}{Hom}

\DeclareMathOperator{\Osc}{Osc}
\DeclareMathOperator{\arccosh}{arccosh}

\newcommand{\bdry}{\partial_\infty}

\DeclareMathOperator{\Isom}{Isom}

\DeclareMathOperator{\Heis}{Heis}
\DeclareMathOperator{\DL}{DL}
\DeclareMathOperator{\BS}{BS}
\DeclareMathOperator{\ad}{ad}

\newcommand{\set}[1]{\left\{#1\right\}}
\newcommand{\setcon}[2]{\left\{#1 \, :\, #2\right\}}

\newcommand{\abs}[1]{\left\lvert#1\right\rvert}
\newcommand{\norm}[1]{\left\lvert\left\lvert#1\right\rvert\right\rvert}

\newcommand{\SOL}{\mathrm{SOL}}
\newcommand{\Sol}{\mathrm{SOL}}
\newcommand{\NIL}{\text{NIL}}

\newcommand{\bbK}{\mathbb{K}}

\newcommand{\cT}{\mathcal{T}}

\newcommand{\ra}{\rightarrow}
\newcommand{\R}{\mathbb{R}}
\newcommand{\C}{\mathbb{C}}

\newcommand{\N}{\mathbb{N}}
\newcommand{\Z}{\mathbb{Z}}
\newcommand{\HH}{\mathbb{H}}

\def\XXint#1#2#3{{\setbox0=\hbox{$#1{#2#3}{\int}$}
\vcenter{\hbox{$#2#3$}}\kern-.5\wd0}}

\numberwithin{equation}{section}

\begin{document}

\begin{abstract}
 Poincar\'e profiles are analytically defined invariants, which provide obstructions to the existence of coarse embeddings between metric spaces. We calculate them for all connected unimodular Lie groups, Baumslag--Solitar groups and Thurston geometries, demonstrating two substantially different types of behaviour. For Lie groups, our dichotomy extends both the rank one versus higher rank dichotomy for semisimple Lie groups and the polynomial versus exponential growth dichotomy for solvable unimodular Lie groups. We provide equivalent algebraic, quasi-isometric and coarse geometric formulations of this dichotomy. 

As a consequence, we deduce that for groups of the form $N\times S$, where $N$ is a connected nilpotent Lie group, and $S$ is a rank one simple Lie group, both the growth exponent of $N$, and the conformal dimension of $S$ are non-decreasing under coarse embeddings. These results are new even for quasi-isometric embeddings and give obstructions which in many cases improve those previously obtained by Buyalo--Schroeder.  
\end{abstract}

\maketitle
\tableofcontents

\section{Introduction}

The notion of coarse embedding is very natural, 
since the inclusion of one compactly generated locally compact group as a closed subgroup of another automatically yields a coarse embedding with respect to the relevant word metrics.  
While remarkable progress has been made on the much more restrictive class of \emph{quasi-isometric} embeddings -- especially for high rank symmetric spaces and their lattices \cite{KL-97-qi-sym-space, Eskin-Farb-97-symm-spaces, Fisher-Whyte-18-QI-emb-symm-spaces, Fisher-Nguyen-20-QI-emb-nonunif-lattices} -- the techniques involved typically say nothing about coarse embeddings. As a consequence, many natural questions have been so far intractable;  for instance, whether there is a coarse embedding $\HH_\R^3\to\HH_\R^2\times\R^d$ for some $d\in\N$ (cf.\ \cite[Question 5.4]{BenSchTim-12-separation-graphs}).\footnote{Using asymptotic dimension, there are no coarse (or even regular) embeddings for $d=0$. In the case $d=1$, coarse embeddings do not exist by \cite[Corollary 4.49]{Licohom}. }

The separation profile was introduced by Benjamini, Schramm and Timar in \cite{BenSchTim-12-separation-graphs} as a new tool to provide obstructions to {\it regular maps} between bounded degree graphs: in this setting, coarse embeddings are examples of regular maps. In \cite{HumeMackTess}, we introduced a new family of invariants: the $L^p$-Poincar\'e profiles, which interpolate between the separation profile (for $p=1$) and a function of the volume growth (for $p=\infty$). We computed all the $L^p$-Poincar\'e profiles in a number of instances, including rank 1 simple Lie groups and groups of polynomial growth. This already produced new obstructions to coarse embeddings: e.g.\ from $\HH_\HH^2$ to $\HH_\R^{10}$ (see \cite[Corollary 15]{HumeMackTess} for a general statement). In this paper we push this study much further by computing the $L^p$ profiles for all connected unimodular Lie groups, deducing in particular a negative answer to the question above (see Corollary \ref{corIntro:BuySch}).  

For the rest of this introduction, let us call a metric space \textbf{standard} if it is quasi-isometric to a bounded degree graph. The class of such spaces includes bounded degree graphs themselves, which in this paper are assumed to be connected, but also Riemannian manifolds with bounded geometry, and compactly generated locally compact groups equipped with their word metric. Our main focus will be on connected Lie groups, which are duly compactly generated. 

This introduction is organized as follows: 
in \S \ref{secIntro:Background}, we recall the definitions of Poincaré profiles, and of regular, coarse and quasi-isometric embeddings. We then introduce the notions of analytically thin/analytically thick metric spaces.
From there on, we state our results. 
 \S \ref{secIntro:Dicho} contains our first main contribution: we show that the $L^p$-profiles have two distinct types of asymptotic behaviour (analytically thin/thick), and we characterize each one in terms of the Lie algebra of the group (algebraically thin/thick).
 In \S \ref{secIntro:calculationsThin}, the complete calculation of the $L^p$ profiles of unimodular connected Lie groups is given. We also obtain a range of new obstructions to coarse embeddings which mainly (but not exclusively) follow from these calculations of $L^p$-profiles.  
\subsection{Background}\label{secIntro:Background}

\subsubsection{Poincar\'e profiles} 
Poincar\'e inequalities are fundamental tools in analysis, controlling function norms by the norm of their derivatives on a given space.  
For a finite graph $\Gamma$, with vertex set $V\Gamma$ and edge set $E\Gamma$ we can quantify the extent to which an $L^p$-Poincar\'e inequality holds by defining its \textbf{$L^p$--Poincar\'{e} constant}, for $p\in[1,\infty]$:
\[
 h^p(\Gamma) = \inf\setcon{\frac{\norm{\nabla f}_p}{\norm{f}_p}}{f:V\Gamma\to\R,\ \sum_{v\in V\Gamma} f(v)=0, f \not\equiv 0}
\]
where $\nabla f:V\Gamma\to\R$ is defined by $\nabla f(x)=\max\setcon{\abs{f(x)-f(y)}}{xy\in E\Gamma}.$  
For $p=1$, we recover the Cheeger constant of the graph, while for $p=2$, $h^2(\Gamma)^2$ is comparable to the first positive eigenvalue of the graph Laplacian (and indeed would equal it for a different choice of gradient norm). This constant is usually interpreted as a measure of how ``well-connected'' the graph $\Gamma$ is (in particular it is positive if and only if the graph is connected).

Inspired by Benjamini--Schramm--Timar's ``separation profile''~\cite{BenSchTim-12-separation-graphs}, in a previous paper~\cite{HumeMackTess} we used $L^p$-Poincar\'e constants to define a family of invariants for infinite graphs.

\begin{definition}
For an infinite bounded degree graph $X$, we define its \textbf{$L^p$--Poincar\'{e} profile} $\Lambda^p_X:\N\to\R$ to be
\[
 \Lambda^p_X(r) = \sup\setcon{\abs{V\Gamma}h^p(\Gamma)}{\Gamma\leq X,\ \abs{V\Gamma}\leq r}.
\]
\end{definition}
We consider functions up to the natural order $\lesssim$ where $f\lesssim g$ if there exists a constant $C$ such that $f(r)\leq Cg(Cr+C) + C$ for all $r$, and $f\simeq g$ if $f\lesssim g$ and $g\lesssim f$.
As mentioned above, $L^p$-Poincar\'e profiles interpolate between the separation profile (for $p=1$) and a function of the volume growth (for $p=\infty$) \cite{HumeMackTess}.

It turns out that the asymptotic behaviours of Poincaré profiles are invariant under quasi-isometry~\cite{BenSchTim-12-separation-graphs,HumeMackTess}. Hence one can define the Poincaré profiles of a standard metric space to be those of a fixed bounded degree graph which is quasi-isometric to it.  
This definition is indirect and not always useful in practice, but we shall stick to it in this introduction in order to keep the presentation as elementary as possible.  Let us simply mention that it is possible to generalize the definition of Poincar\'e profiles to a class of metric {\it measure} spaces including bounded degree graphs, Riemannian manifolds with bounded geometry and compactly generated locally compact groups equipped with their word metric and Haar measure \cite[\S 4]{HumeMackTess} (we shall use this definition in \S \ref{sec:HeisRtimesZ}).

\subsubsection{Poincar\'e profiles as obstructions to embeddings}

In addition to their natural interest, $L^p$-Poincar\'e profiles are of use as obstructions to regular maps in the sense of \cite{BenSchTim-12-separation-graphs}: if $X$ and $X'$ are bounded degree graphs and there exists a regular map $X\to X'$, then $\Lambda^p_X\lesssim  \Lambda^p_{X'}$ for all $p \in [1,\infty]$ (\cite[Lemma 1.3]{BenSchTim-12-separation-graphs} for $p=1$, \cite[Theorem 1]{HumeMackTess} for all $p$). Recall 

\begin{definition}[{\cite[\S 1.1]{BenSchTim-12-separation-graphs} and \cite[Definition 1.3]{Ben-Sch-Inventiones}}]  
A map $\phi:X\to Y$ between bounded degree graphs is \textbf{regular} if it is Lipschitz and (at most $m$)-to-one for some $m\in \N$, i.e.\  for all $y\in Y$,  $|\phi^{-1}(\{y\}|\leq m$. 
\end{definition}
Note that regularity is stable under post and pre-composition by quasi-isometries. This allows us to define a regular map $\phi:X\to Y$ between two standard metric spaces as follows: if $\Gamma_X$ and $\Gamma_Y$ are bounded degree graphs and $i_X:\Gamma_X\to X$  and $p_Y:Y\to \Gamma_Y$ are quasi-isometries, then $\phi:X\to Y$ is regular if and only if $p_Y\circ\phi\circ i_X$ is regular. By the remark above, this definition is independent of the choice of $\Gamma_X, \Gamma_Y,i_X, p_Y$. 

The prototypical example of regular map between bounded degree graphs is an injective Lipschitz map. In fact, it is easy to see that on replacing $Y$ by $Y\times F$ for some finite graph  $F$, every regular map is at bounded distance from an injective Lipschitz map.

Recall that a map $\phi:(X,d_X)\to (Y,d_Y)$ between metric spaces is a \textbf{coarse embedding} if there exist increasing functions $\rho_{\pm}:[0,\infty)\to [0,\infty)$ such that $\rho_-(r)\to\infty$ as $r\to\infty$ and for all $x,x'\in X$
\[
 \rho_-(d_X(x,x')) \leq d_Y(\phi(x),\phi(x')) \leq \rho_+(d_X(x,x')).
\]
When $\rho_-, \rho_+$ are affine functions $\phi$ is called a \textbf{quasi-isometric embedding}.  In the context of graphs, coarse embeddings are obviously $\rho_+(1)$-Lipschitz. Moreover, for all $y=f(x)$, $\phi^{-1}(\{y\})$ is contained in $B(x,\rho_-^{-1}(0))$. Hence coarse embeddings between bounded degree graphs are regular maps. More generally we deduce that coarse embeddings between standard metric spaces are regular maps.

\subsubsection{Analytically thin versus thick metric spaces}
Now let us focus for the moment on the $p=1$ case, where the Poincar\'e profile is equivalent to the separation profile of Benjamini--Schramm--Tim\'ar.  It follows by definition that for all bounded degree graph $X$, one has $\Lambda^1_X(r)\lesssim r$. 
For virtually nilpotent groups, or Gromov hyperbolic groups, the separation profile has a bound $\Lambda_G^1(r) \lesssim r^a$ for some $a<1$~\cite{BenSchTim-12-separation-graphs}.
Second, for the product of two non-abelian free groups $F \times F$, the separation profile is $\Lambda_{F\times F}^1(r) \simeq r/\log(r)$~\cite{BenSchTim-12-separation-graphs}. The lower bound of $r/\log(r)$ therefore holds for the separation profile of any finitely generated group containing $F\times F$ as a subgroup. More generally, all examples of groups whose separation profile has been calculated exactly satisfy exactly one of the following two properties.

\begin{definition}\label{def:algthickthin}
We say that a standard metric space is {\bf analytically thin} if there exists $a<1$ such that $\Lambda^1_X(r)\lesssim r^a$. On the other hand we call it {\bf analytically thick} if $\Lambda^1_X(r)\gtrsim \frac{r}{\log r}.$
\end{definition}
The corresponding version of Definition \ref{def:algthickthin} for the $L^\infty$ profile is that ``thin'' spaces have (at most) polynomial growth and ``thick'' spaces have exponential growth (this follows from \cite[Proposition 6.1]{HumeMackTess}). Many classes of groups such as linear and elementary amenable groups do not contain any intermediate growth groups. More specifically for connected Lie groups, this dichotomy has an elegant algebraic formulation which can be read off the Lie algebra \cite{Gui-73-crois-poly}.  One of the main objectives of this paper is to show that connected unimodular Lie groups are either analytically thick or analytically thin, and to provide nice and workable algebraic translations of these properties.

\subsection{An analytic, geometric and algebraic dichotomy}\label{secIntro:Dicho}

\subsubsection{A dichotomy for connected unimodular Lie groups}\label{secIntro:dichoLie}
\label{ssec:intro-main}
Connected Lie groups offer a fascinating playground for exploring the relationship between the algebraic properties of a group and the geometric properties of the metric spaces on which it acts, as their algebraic properties are conveniently encoded in the Lie algebra. 
There are many examples of these relationships, including the already mentioned  algebraic characterization of Lie groups of polynomial growth; Varopoulos's classification according to the large time behaviour of symmetric random walks \cite{VSC} using both analytic and geometric methods; 
Pansu's $L^p$-cohomology methods characterizing Gromov hyperbolicity for such groups \cite{PanCohLp,CorTes-LpCoh};  and Cornulier--Tessera's algebraic characterization of Lie groups whose Dehn
function is polynomially bounded \cite{CorTesIHES}. 

Our main result has a similar flavour as it consists in identifying algebraic counterparts of being analytically thin/thick.  
\begin{definition}\label{defn:algthickthin}
	A connected Lie group $G$ with solvable radical $R$ and Levi factor $S$ is \textbf{algebraically thin} if 
	\begin{itemize}
\item its $\R$-rank is  at most $1$;
\item $[S_{\mathrm{nc}},R]=1$, where $S_{\mathrm{nc}}$ is the non-compact part of $S$;
\item $R$ is an NC-group. 
	\end{itemize}
Otherwise it is called \textbf{algebraically thick}.
\end{definition}
The concept of NC-group appears in various articles by Varopoulos, for instance \cite[\S 1.2]{Varo96}; we refer to \S\ref{ssec:NC-groups} for the full definitions of this and of $\R$-rank.
In the case $R$ is a solvable connected real Lie group, then $R$ is an NC-group if it admits a closed normal subgroup $E$ such that $R/E$ has polynomial growth and some element of $R$ acts on $E$ as a contraction, see Lemma~\ref{lem:NC}.  This includes the case $R$ itself has polynomial growth, since any $\alpha$ acts on $E=\{1\}$ as a contraction (since $\alpha^n$ converges to the identity on compact sets).  If such solvable connected $R$ is unimodular, then it is NC if and only if has polynomial growth.

The examples of algebraically thin groups that will be most important to us are direct products $R\times S$ where $R$ has polynomial growth and $S$ is either trivial or semisimple of rank $1$. While we already know that rank $1$ semisimple Lie groups and connected Lie groups of polynomial growth are analytically thin, we are now able to show that their direct product is as well. More generally, we have the following dichotomy.

\begin{theorem}\label{thmIntro:unimodularDicho} Let $G$ be a connected unimodular Lie group. Then $G$ is algebraically thin (resp.\ thick) if and only if it is analytically thin (resp.\ thick). Moreover, if it is algebraically thick, then $\Lambda^p_G(r)\simeq r/\log(r)$ for every $p\in[1,\infty]$.
\end{theorem}

As each polycyclic group is virtually a uniform lattice in a connected unimodular solvable Lie group, such groups satisfy 
a similar dichotomy.

\begin{corollary}\label{corIntro:polycyclic}
	Let $G$ be a polycyclic group. If $G$ has polynomial growth, then it is analytically thin. Otherwise, it is analytically thick, and moreover satisfies $\Lambda^p_G(r)\simeq r/\log(r)$ for every $p\in[1,\infty]$.
	\end{corollary}

Some remarks are in order.
\begin{enumerate} 
\item
	Since connected Lie groups have finite Assouad--Nagata dimension, from \cite[\S 9]{HumeMackTess} we deduce that every connected Lie group $G$ satisfies $\Lambda^p_G(r)\lesssim r/\log(r)$ for every $p\in[1,\infty]$, giving sharp upper bounds for analytically thick groups.

	\item Further examples of NC-groups include all direct products $R=H\times N$ where  $N$ is a connected real nilpotent Lie group and $H$ is a \textbf{Heintze group} (i.e.\ of the form $E\rtimes\R$ with every positive element of $\R$ acting as a contraction on $E$). To see this, write $R= E\rtimes (\R\times N)$: it is clear that $E$ is a closed normal subgroup of $R$ and that $R/E$ has polynomial growth. Finally, any non-trivial element (or its inverse) of the $\R$ factor acts on $E$ as a contraction.
	\item The condition $[S_{\mathrm{nc}},R]=1$ appears in various works dealing with algebraic characterizations of certain analytic properties of Lie groups. A first occurrence of this condition appears in Varopoulos's work on the diffusion of the heat kernel in \cite[\S 1.8]{Varo96}, as reported in \cite[Theorem 7.1]{CPS}: in his context, $G$ is assumed to be unimodular, and $R$ of polynomial growth. 
It appears in the characterization of the Haagerup property \cite{CCV},  and of weak amenability \cite{CDSW}: there $S_{\mathrm{nc}}$ is assumed to contain only certain rank one factors. More recently, 
Chatterji, Pittet and Saloff-Coste proved that a connected Lie
		group has Property RD if and only if its Lie algebra has the form $\mathfrak{r} \rtimes \mathfrak{s}$ with $[\mathfrak{s}_{\mathrm{nc}}, \mathfrak{r}] = 0$ and $\mathfrak{r}$ has type R \cite[Theorem 0.1.]{CPS}.
\end{enumerate}

\subsubsection{More characterizations of algebraically thin groups}\label{secIntro:thin}

Let $G$ be an algebraically thin connected Lie group with solvable radical $R$ and Levi factor $S$. Then exactly one of the following holds (see Proposition \ref{prop:algthinDescrip}):
	\begin{itemize}
		\item[(a)] $G$ has polynomial growth;
		\item[(b)] $S$ has $\R$-rank $1$, $[S_{\mathrm{nc}},R]=1$ and $R$ has polynomial growth; or
		\item[(c)] $S$ is compact and $R$ is an NC-group with $\R$-rank $1$.
	\end{itemize}
We observe that (a) and (b) exactly correspond to the case where $G$ is unimodular (as an NC-group is unimodular if and only if it has polynomial growth). We prove that unimodular algebraically thin groups reduce up to quasi-isometry to the following class of groups (see Corollary \ref{cor:algthinGeneral} for a more algebraic statement which implies this one).
\begin{proposition}\label{propIntro:thinQIproduct}
An algebraically thin connected unimodular Lie group is quasi-isometric to either $P$ or a direct product $P\times \HH_\bbK^m$, where $P$ is a connected Lie group of polynomial growth, $\bbK\in \{\R,\C,\HH,\mathbb{O}\}$ and $m\geq 2$ with $m=2$ when $\bbK=\mathbb{O}$. 
\end{proposition}

 Below is an easy consequence of Theorem \ref{thmIntro:unimodularDicho} -- and the Bonk-Schramm embedding theorem \cite{BS-00-gro-hyp-embed} -- providing an algebraic characterization of connected unimodular Lie groups which admit certain embeddings into certain standard product spaces.
\begin{theorem}\label{thmIntro:geomDichoThin} Let $G$ be a connected unimodular Lie group.
The following are equivalent:
\begin{itemize}
\item[(i)] $G$ is algebraically thin;
\item[(ii)] $G$ admits a regular map into  $\R^n\times \HH_\R^m$ for some $n, m \geq 0$;
\item[(iii)] idem with `coarse embedding'; 
\item[(iv)] $G$ admits a quasi-isometric embedding into $P \times \HH_\R^m$ for some connected Lie group $P$ with polynomial growth, and some $m\geq 0$.
\end{itemize} 
\end{theorem}
Although it only applies to unimodular groups, this theorem should be compared with Cornulier's algebraic characterization of connected Lie groups admitting a quasi-isometric embedding into a CAT(0) space \cite{Cornulier_dimcone}.

Note that in (iv) one needs such a $P$ rather than $\R^n$ since, for example, the Heisenberg group does not quasi-isometrically embed into any $\R^n$.

\subsubsection{More characterizations of algebraically thick groups}\label{secIntro:thick}
 The following result is a partial version of Theorem \ref{thmIntro:unimodularDicho} valid without the unimodularity condition.

\begin{theorem}\label{thmIntro:alg thick Implies an thick} Let $G$ be a connected Lie group. If $G$ is algebraically thick, then $\Lambda^p_G(r)\simeq r/\log(r)$ for every $p\in[1,\infty]$. In particular $G$ is analytically thick.
\end{theorem}

An important ingredient in the proof of Theorem \ref{thmIntro:alg thick Implies an thick} is the following useful characterization of algebraically thick Lie groups. 
 The following proposition says that within a slightly restricted  class of connected Lie groups, there are two types of `minimal' algebraically thick groups:
$\Sol_a=\R^2\rtimes_{(1,-a)}\R$, for $a>0$, and the \textbf{split oscillator group} $\Osc=\Heis_3\rtimes_{(1,-1,0)}\R$. We denote by $\mathfrak{sol}_a$ and $\mathfrak{osc}$ their respective Lie algebras.

  \begin{proposition}\label{prop:algThick} Let $G$ be a connected linear real Lie group whose radical $R$ is real-triangulable. The following are equivalent:
	\begin{itemize}
		\item[(i)] $G$ is algebraically thick;
		\item[(ii)]  $G$ admits a closed undistorted subgroup isomorphic to either $\Sol_a$ for some $a>0$, or $\Osc$;
		\item[(iii)]  $\mathfrak{g}$ has a Lie subalgebra isomorphic to either $\mathfrak{sol}_a$ for some $a>0$, or $\mathfrak{osc}$.
	\end{itemize}
\end{proposition}
Recall that a \textbf{real-triangulable Lie group} is a connected, simply connected Lie group which admits a continuous faithful triangulable real representation.
To the best of our knowledge, this characterization is new, but its statement and proof are similar to previous works (for instance \cite[Proposition 8.2]{CDSW}). The assumptions that the group is linear and has real-triangulable radical are here only to avoid inessential complications. We shall indeed see that an algebraically thick connected Lie group is quasi-isometric to an algebraically thick connected Lie group of that form (see Theorem \ref{thm:reduc} for a more precise statement).

We recall that connected Lie groups have finite Assouad--Nagata dimension, and therefore by \cite{HumeMackTess}, they satisfy $\Lambda_{G}^p(r) \lesssim r/\log(r)$ for every $p\in[1,\infty]$.
With Proposition \ref{prop:algThick} at hand, the proof of Theorem \ref{thmIntro:alg thick Implies an thick}  boils down to showing that $\Sol_a$ and $\Osc$ both satisfy $\Lambda^p_G(r)\gtrsim r/\log(r)$ (Theorems \ref{DLlowerbound} and  \ref{thm:HeisExt}).
The lower bound in the case $\Osc$ is treated by a direct (involved) computation which uses the definition of Poincar\'e profiles of metric measure spaces from \cite{HumeMackTess}. We proceed more indirectly for $\Sol_a$: indeed, we first prove that the Diestel-Leader graph $\DL(2,2)$ quasi-isometrically embeds into it, and then that $\Lambda^p_{\DL(2,2)}(r)\simeq r/\log(r)$. 

We have the following geometric characterization of unimodular algebraically thick groups. 

\begin{theorem}\label{thmIntro:geomDichoThick}
Let $G$ be a connected unimodular Lie group or a polycyclic group. The following are equivalent:
\begin{itemize}
		\item[(i)] $G$ is algebraically thick;
		\item[(ii)]  either $\DL(2,2)$, or $\Osc$ {\it regularly maps} to $G$;
		\item[(iii)]  idem with  ``{\it coarsely embeds} into $G$'';
		\item[(iv)]  idem with ``{\it quasi-isometrically embeds} into $G$''.
	\end{itemize}
\end{theorem}

\subsubsection{A word on the non-unimodular case}\label{sectionIntro:nonunimodular}
The question whether all non-uni\-modu\-lar algebraically thin groups are analytically thin remains open.  We can prove it though when $G$ is 
{\it a direct product of a group of polynomial growth with a Heintze group}. We shall provide explicit upper and lower bounds on their Poincaré profiles below. For now, let us indicate an indirect argument showing that they are analytically thin:
By Heintze's theorem \cite{Heintze}, a Heintze group admits a negatively curved left-invariant Riemmanian metric, and therefore is Gromov hyperbolic. Applying the Bonk-Schramm embedding theorem, we see that every Heintze group quasi-isometrically embeds into some $\HH_\R^n$. Next, every group of polynomial growth satisfies the doubling property, so by Assouad's embedding theorem we obtain a coarse embedding into some $\R^m$ \cite{Assouad}. Thus, the product of a Heintze group with a group of polynomial growth coarsely embeds into some $\HH^n_\R\times\R^m$, and is therefore analytically thin by Theorem \ref{thmIntro:unimodularDicho}.

More generally, any {\it hypercentral-by-Heintze group} is analytically thin: this is because such a group is a closed subgroup of a direct product of a nilpotent connected Lie group with a Heintze group (see Corollary \ref{cor:hypercentralbyHeintze}). An example is the semi-direct product of $\R^3\rtimes\R$, where the action of $\R$ is through matrices
$\left(\begin{array}{ccc} 1 & t & 0 \\ 0 & 1 & 0 \\ 0 & 0 & e^t \end{array}\right)$ (this appears as a special case of Corollary \ref{cor:NCabelianbyR}).

The smallest example of an algebraically thin (actually an NC-group) for which we are unable to prove analytic thinness is $G=\Heis_3\rtimes_{(1,0,1)}\R$, also isomorphic to $\R^2\rtimes \R^2$, where the first factor acts through matrices  $\left(\begin{array}{cc} e^t & 0 \\ 0 & e^t \end{array}\right)$ and the second through matrices $\left(\begin{array}{cc} 1 & t \\ 0 & 1 \end{array}\right)$.

\subsection{Poincar\'e profiles and obstructions to regular maps}\label{secIntro:calculationsThin}

\subsubsection{Precise calculation of Poincaré profiles of thin groups} 
A slightly disappointing consequence of Theorem \ref{thmIntro:alg thick Implies an thick} is that Poincar\'e profiles do not allow to distinguish between algebraically thick groups. By contrast, they provide very refined invariants for unimodular algebraically thin groups, as shown by the combination of Proposition \ref{propIntro:thinQIproduct} and the following result.

\begin{theorem}\label{thmIntro:profilesDirectProductLie}
Let $X$ be a direct product $P\times H$, where $P$ is a connected Lie group of polynomial growth of degree $d\geq 0$, and $H$ is one of the following:
\begin{itemize}
	\item (a uniform lattice in the group of isometries of) $\HH^m_\bbK$, for some $\bbK\in \{\R,\C,\HH,\mathbb{O}\}$ and $m\geq 2$ with $m=2$ when $\bbK=\mathbb{O}$, and with $Q=(m+1)\dim_\R\bbK - 2$, or more generally
\item a Gromov hyperbolic discrete group whose conformal dimension $Q$ is attained by a metric admitting a $1$-Poincar\'e inequality, or
\item a non-abelian free group of finite rank with $Q=0$.
\end{itemize}
Then 
\begin{equation*}
 \Lambda^p_{X}(r) \simeq \left\{
 \begin{array}{lcl}
 r^{1-\frac{1}{Q+d}} & \textup{if} & 1\leq p < Q
 \\
 r^{1-\frac{1}{Q+d}} \log^{\frac{1}{Q+d}}(r) & \textup{if} & p = Q
 \\
 r^{1-\frac{1}{p+d}} & \textup{if} & Q<p<\infty.
 \end{array}\right.
\end{equation*}
\end{theorem} 
Theorem \ref{thmIntro:profilesDirectProductLie} also applies to the case where $P$ is a finitely generated virtually nilpotent group as any such group is quasi-isometric to a connected nilpotent Lie group.

The proof of this result uses the metric structure of the boundary at infinity $\bdry H$ of $H$.
Recall that the boundary at infinity $\bdry G$ of a Gromov hyperbolic group $G$ admits a `visual metric' which is \textbf{Ahlfors $Q$-regular} for some $Q \geq 1$: the measure of balls of radius $r$ is comparable to $r^Q$.  The \textbf{conformal dimension} of $G$ is the infimum of values of $Q$ so that $\bdry G$ is quasisymmetric to an Ahlfors $Q$-regular space, and is a quasi-isometric invariant of $G$~\cite{Pan-89-cdim}.  We say the conformal dimension of $G$ is \textbf{attained} if this infimum is a minimum. 
 In certain cases one can find a metric on the boundary of a hyperbolic group which admits a `$1$-Poincar\'e inequality' in the sense of Heinonen and Koskela.  This is the case for (uniform lattices in) rank 1 simple Lie groups $\Isom(\HH_\bbK^m)$, where the conformal dimension is $(m+1)\dim_\R\bbK-2$ as in Theorem \ref{thmIntro:profilesDirectProductLie},  and the isometry groups for a family of Fuchsian buildings studied by Bourdon and Bourdon--Pajot, where the conformal dimension can take a dense set of values in $(1,\infty)$ (cf.\ the discussion in \cite[\S 11]{HumeMackTess}).

Note that Theorem \ref{thmIntro:profilesDirectProductLie} is new even when $p=1$ and $H\times P$ is quasi-isometric to $\text{PSL}(2,\R)\times\R$, or equivalently $\HH^2_\R\times\R$.  In this case, the correct lower bound of $r^{1/2}\log^{1/2}(r)$ was found by Benjamini--Schramm--Tim\'ar \cite[Corollary $3.3$]{BenSchTim-12-separation-graphs}. 

Observe that these computations show that the polynomial growth exponent is monotonous under regular maps. For instance, it follows from Theorem 1.12 that
$\HH^2_\R\times\R$ does not regularly map  into $\HH^n_\R$ for any $n\geq 2$, but this does not yet rule out a regular map from $\HH^3_\R$ to $\HH^2_\R\times \R$. We shall see in Theorem \ref{thmIntro:directProductCE} that the quantity $Q$ is also monotonous under regular maps when the domain satisfies the hypotheses of Theorem \ref{thmIntro:profilesDirectProductLie}.

Theorem \ref{thmIntro:profilesDirectProductLie} is a consequence of a more general statement, which we state as two theorems.  
One of them is general upper bound on the Poincaré profile (Theorem \ref{thm:hyp-times-nilp-upper-bound}), and the other is a lower bound for hyperbolic spaces whose boundaries admit a $1$-Poincar\'e inequality (Theorem \ref{thm:hyp-x-nilp-lower} and Corollary \ref{cor:treexpoly-lower}). 
These bounds give information in other cases too, such as Heintze groups, see \S\ref{sec:qu}.

\subsubsection{Obstructions to regular maps}\label{sec:obstructions}

Below is a general non-embeddability result which cannot be solely deduced from Poincaré profile estimations.

\begin{theorem}[Corollary \ref{cor:prod-no-embed}]\label{thmIntro:directProductCE}
Assume $G_1=H_1\times P_1$ and  $G_2=H_2\times P_2$, where for $i=1, 2:$
\begin{itemize}
\item $H_i$ is a non-elementary finitely generated hyperbolic group of conformal dimension $Q_i \geq 0$, and
\item $P_i$ is a locally compact group with polynomial growth of degree $d_i \geq 0$.
\end{itemize}
 If there exists a regular map $G_1\to G_2$, then $d_1 \leq d_2$. Moreover, if $H_1$ has its conformal dimension $Q_1>1$ attained by a metric admitting a $1$-Poincar\'e inequality, then $Q_1\leq Q_2$.
\end{theorem}

Here is a specialization of the above theorem to the family of connected Lie groups from Proposition \ref{propIntro:thinQIproduct}.
\begin{corollary}\label{corIntro:BuySch} If there is a regular map $\HH^{m_1}_{\mathbb K_1}\times\R^{d_1} \to \HH^{m_2}_{\mathbb K_2}\times\R^{d_2}$, then $(m_1+1)\dim_\R(\mathbb K_1)-2 \leq (m_2+1)\dim_\R(\mathbb K_2)-2$ and $d_1\leq d_2$.
\end{corollary}
This corollary answers \cite[Question 5.4]{BenSchTim-12-separation-graphs}, which asked for an obstruction to the existence of a regular map $\HH^3_\R\to\HH^2_\R\times\R$, indeed we show that there is no regular map $\HH^3_\R\to\HH^2_\R\times\R^d$ for any $d$ using the monotonicity of the conformal dimension of the hyperbolic factor.

There are several further points to note about Theorems~\ref{thmIntro:profilesDirectProductLie}, \ref{thmIntro:directProductCE} and Corollary~\ref{corIntro:BuySch}:

\begin{remark} \ 
	\begin{enumerate}
		\item Theorems \ref{thmIntro:profilesDirectProductLie} and \ref{thmIntro:directProductCE} (and the more general upper bound in Theorem \ref{thm:hyp-times-nilp-upper-bound}) are new even in the case of hyperbolic groups (i.e.\ $d=0$, respectively $d_1=d_2=0$), since the technical hypothesis about ``equivariant conformal dimension'' from our previous paper is no longer needed \cite[Corollary 12.6]{HumeMackTess}. 
However for a different class of maps including coarse embeddings, Pansu earlier ruled out maps $H_1\to H_2$ unless $Q_1\leq Q_2$ for groups satisfying the second part of Theorem \ref{thmIntro:directProductCE} \cite[Corollary 1]{Pan-16-coarse-conformal}.
		
		\item One particular case of Theorem~\ref{thmIntro:directProductCE} is that there is no regular map from $F_2\times \Z$ (i.e.\ the product of a $4$-regular tree and a line) to any hyperbolic group.
		\item The fact that $H_1$ is non-elementary is important, since there certainly are coarse embeddings $\R^d \to \HH_\R^{d+1}$ or $\R^d \to \HH_\R^d\times \R$, etc., using horospheres.
		\item The monotonicity of $d$ in Theorem \ref{thmIntro:directProductCE} does not follow from the separation profile alone, and indeed to deduce that $d_1\leq d_2$ above, we will need to consider $L^p$-Poincar\'e profiles with $p>Q_1,Q_2$. 
		\item As already mentioned, the Poincaré profile is not enough to imply the monotonicity of the conformal dimension of the hyperbolic factor in Theorem \ref{thmIntro:directProductCE} and Corollary \ref{corIntro:BuySch}. This is obtained by a different argument based on the same techniques in  \S \ref{sec:beyondProfile}.

		\item For coarse embeddings, obstructions in the case where $d_1>0$ and $d_2=0$ are obtained in \cite{HumeSisto_coarsehyp}. Apart from this case, we believe that the monotonicity of the growth exponent of the polynomial factor is new.

		\item It is natural to ask whether the last statement of Theorem~\ref{thmIntro:directProductCE} holds for $Q_1=1$, that is, can one show there is no regular map $\HH_\R^2 \to F_2 \times \R^n$ for any $n \geq 0$.  We answer this question using different methods in forthcoming work~\cite{HumeMackTess-genasdim}.
	\end{enumerate}
\end{remark}

\subsubsection{Comparison with  `dimension-based' obstructions}\label{sec:comparisonDimResults}
It is worth comparing our results with those obtainable using Gromov's asymptotic dimension and its variants.  
If there is a coarse embedding $X \to Y$ then $\asdim X \leq \asdim Y$, and in fact the same is true for regular maps \cite[\S 6]{BenSchTim-12-separation-graphs}.
Asymptotic dimension does not rule out maps $\HH_\R^k \to \HH_\R^{k-1}\times \R^d$, since for $d \geq 1$ we have $\asdim \HH_\R^k = k \leq k-1+d = \asdim (\HH_\R^{k-1}\times\R^d)$. 
Buyalo--Schroeder \cite{BuSh-07-hyp-tree-bound-below} used a variation on asymptotic dimension to show that if there is a quasi-isometric embedding $\HH^{m_1}_{\bbK_1} \times \R^{d_1} \to \HH^{m_2}_{\bbK_2} \times \R^{d_2}$ we must have monotonicity of the asymptotic dimension of the hyperbolic factors, that is $m_1 \dim_\R (\bbK_1) \leq m_2 \dim_\R(\bbK_2)$.
Thus they can rule out quasi-isometric embeddings $\HH_\R^k \to \HH_\R^{k-1}\times \R^d$, however their variation does not behave well with respect to coarse or regular maps.

Corollary \ref{corIntro:BuySch} applies to regular maps, shows monotonicity of the growth exponent of the Euclidean factor, and when the hyperbolic factor in the domain has large conformal dimension compared to its asymptotic dimension we get bounds stronger than those of Buyalo--Schroeder.
For example, if there is a quasi-isometric embedding $\HH^{2}_\HH \to\HH^m_\R\times \R^d$ then Buyalo--Schroeder get $m \geq 8$, while Corollary \ref{corIntro:BuySch} gives $m\geq 11$.
On the other hand, when the hyperbolic factor in the codomain has large conformal dimension but small asymptotic dimension, Buyalo--Schroeder's bound may be stronger.  
For instance, if there is a quasi-isometric embedding $\HH^{11}_\R\to\HH^m_\C\times \R^d$ then Corollary \ref{corIntro:BuySch} gives $m\geq 5$, while Buyalo--Schroeder can conclude that $m\geq 6$.

\subsubsection{Further results and applications}\label{sec:furtherApp}
Another natural class of groups which satisfies the thick/thin dichotomy are Baumslag-Solitar groups.

\begin{theorem}[\S \ref{sec:lbproducts}]\label{thmIntro:BS}
 For all $p\in[1,\infty)$
\[ 
 \Lambda^p_{\BS(m,n)}(r) \simeq_p \left\{
		 \begin{array}{lcl}
		 r^{\frac12} & \textup{if} & \abs{m}=\abs{n}=1,
		 \\
		 r^{1-\frac{1}{p+1}} & \textup{if} & \abs{m}=\abs{n}\geq 2,
		 \\
		 r/\log(r) & \textup{if} & \abs{m}\neq\abs{n}.
		 \end{array}\right.
\]
\end{theorem}
The lower bound in the case $\abs{m}\neq\abs{n}$ of Theorem \ref{thmIntro:BS} is proved by showing that $\BS(m,n)$ admits a quasi-isometrically embedded copy of $\DL(2,2)$. Theorem \ref{thmIntro:BS} implies that a Baumslag--Solitar group regularly embeds into some hyperbolic group if and only if it is virtually abelian, generalising results for coarse embeddings in \cite{HumeSisto_coarsehyp}.

Next, we observe that the $L^1$-Poincar\'e profile (i.e., the separation profile) distinguishes the non-compact  Thurston geometries, except of course for the quasi-isometric $\HH^2_\R\times\R$ and $\widetilde{\PSL(2,\R)}$:
{\renewcommand{\arraystretch}{1.5}\begin{table}[h]
\begin{tabular}{c||c|c|c|c|c|c}
	$X$ & $\mathbb{S}^2\times\mathbb{R}$ & $\mathbb{H}^3_\R$ & $\mathbb{H}^2_\R\times\mathbb{R}$,\ $\widetilde{\PSL(2,\R)}$  &  $\mathbb{R}^3$ & $\NIL$ & $\SOL$ \\ \hline
	$\Lambda^1_X(r)$ & $1$ & $r^{\frac12}$& $r^{\frac12}\log^{\frac12}(r)$ & $r^{\frac23}$ & $r^{\frac34}$ & $r/\log(r)$
\end{tabular}
\end{table}}

The next result is a direct consequence of the fact that spaces admitting regular maps into analytically thin spaces are themselves analytically thin.

\begin{corollary}\label{corIntro:dichotobstruction} Let $H$ be a locally compact group which contains a closed subgroup isomorphic to any of the following:
\begin{itemize} 
 \item a wreath product $K\wr L$ where $K$ is nontrivial and $L$ is infinite finitely generated;
 \item a Baumslag--Solitar group $\BS(m,n)=\langle a,t\mid ta^mt^{-1}=a^n\rangle$ with $|m|\neq|n|$;
 \item a solvable group of exponential growth;
 \item a uniform lattice in a semisimple Lie group of real rank $\geq 2$.
 \end{itemize}
 Then there is no regular map $H \to X\times N$ whenever $X$ is a bounded degree hyperbolic graph and $N$ a nilpotent group.
\end{corollary}
Our results prove that the wreath products, Baumslag-Solitar groups and lattices mentioned above are all analytically thick. Le Coz-Gournay prove that solvable groups of exponential growth are not analytically thin  \cite{CozGour}.

Obstructions of coarse embeddings of the groups $H$ considered in Corollary \ref{corIntro:dichotobstruction} into hyperbolic groups were established in \cite{HumeSisto_coarsehyp}, but as far as we are aware the stronger result of Corollary \ref{corIntro:dichotobstruction} is new even for quasi-isometric embeddings.

\

The monotonicity of the dimension of the Euclidean factor also has applications, for instance it provides a coarse geometric proof of a result in Lorentzian geometry originally due to Zeghib \cite[Theorem 4.2(i)]{Zeg_idcomp}.

\begin{corollary}\label{corIntro:Lorentz}
Let $G$ be the identity component of the isometry group of a compact Lorentz manifold. If $G$ has a closed subgroup $H$ locally isomorphic to $\PSL(2,\R)$, then it has finite center. Moreover, if it has a closed subgroup locally isomorphic to  $\PSL(2,\R) \times\R$, then the abelian factor is compact. 
\end{corollary}
The proof goes as follows: by a fundamental observation of Gromov, $G$ coarsely embeds into some real hyperbolic space (cf.\ \cite{Gromov-Rigid,Frances}). If $H$ had infinite center, then it would be quasi-isometric to $\HH_\R^2\times \R$. So we invoke Corollary \ref{thmIntro:directProductCE} (or  \cite{HumeSisto_coarsehyp}) which implies that $\HH_\R^2\times\R$ does not coarsely embed into any real hyperbolic space. In the second case, we  similarly argue that $G$ would otherwise contain a closed subgroup quasi-isometric to 
$\HH_\R^2\times\R$.

\subsection{Plan of the paper}
\S\ref{sec:AlgebraicDich} is dedicated to all the Lie theoretic results that are needed in the paper. In particular, Theorem \ref{thm:reduc} proves that any connected Lie group is commable to a group of the form (real-triangulable)$\rtimes$(linear semisimple), and that this reduction preserves all the properties that are relevant to us. In \S \ref{sec:nondistotedsubg}, we prove that subgroups isomorphic to $\SOL_a$ and $\Osc$ are always closed and non-distorted. Finally, and most importantly, \S \ref{sec:AlgThick} is dedicated to the proof of Proposition \ref{prop:algThick}. 

 In \S \ref{sec:DL}, we show that $\DL(2,2)$ quasi-isometrically embeds into various groups. In \S \ref{sec:Poincare n/logn}, we prove that the Poincar\'e profiles of $\DL(2,2)$  and $\Osc$ are $\gtrsim r/\log r$; in both cases the proof relies on curve counting arguments. 

In \S \ref{sec:capacityProfile} and \S \ref{sec:product-spaces}, we study direct products of hyperbolic spaces with locally compact groups of polynomial growth. 
Complete Poincar\'e profile calculations for the individual factors appear in our previous work \cite{HumeMackTess} and many of these techniques are also needed to consider products. 
The lower bound follows fairly quickly from our previous work on considering product subgraphs, but the upper bound is much more challenging as we have to find good functions on arbitrary subgraphs of $H\times P$.  This entails developing a theory of `capacity profiles' of weighted graphs arising from projections onto $H$, and a general upper bound formula for products which may be useful in other contexts (Theorem \ref{thm:wt-prof-product-vnilp}); it also results in a connection between conformal dimension and hyperbolic cones (Theorem~\ref{thmIntro:Ahlreg}).
The last part of the section, \S \ref{sec:beyondProfile} is dedicated to the proof of our non-embeddability result Theorem \ref{thmIntro:directProductCE}, which is not a direct consequence of our calculations of Poincar\'e profiles but is based on similar ideas. 
In \S \ref{sec:concludingsteps}, we end the proofs of Theorems \ref{thmIntro:alg thick Implies an thick}, 
\ref{thmIntro:unimodularDicho}, \ref{thmIntro:geomDichoThick}, \ref{thmIntro:geomDichoThin}, and Corollary \ref{corIntro:polycyclic}.
Finally in \S\ref{sec:qu} we raise some open questions.

\subsection{Acknowledgments}
We would like to thanks Charles Frances for communicating Corollary~\ref{corIntro:Lorentz} and its proof to us, Marc Bourdon for comments on a previous version of the paper. We are especially indebted to Yves Cornulier for greatly helping us with the Lie-theoretic aspects. 
We also thank two anonymous referees for their corrections and very helpful  suggestions, particularly regarding the results of \S\ref{sec:AlgebraicDich}.

\section{Lie theoretic results}\label{sec:AlgebraicDich}
\subsection{NC-groups and algebraically thin groups}\label{ssec:NC-groups}
In this subsection we elaborate on the Lie theoretic definition of algebraically thin groups. 
We start by recalling the notion of weight used in the definition of NC-groups. We refer to \cite[\S 1.2]{Varo96} for more details. 
We let $\mathfrak{r}$ be a solvable real Lie algebra. 
Denote by $\ad_\C$ the adjoint action of $\mathfrak{r}$ on $\mathfrak{r}_\C:=\mathfrak{r}\otimes \C$.
A \textbf{root} $\lambda:\mathfrak{r}\to \C$ is a Lie algebra morphism such that \[\bigcap_{y\in \mathfrak{r}}\ker(\ad_\C y -\lambda(y))\neq 0.\] 
A \textbf{weight} is the real part of a root. Note that roots may be viewed as elements of $\Hom(\mathfrak{r}_{\mathrm{ab}},\C)$, where $\mathfrak{r}_{\mathrm{ab}}$ is the abelianization of $\mathfrak{r}$, and that weights are elements of the dual real vector space $\mathfrak{r}_{\mathrm{ab}}^*$ of $\mathfrak{r}_{\mathrm{ab}}$. 
We now extend the usual notion of rank of a semisimple Lie group to arbitrary connected Lie groups, as suggested by a referee.
\begin{definition}\label{def:total real rank}
	The \textbf{$\R$-rank} of a connected solvable Lie algebra $\mathfrak{r}$ is the dimension of the subspace of $\mathfrak{r}_{\mathrm{ab}}^*$ spanned by the weights. 
The $\R$-rank of a connected Lie algebra is the sum of the rank of its semisimple part, and of its solvable radical. Finally the $\R$-rank of a connected Lie group is the $\R$-rank of its Lie algebra.
\end{definition}

\begin{remark}\label{rem:rank0}
	Note that a connected Lie group has real $\R$-rank $0$ if and only if it has polynomial growth  (see for instance \cite{Gui-73-crois-poly}). 
\end{remark}

Our definition of algebraically thin (Definition~\ref{def:algthickthin}) uses the following definition of Varopoulos.

\begin{definition}[{\cite[\S 1.2]{Varo96}}] \label{def:NC}
A solvable Lie algebra has {\bf Property C} if $0$ is in the convex hull of
		non-zero weights; else it has {\bf Property  NC}.  A solvable connected Lie group has Property C (resp.\ NC) if its Lie algebra has C (resp.\ NC). 
\end{definition}
We now show the following trichotomy stated in the introduction.

\begin{proposition}\label{prop:algthinDescrip}
Let $G$ be an algebraically thin connected Lie group with solvable radical $R$ and Levi factor $S$. Then exactly one of the following holds: 
\begin{itemize}
		\item[\textrm{(a)}] $G$ has polynomial growth;
		\item[\textrm{(b)}] $S$ has $\R$-rank $1$, $[S_{\mathrm{nc}},R]=1$ and $R$ has polynomial growth; or
		\item[\textrm{(c)}] $S$ is compact and $R$ is an NC-group with $\R$-rank $1$.
	\end{itemize}
	Moreover $G$ is unimodular if and only if we are in cases (a) or (b).
\end{proposition}
\begin{proof}
Recall that an algebraically thin connected Lie group satisfies  $[S_{\mathrm{nc}},R]=1$, $R$ is an NC-group, and its $\R$-rank is at most $1$.  Assume that the rank of $S$ is $1$, then the rank of $R$ is $0$, and we are in case (b) by Remark \ref{rem:rank0}.
 If the rank of $G$ is zero, then again Remark \ref{rem:rank0} implies that we are in case (a). Finally, if the rank of $S$ is zero and the rank of $R$ is one, we have case (c). Regarding unimodularity, the only non-obvious statement is that groups of type (c) are never unimodular. This is because the presence of an element $r\in R$ which contracts the exponential radical $E$ of $R$ implies  that the Haar measure of $E$ is not preserved by conjugation by $r$. Since $R/E$ is unimodular, we deduce that $R$ is not. And since $S$ is unimodular, $G$ is not.
\end{proof}

We close this subsection with further examples of algebraically thin groups as in \S\ref{sectionIntro:nonunimodular}; these corollaries are not needed elsewhere in the paper.

We denote by $\mathfrak{e}_\C\subset \mathfrak{r}_\C$ the Lie subalgebra spanned by characteristic subspaces of roots with non-zero real part, and $\mathfrak{e}=\mathfrak{e}_\C \cap \mathfrak{r}$.
We now assume that $\mathfrak{r}$ is the Lie algebra of some solvable connected Lie group $R$. 
By
  \cite[Proposition 5]{Guivarc'hExpRad}, the subgroup $E=\exp(\mathfrak{e})$ is a nilpotent connected subgroup and it coincides with the minimal closed normal subgroup of $R$ such that $R/E$ has polynomial growth. It was later rediscovered by Osin in \cite{Osin_exprad} who named it the \textbf{exponential radical} of $R$. The following lemma is standard: see for instance \cite[Proposition 4.B.5.]{CorTesIHES}.\footnote{Although the assumptions there are slightly more restrictive, the same proof applies.} 
  \begin{lemma}\label{lem:NC} Let $R$ be a solvable connected Lie group.
The following are equivalent: 
\begin{itemize}
	\item[(i)] $R$ is an NC-group.  
\item[(ii)] there is some element of $R$ acting as a contraction on its exponential radical $E$. 
\end{itemize}
\end{lemma}
We recall that a \textbf{contraction} of a locally compact group $G$ is an automorphism $\alpha$ such that $\alpha^n(g)\to 1$ for all $g\in G$, uniformly on compact subsets. Recall that a group is a \textbf{Heintze group} if it is isomorphic to a semidirect product $E\rtimes\R$, where $E$ is either trivial or a simply connected nilpotent Lie group, with every positive element of $\R$ acting as a contraction on $E$.

\begin{corollary}\label{cor:hypercentralbyHeintze}
Let $1\to N\to G\to H\to 1$ be an exact sequence of connected Lie groups, such that $N$ is hypercentral in $G$ (i.e.\ covered by the ascending central series of $G$) and $H$ is a Heintze group. 
Then the diagonal map $G\to G/E \times H$ induces an isomorphism of $G$ onto a closed subgroup of $G/E\times H$, where $E$ is  the exponential radical of $G$. 
\end{corollary}
\begin{proof} 
	The characterization (i) of NC-groups in Lemma \ref{lem:NC} makes it clear that being NC is stable by central, and therefore by hypercentral extensions. Since Heintze groups are NC, we deduce that $G$ is NC. 
Thus there exists an element of $G$ that contracts $E$, hence $E\cap N=\{1\}$. Therefore, the morphism $G\to G/E\times G/N$ is injective (and obviously has closed image). 
\end{proof}

\begin{corollary}\label{cor:NCabelianbyR}
Let $G$ be a NC-group which is isomorphic to a semi-direct product $\R^n\rtimes \R$. Then $G$ is a closed subgroup of a group of the form $P\times H$, where $P$ is a connected Lie group of polynomial growth, and $H$ is a Heintze group. 
\end{corollary}
\begin{proof}

Denote $U=\R^n$, and $D=\R$ so that $G=U\rtimes D$, and let $d$ be a non-zero element of $D$ that contracts the exponential radical $E$ of $G$ (provided by Lemma \ref{lem:NC}). 
	By decomposing $U$ into characteristic subspaces of $\ad(d)$, we see that $U$ decomposes as a $D$-equivariant direct sum $U=E\oplus N$, where $N$ is the sum of characteristic subspaces associated to eigenvalues of modulus $1$. Thus $G\cong (E\oplus N)\rtimes D$, which embeds as a closed subgroup of $(E\rtimes D)\times (N\rtimes D)$. Since $d$ contracts $E$, the first factor is a Heintze group, while $N\rtimes D$ has polynomial growth by Remark \ref{rem:rank0}.
\end{proof}

\begin{remark}
Corollaries \ref{cor:NCabelianbyR} and  \ref{cor:hypercentralbyHeintze} provide slightly different kinds of examples: for instance the semi-direct product of $(\C\times \R)\rtimes \R$, where $\R$ acts by rotation on the complex factor and by homothety on the $\R$ factor. This example satisfies the conditions of Corollary \ref{cor:NCabelianbyR} but not of Corollary \ref{cor:hypercentralbyHeintze}: indeed, in this example, $N=\C$, which is not hypercentral.
\end{remark}

\subsection{Reduction to linear, real  triangular by semisimple}
The goal of this section is to prove that any connected Lie group is commable to a linear connected Lie group whose radical is real-triangular (which is unimodular and is an NC-group if and only if $G$ is). 
Recall that two locally compact groups $G$ and $G'$ are \textbf{commable}  (see \cite{Cornulier_comma}) if there exists $n\geq 1$ and a sequence 
 \[G=G_0 -  G_1 - \cdots - G_{n-1} - G_n=G',\] 
 where the $G_i$ are locally compact groups and $G_{i-1}-G_i$ denotes the existence of a proper continuous group homomorphism with cocompact image $G_{i-1}\to G_i$ or $G_i\to G_{i-1}$. We shall call these maps \textbf{commability arrows} associated to the commability from $G$ to $G'$.

 Commability is a natural generalization of commensurability for discrete groups, and like commensurable finitely generated groups, commable locally compact groups are always quasi-isometric.

We say a Lie group $G$ has \textbf{Property V} (for Varopoulos) if $[S_{\mathrm{nc}},R]=1$ where $R$ is the solvable radical and $S_{\mathrm{nc}}$ is the non-compact part of a (any) Levi factor $S$; equivalently the corresponding Lie algebras satisfy $[\mathfrak{s}_{\mathrm{nc}},\mathfrak{r}] = 0$.
 We now state the goal of this subsection.

\begin{theorem}\label{thm:reduc}
Let $G$ be a connected Lie group. Then $G$  is commable to a linear connected Lie group $G'$ with same $\R$-rank of the form  (real-triangular)$\rtimes$(semisimple without compact factor). Moreover each of the following properties is true for $G'$ if and only it is true for $G$:
\begin{itemize}
\item unimodular;
\item Property V;
\item the solvable radical has\footnote{Recall that a solvable connected Lie group has Property R if its roots are purely imaginary, or equivalently if it has polynomial growth \cite{Gui-73-crois-poly}.} Property R;
\item the solvable radical has Property C (or NC);
\item  algebraically thin (or thick).
\end{itemize}
\end{theorem}

We will proceed in two steps: Proposition \ref{prop:reducLinear} treats the defect of linearity of the semisimple part, while Proposition \ref{prop:reducRealTrian} deals with the property that the amenable radical is real-triangular. 
(Given $G$ a connected Lie group, we denote by $\Am(G)$ the  \textbf{amenable radical}\footnote{At the level of Lie algebras, the amenable radical  differs from the solvable radical in that we add the compact semisimple factors to it.} : the maximal normal amenable subgroup of $G$, which turns out to be closed.)
This argument replaces our approach in an earlier version of the paper, following helpful suggestions of the referee.

Let us start recalling a few basic facts about connected Lie groups (see  \cite{OV}), and more specifically about the linearity of connected Lie groups (see \cite{Malcev,Hoch}).
We let $K(G)$ be the intersection of all kernels of continuous  linear finite dimensional representations of $G$. Considering the adjoint representation of a Lie group, we see that $K(G)$ is central. 
\begin{example}\label{ex:psl(2,R)tilde}
A typical example of non-linear simple Lie group is $\widetilde{\PSL(2,\R)}$: the universal cover of $\PSL(2,\R)$. 
\end{example}
When the group $G$ is semisimple, $K(G)$ has finite index in the center $Z(G)$ of $G$, which is discrete.  In particular,  linear connected  Lie groups have finite center.  

Let $T$ be a Levi factor of $\Am(G)$ in $G$: this a Lie subgroup of $G$ which is locally isomorphic to a sum of simple factors of positive rank and such that $G=\Am(G)T$. 
\begin{example}\label{ex:centershit}
Note that $T$ is not necessarily closed. A counterexample is for instance given by the group $G=(\widetilde{\PSL(2,\R)}\times \R/\Z)/Z$, where $Z$ is the cyclic subgroup generated by $(z,t)$, where $t\in \R/\Z$ is irrational and $z$ generates $Z(\widetilde{\PSL(2,\R))}$: here the amenable radical is the image of $\R/\Z$ in $G$, and a Levi factor $T$ is the image of $\widetilde{\PSL(2,\R)}$, their intersection being the dense subgroup of $\R/\Z$ spanned by $t$. We also observe that in this example, $K(G)=\R/\Z$ .   
\end{example}

We shall use the fact that a semisimple Lie group with finite center admits a real-triangulable cocompact subgroup. For instance in $\PSL(2,\R)$, this would be the subgroup of upper triangular matrices, while in $\PSL(2,\C)$, it would the subgroup of upper triangular matrices whose diagonal entries are real.

\begin{proposition}\label{prop:reducLinear}
Let $G$ be a connected Lie group. Then then there exists a connected Lie group
$G'=G''\times V$, commable to $G$, such that
\begin{itemize}
\item $V\cong \R^d$ for some $d\in \N$;
\item $G'$ and $G''$ are (amenable)$\rtimes$(linear semisimple without compact factor);
\item $G$ and $G''$ are locally isomorphic.
\end{itemize}
In particular,
\begin{itemize}
	\item $G$ is unimodular if and only if $G'$ is unimodular; 
\item $\Am_0(G') = \Am_0(G)\times V$ ; 
\item $G/\Am(G)$ and $G'/\Am(G')$ are isomorphic;
\item  $G$ has Property V if and only if $G'$ has Property V.
\end{itemize}

\end{proposition}
\begin{proof}
Write $A= \Am(G)$ and $A_0= \Am_0(G)$, and
let $T$ be a Levi factor of $A_0$ in $G$, i.e.\ a semisimple subgroup of $G$ satisfying $G=A_0\cdot T$, and $A_0\cap T\subset Z(G)$.
	Let $Z_0=K(T)$, so $Z_0$ has finite index in $Z(T)$ and is central in $G$. Note that the amenable radical of $G/A_0$ coincides with the center of $T/(A_0\cap T)$, so that $A\subset A_0Z(T)$.
	
	 In $Z_0$, let $Z_1$ be a maximal subgroup among those intersecting $A_0$ trivially. By maximality, $Z_1 \cdot (Z_0\cap A_0)$ has finite index in $Z_0$ and therefore in $Z(T)$. This implies that the group $A_1=Z_1\cdot A_0\cong Z_1\times A_0$ has finite index in $A_0Z(T)$, and therefore in $A$.	
Thus the semisimple group $G/A_1$ has finite center.

Note also that $Z_1$ is discrete: indeed, it maps injectively to the (discrete) center of the semisimple quotient $G/A_0$. 
Let $P_1=L_1/A_1$ be a closed cocompact real-triangulable subgroup of $G/A_1$. $P_1$ being simply connected, its preimage $L_1/A_0$ under the projection (which is a covering map) $G/A_0\to G/A_1$ is a direct product $ (A_1/A_0)\times P_0$, where $P_0=L_0/A_0$ is isomorphic to $P_1$. 
 In restriction to $T$, the surjection $G\to G/A_0$ yields a surjective morphism $T\to G/A_0$ with discrete kernel, which splits in restriction to the simply connected subgroup $P_0$. Thus lifting $P_0$ we see that 
$L_0=A_0\rtimes P_0$. 

Denote $S=T/Z_0$, which is by definition of $Z_0=K(T)$ the largest linear quotient of $T$. The $T$-action on $A_0$ induces an action of $S$ on $A_0$. Embed $A_1/A_0$ as a uniform lattice in some connected abelian Lie group $V'$, with non-compact factor $V$. We have the following cocompact inclusions:
\begin{equation}\label{eq:inclusions}
G\supset L_1=Z_1\times L_0\subset V'\times L_0\supset V\times L_0 =V\times (A_0\rtimes P_0)\subset V\times(A_0\rtimes S).\end{equation}
Thus $G$ is commable to $G':=V\times G''$, with  $G'':=A_0\rtimes S$. 

	Let  $\mathfrak{a}$ denote the common Lie algebra of $A_0$ and $A_1$, $\mathfrak{p}$ the Lie algebra of $P_0 \cong P_1$, and let $\mathfrak{g}, \mathfrak{s}, \ldots$ denote the Lie algebras of $G, S, \ldots$. We already have $\mathfrak{g}=\mathfrak{a}\rtimes \mathfrak{t}$. Since $T$ and $S$ are locally isomorphic, we deduce that $\mathfrak{g}=\mathfrak{a}\rtimes \mathfrak{s}$, and therefore $G$ and $G''$ are locally isomorphic. More precisely, on the level of Lie algebras, the sequence of inclusions  \eqref{eq:inclusions} first passes through the subalgebra $\mathfrak{a}\rtimes \mathfrak{p}$, then takes a direct product with $\mathfrak{v}'$ (which is subsequently reduced to $\mathfrak{v}$), and finally ends up with $\mathfrak{v}\times (\mathfrak{a}\rtimes \mathfrak{t})=\mathfrak{v}\times \mathfrak{g}$.   

The additional preservation assertions are clear, observing that they are unchanged under taking local isomorphisms and direct products with abelian groups.
 \end{proof}

\begin{examples}
It is instructive to illustrate the proof on the examples discussed above. First the case of $G=\widetilde{\PSL(2,\R)}$: here the group $P_0=L_0$ is the subgroup of upper triangular matrices in $\SSL(2,\R)$. We have $K(G)=\pi_1(\SSL(2,\R))\cong \Z$.
Hence $V=\R$, and we finally get $G'=\R\times \SSL(2,\R)$. 
In Example \ref{ex:centershit}, we leave the reader to check that $G'=G''=\R/\Z\times \SSL(2,\R)$.
\end{examples}

The second reduction step consists in passing from $A\rtimes S$, where $A$ is amenable and $S$ is (linear) semisimple without compact factors, to $N\rtimes S$, where $N$ is real-triangular.

\begin{proposition}\label{prop:reducRealTrian}
Let $G=A\rtimes S$ be a connected Lie group, such that $A=\Am_0(G)$, and $S$ has finite center. Then $G$ is commable to a group of the form $N\rtimes S$, where $N$ is real-triangulable. Moreover, each  commability  arrow  is of the form  
$\rho_i:A_{i}\rtimes S\to A_{i+1}\rtimes S$ (or $\rho_i:A_{i+1}\rtimes S \to A_i \rtimes S$) such that 
\begin{itemize}
\item $\rho_i$ is compatible with the semi-direct decomposition;
\item its restriction to $S$ is the identity;
\item its restriction $A_i\to A_{i+1}$ (or $A_{i+1}\to A_i$) has coabelian image. 
\end{itemize}
\end{proposition}
\begin{proof}
We will use the following easy observation. 
\begin{fact}
Let $M$ be a Zariski dense subgroup of an algebraic group $L$. If $[M,M]$ is Zariski closed in $L$, then $[M,M]=[L,L]$.
\end{fact}	
\begin{proof}[Proof of the fact]
The map $L\times L\to L$ defined by $(g,g')\to [g,g']$ being Zariski continuous, the Zariski closure of $[M,M]$ contains $[L,L]$, since $L$ is the Zariski closure of $M$. 
\end{proof}
	This follows from the proof of \cite[Lemma 2.4]{Cornulier_dimcone}: there Cornulier treats the case where $S=\{1\}$, but one checks that even in presence of a non-trivial $S$, both the commability arrows can be made $S$-equivariant.  The first arrow in his proof is (replacing his $G$ by our $A$) $A \to H=KA$ where $K\cong H/A$ is abelian.  The second arrow is (denoting his $T_1$ by $N$) the cocompact inclusion of $N \to H$. Now $[A,A]$ is unipotent, hence Zariski closed, so $A$ being Zariski dense in $H$ we have $[A,A]=[H,H]$; Cornulier shows $[A,A] \subset N$.  Thus both arrows have coabelian image. The last statement of Proposition \ref{prop:reducRealTrian} follows from the following observation: \cite[Lemma 2.4]{Cornulier_dimcone}  provides maps with coabelian image: indeed, in the argument, $[G,G]$ is unipotent, hence Zariski closed, so $G$ being Zariski dense in $H$, we have $[G,G]=[H,H]$.
\end{proof}

\begin{remark}\label{rem:preservation}
The fact that this reduction preserves unimodularity is clear as commability preserves unimodularity among amenable locally compact groups. Since the commability arrows have coabelian image, the set of non-zero weights of the radical is preserved, thus so are Property R, $\R$-rank, and Property C or NC. We will prove that Property V (recall this is the condition $[S_{\mathrm{nc}},R]=1$) is preserved in the following lemma (\ref{lem:PropVpreserved}). Given these results the fact that this reduction preserves algebraic thinness is simply a consequence of the others. \end{remark}
\begin{lemma}\label{lem:PropVpreserved}
Let $S$ be a semisimple Lie group, and let $R$ and $R'$ be connected solvable Lie groups with $S$-actions. Let $f:R\to R'$ be an $S$-equivariant proper continuous group homomorphism with cocompact coabelian image. Then $R\rtimes S$ has Property V if and only $R'\rtimes S$ does. 
\end{lemma}
\begin{proof}
On refining the sequence of commability arrows, we can assume $f$ is either injective or surjective. In each case, one implication is trivial. Let us treat the non-trivial directions. Assume $f$ is surjective with compact kernel $K$ and $[\mathfrak{s}_{\mathrm{nc}},\mathfrak{r}']=0$. Denote the Lie algebra of $K$ by $\mathfrak{k}$. As $\mathfrak{s}_{\mathrm{nc}}$ is semisimple, $\mathfrak{k}$ admits a vector complement $W$ in $\mathfrak{r}$ which is stable by the adjoint action of  $\mathfrak{s}_{\mathrm{nc}}$. Since $[\mathfrak{s}_{\mathrm{nc}},\mathfrak{r}']=0$, the adjoint action of  $\mathfrak{s}_{\mathrm{nc}}$ on $W$ is trivial.
Moreover, since $S_{\mathrm{nc}}$ acts trivially on $K$, we have $[\mathfrak{s}_{\mathrm{nc}},\mathfrak{k}]= 0$. Combining these two facts, we conclude that $[\mathfrak{s}_{\mathrm{nc}},\mathfrak{r}]=0$.

Suppose now that $f$ is injective, and that $[\mathfrak{s}_{\mathrm{nc}},\mathfrak{r}]=0$ while $[\mathfrak{s}_{\mathrm{nc}},\mathfrak{r}']\neq 0$. By semisimplicity, 
	and since $\mathfrak{r}$ is an ideal in $\mathfrak{r}'$ we have $[\mathfrak{s}_{\mathrm{nc}},\mathfrak{r}'/\mathfrak{r}]\neq 0.$ 
	But then, $S_{\mathrm{nc}}$ acts non-trivially on the compact group $R'/R$: contradiction.
 \end{proof}
 
 \begin{proof}[Proof of Theorem \ref{thm:reduc}]
 Starting with a connected Lie group $G$, Proposition \ref{prop:reducLinear} reduces to the case where $G$ satisfies the assumptions of Proposition \ref{prop:reducRealTrian}. The fact that the properties listed in Theorem \ref{thm:reduc} are preserved 
is obvious as the Lie algebras only differ by an abelian factor.
	 Next, applying Proposition \ref{prop:reducRealTrian} reduces to the case where $G$ is (real-triangular)$\rtimes$(linear semisimple without compact factor). The resulting group is linear by Malcev's theorem \cite{Malcev}: a connected Lie group with solvable radical  $R$ and Levi factor $S$ is linear if and only if both $R$ and $S$ are linear.  The preservation of the relevant properties is justified in Remark \ref{rem:preservation}.
 \end{proof}

\subsection{Algebraically thin unimodular connected Lie groups}

\begin{proposition}\label{prop:algthinRealTriang}
	Let $G=R\rtimes S$, where $R$ is real-triangulable, and  $S=S_{\mathrm{nc}}$ is semisimple with finite center. Then  $G$ is unimodular and algebraically thin if and only if $G=R\times S$, $R$ is simply connected nilpotent, and $S$ has rank $1$ or is trivial.
\end{proposition}
\begin{proof}
Suppose $G$ is unimodular and algebraically thin.  By Proposition~\ref{prop:algthinDescrip} we have two cases. In the first case $G$ has polynomial growth, hence $S$ is trivial and $G=R$, being triangulable and polynomial growth, is simply connected nilpotent.  In the second case, we have $[R,S]=1$, and $R$ has polynomial growth (thus again is simply connected nilpotent). The intersection $R\cap S$ is contained in the (finite) center of $S$. Hence it must be trivial as $R$ is torsion-free.
\end{proof}

\begin{corollary}\label{cor:algthinGeneral}
	Any algebraically thin, unimodular connected Lie group is commable to a direct product $G=R\times S$, where $R$ is simply connected nilpotent, and $S$ is simple of rank $1$ with finite centre, or trivial.
\end{corollary}
\begin{proof}
	We apply Theorem \ref{thm:reduc} and then Proposition \ref{prop:algthinRealTriang}.
\end{proof}

\subsection{Non-distortion of certain subgroups in linear  connected Lie groups}\label{sec:nondistotedsubg}
This section gives sufficient conditions for a subgroup of a Lie group to be undistorted. Although we could not find the following theorem in the literature, it is probably known to the experts.  The role it plays in the paper is to ensure that the subgroups $\SOL_a$ and $\Osc$ from Proposition \ref{prop:algThick} are closed and undistorted. 

Let $G$ and $H$ be locally compact compactly generated groups such that $H$ is a closed subgroup of $G$. We say that $H$ is \textbf{undistorted} in $G$ if the inclusion $H\to G$ is a quasi-isometric embedding (with respect to word metrics on both groups). In what follows, we shall use repeatedly the following obvious remark: if $\phi:G\to G'$ is a continuous homomorphism from $G$ to another compactly generated group $G'$ such that $\phi|_H$ is injective and $\phi(H)$ is undistorted in $G'$, then $H$ is undistorted in $G$.

\begin{theorem}\label{thm:Undistorted}
	We consider a group $H=U\rtimes A$, where $U$ is a simply connected nilpotent connected Lie group, and $A\cong \R^r$  and
\begin{itemize}
\item[(i)] for each non-trivial $a\in A$, the action by conjugation of $a$ on $V=U/[U,U]$ has an eigenvalue of modulus distinct from $1$, 
\item[(ii)] there exists some $a_0\in A$ such that all its (possibly complex) eigenvalues on $V$ have modulus distinct from $1$. 
\end{itemize}
	Then for any linear connected Lie group $G$, any injective morphism $f:H\to G$ has closed and undistorted image in $G$.
	\end{theorem}
\begin{proof}
On composing with a faithful linear representation, we can suppose that $G=\GL(d,\C)$. 
By Lie's theorem, we can assume that $H$ is contained in the subgroup of upper-triangular matrices. 
We deduce from (ii) that $V$ is contained in (and therefore equal to) the derived subgroup of $V\rtimes A$, since the map $V\to V, v \mapsto [v,a_0]$ is surjective.  Therefore we have $U = [H,H][U,U]$, which implies that $U/[H,H]$ is a perfect group, hence trivial as it is solvable. Hence 
	$U=[H,H]$.  It follows that  $U$ is contained in the subgroup of upper unipotent matrices. Let $\mathfrak g$ be the Lie algebra of $G$, that we equip with a norm $\|\cdot\|$. We denote by $|\cdot|_G$, $|\cdot|_H$ and $|\cdot|_A$ word lengths on respectively $G$, $H$ and $A$ associated to compact generating subsets. We also consider the operator norm $\|\cdot\|_{op}$ of $\GL(d,\C)$ acting on $\C^d$ equipped with the usual Euclidean metric. Note that since $\|\cdot\|_{op}$ is submultiplicative, $|g|_{op}:=\log \max\{\|g\|_{op},\|g^{-1}\|_{op}\}$ satisfies $|gg'|_{op}\leq |g|_{op}+|g'|_{op}$. Hence
\begin{equation}\label{eq:NormOp}
|g|_{op}\lesssim |g|_G.
\end{equation} 
	In particular, a straightforward calculation shows that  for all $x\in \mathfrak g\setminus \{0\}$, $\log\|x\|\lesssim |\exp(x)|_{op}$, from which we deduce that 
\begin{equation}\label{eq:logExp}
\log(1+\|x\|)\lesssim |\exp(x)|_G.
\end{equation}
	On the other hand condition (ii) implies that $U$ is the exponential radical of $H$ (see \cite[Proposition 5]{Guivarc'hExpRad}). 
By the corollary following \cite[Proposition 5]{Guivarc'hExpRad}, we have for all $u\in U$,
\[|u|_H\lesssim \log(1+ \|\log u\|).\]
	Combining this with (\ref{eq:logExp}) and with the obvious inequality $|h|_G\lesssim |h|_H$, valid for all $h\in H$, we deduce that for all $u\in U$ 
\begin{equation}\label{eq:Unondist}
|u|_G\simeq |u|_H\simeq \log(1+ \|\log u\|).
\end{equation}
	A consequence of (i) is that for all $h=ua\in H$, $|h|_{op}\gtrsim |a|_A$. Indeed, note that $\|\cdot\|_{op}$ defines a norm on the vector space of square matrices $\mathbb M(d,\C)$, and as such is bi-Lipschitz equivalent to any other norm. Now, if one considers the norm $\|\cdot\|_1$ consisting of the sum of absolute values of coefficients, one clearly has $\|h\|_{1}=\|a\|_1+\|u\|_1\geq \|a\|_1$. 
On the other hand, we observe that $\log \max\{\|a\|_1, \|a^{-1}\|_1\}\simeq |a|_A$, so  $|h|_{op}\gtrsim |a|_A$ as claimed. 	
	Since there is an obvious projection of $H$ onto $A$, $|a|_A\simeq |a|_H$, so $|h|_{op}\gtrsim |a|_H$. We then deduce from \eqref{eq:NormOp} that
\begin{equation}\label{eq:Znondist}
|h|_G \gtrsim |a|_H.
\end{equation} 
Assume for a contradiction that there exists $h_k=(u_k,a_k)\in H$ such that $|h_k|_G=o(|h_k|_H)$. 
	Then by \eqref{eq:Znondist}, $|a_k|_H=o(|h_k|_H)$, which implies by the triangle inequality that $|h_k|_H\simeq |u_k|_H$, and $|u_k|_G \lesssim |h_k|_G +|a_k|_G \lesssim |h_k|_G+|a_k|_H = o(|h_k|_H)=o(|u_k|_H)$. But the latter contradicts \eqref{eq:Unondist}, so we are done.
\end{proof}

\begin{examples}\label{ex:SOL}
	The class of groups $H$ satisfying the conditions of Theorem \ref{thm:Undistorted} is stable under finite direct product, and contains the examples that are relevant to us: $\SOL_a$ for all $a>0$, and $\Osc$. But it also contains the subgroup of upper triangular matrices  whose diagonal entries are real and positive in $\SSL(d,K)$, for $d\geq 2$ and $K\in \{\R,\C\}$. In this last case $U$ is the upper triangular unipotent matrices and $A$ consists of diagonal matrices with real and positive entries. 
\end{examples}

\subsection{Algebraically thick connected Lie groups}\label{sec:AlgThick}

The main goal of this section is to prove Proposition \ref{prop:algThick}. We will proceed in various steps.
\begin{proposition}\label{prop:SufficientAlthick}
Let $G$ be a linear connected Lie group. If either:
	\begin{itemize}\item $[S_{\mathrm{nc}},R]\neq 1$, or
			\item the $\R$-ranks of both $S_{\mathrm{nc}}$ and $R$ are positive, or 
			\item the $\R$-rank of $S_{\mathrm{nc}}$ is at least 2,
	\end{itemize}
	then $G$ has an undistorted closed subgroup isomorphic to $\SOL_1$ or $\Osc$.
\end{proposition}
\begin{proof} Let us first prove the Lie algebra analogue. In the first case, by \cite[Proposition 8.2]{CDSW}, $\mathfrak{g}$ has a subalgebra isomorphic to $\mathfrak{v}_n\rtimes\mathfrak{sl}(2,\R)$ for some irrreducible $n$-dimensional representation $\mathfrak{v}_n$ for  $n\geq 2$, or to the $1$-dimensional central extension
$\mathfrak{h}_{2n+1}\rtimes\mathfrak{sl}(2,\R)$ of  $\mathfrak{v}_{2n}\rtimes\mathfrak{sl}(2,\R)$ for some $n\geq 1$. The first one contains a copy of $\mathfrak{sol}_1$, while the second one contains a copy of $\mathfrak{osc}$.

Let us now assume that $[\mathfrak{s}_{\mathrm{nc}},\mathfrak{r}]= 0$ and both $\mathfrak{s}_{\mathrm{nc}}$ and $\mathfrak{r}$ have positive rank, then they both contain a subalgebra isomorphic to the affine Lie algebra $\R\rtimes \R$. Hence $\mathfrak{g}$ contains a subalgebra isomorphic to $\mathfrak{sol}_1$.

We are left with the case where $\mathfrak{s}_{\mathrm{nc}}$ has $\R$-rank at least 2. We can assume without loss of generality that $\mathfrak{g}$ is equal to its semisimple part. If $\mathfrak{g}$ is not simple, then it contains $\mathfrak{sl}(2,\R)\times \mathfrak{sl}(2,\R)$. Since the latter contains a copy of $\mathfrak{sol}_1$, we are done.
	Otherwise, by \cite[Lemma 1.6.2]{BekdlHarpeVal}, it contains a copy of either $\mathfrak{sl}(3,\R)$ or $\mathfrak{sp}(4,\R)$. Since $\mathfrak{sp}(4,\R)$ contains  a copy of $\mathfrak{sl}(2,\R)\times \mathfrak{sl}(2,\R)$, it has already been treated. Finally we conclude from the fact that $\mathfrak{sl}(3,\R)$ contains a copy of $\mathfrak{sol}_1$.

	Let us now prove the proposition. By simple connectedness and the Lie algebra case, we deduce that there is an injective continuous homomorphism $H\to G$, with $H$ either $\SOL_1$ or $\Osc$. By Proposition \ref{thm:Undistorted} and Examples \ref{ex:SOL}, we deduce that the image is closed and undistorted.
\end{proof}

We now treat the case of real-triangulable Lie groups. We start with the following Lie algebra statement. A Lie algebra is called minimal algebraically thick if it is algebraically thick, i.e.\ has $\R$-rank at least 2 or is C (recall Definition~\ref{def:NC}) but no proper subalgebra satisfies these conditions.

\begin{lemma}\label{lem:SufficientCLieAlg}
A real-triangulable Lie algebra is minimal  algebraically thick if and only if it is isomorphic to $\mathfrak{sol}_a$, for some $a>0$, or $\mathfrak{osc}$.
\end{lemma}
\begin{proof}
	Clearly these are minimal   algebraically thick. Conversely, assume that $\mathfrak{g}$ is minimal   algebraically thick, and let $\mathfrak{e}$ be its exponential radical.
 Since the Lie algebra is real-triangulable, its roots are real, and this correspond to weights.  Recall that we can see them as elements of the dual of the abelianization $\mathfrak{g}_{\mathrm{ab}}$ of $\mathfrak{g}$. 
 
  Assume first that $\mathfrak{e}$ has codimension at least 2. Since $\mathfrak{e}\subset [\mathfrak{g},\mathfrak{g}]$, we deduce that $\mathfrak{g}_{\mathrm{ab}}$ has dimension at least 2. Fix a Euclidean structure on $\mathfrak{g}_{\mathrm{ab}}$, so that we can identify it with its dual.
 Consider subalgebras containing $[\mathfrak{g},\mathfrak{g}]$ of dimension $1+\dim[\mathfrak{g},\mathfrak{g}]$, which  are in one-to-one correspondence with lines of $\mathfrak{g}_{\mathrm{ab}}$. The effect of passing to such a subalgebra on the weights is to project them orthogonally to the corresponding line. 
	Since there are only finitely many weights, one can find a line so that any non-zero weight remains non-zero when restricted to the line. 
	Moreover, whether  $\mathfrak{g}$ has C, or has $\R$-rank is at least $2$,  we can find a line such that when projecting onto this line we obtain a subalgebra which has C, which contradicts minimality. 	
	 
	Hence $\mathfrak{e}$ has codimension 1, so $\mathfrak{g}=\mathfrak{e}\rtimes \mathfrak{a}$, where $\mathfrak{a}$ is one-dimensional. Recall that weights are elements of $\Hom(\mathfrak{g}_{\mathrm{ab}},\R)=\Hom(\mathfrak{a},\R)\cong \R$. Since $\mathfrak{g}$ is real-triangulable, weights correspond to the eigenvalues of the adjoint action of $\mathfrak{a}$ on $\mathfrak{e}$. Denote $\mathfrak{e}_x$ the characteristic subspace of $\mathfrak{e}$ associated to the weight $x$. Observe that since the adjoint action is by derivations, $[\mathfrak{e}_x,\mathfrak{e}_y]\subset \mathfrak{e}_{x+y}$. In other words, the vector decomposition $\mathfrak{e}=\bigoplus_x \mathfrak{e}_x$, where $x$ run through weights, defines a  real grading of the Lie algebra $\mathfrak{e}$.
	
Write $(\mathfrak{e}^i)_{i\geq 1}$ for the lower central series of $\mathfrak{e}$. These are graded ideals, so in particular the graduation on $\mathfrak{e}$ induces a graduation on $\mathfrak{e}^i/\mathfrak{e}^j$ for all $i<j$ (associated to the corresponding induced $\mathfrak{a}$-action). Moreover the Lie bracket induces for each pair of weights $\alpha,\beta$, and for all $i,j\in \N^*$ a bilinear map \begin{equation}\label{eq:grading}(\mathfrak{e}^i/\mathfrak{e}^{i+1})_\alpha\times (\mathfrak{e}^j/\mathfrak{e}^{j+1})_\beta\to (\mathfrak{e}^{i+j}/\mathfrak{e}^{i+j+1})_{\alpha+\beta}.\end{equation}

	The fact that $\mathfrak{g}$ has C implies that there are both positive and negative weights---say $s$ and $t$. 
Passing to the graded subalgebra generated by one eigenline in degree $s$ and $t$, and using minimality, we see that $\mathfrak{e}^1/\mathfrak{e}^2$ is 2-dimensional with weights $s$ and $t$. 
We claim that $\mathfrak{e}^2$ is concentrated in degree $0$, and therefore that $\mathfrak{e}=\mathfrak{e}_t\oplus\mathfrak{e}_s\oplus\mathfrak{e}^2_0$. 
For if $\mathfrak{e}^2$ contains a negative (resp.\ positive) weight then $(\mathfrak{e}_s + \mathfrak{e}^2) \rtimes a$ (resp.\  $(\mathfrak{e}_t + \mathfrak{e}^2) \rtimes a$) contradicts minimality.

Now if $t+s\neq 0$, then by (\ref{eq:grading}), we have $\mathfrak{e}^2=0$, and so $\mathfrak{g}\cong \mathfrak{sol}_{-t/s}$.
Assume $s+t=0$. Again by by (\ref{eq:grading}) we have $[\mathfrak{e}_s,\mathfrak{e}^2]=[\mathfrak{e}_t,\mathfrak{e}^2]=0$, so $\mathfrak{e}^2$ is centralized by $\mathfrak{e}_s$ and $\mathfrak{e}_t$. By minimality $\mathfrak{e}_s$ and $\mathfrak{e}_t$ are one-dimensional, and $\mathfrak{e}^2$ is either zero or one-dimensional, according to whether $[\mathfrak{e}_s,\mathfrak{e}_t]$ is zero or not. In the first case, $\mathfrak{g}\cong \mathfrak{sol}$, while in the second case, $\mathfrak{e}$ is the Heisenberg group, and  $\mathfrak{g}\cong\mathfrak{osc}$. 
\end{proof}

\begin{proposition}\label{prop:SufficientC}
Let $G$ be an algebraically thick real-triangular Lie group. Then it has an undistorted closed subgroup isomorphic to $\SOL_a$, for some $a>0$, or $\Osc$.
\end{proposition}
\begin{proof}
Lemma \ref{lem:SufficientCLieAlg} clearly implies the analogous result
for Lie algebras: since an algebraically thick real triangulable Lie algebra contains a minimal one, we deduce that $\mathfrak{g}$ contains a subalgebra isomorphic to $\mathfrak{sol}_a$, for some $a>0$, or $\mathfrak{osc}$. We conclude as in the end of the proof of Proposition \ref{prop:SufficientAlthick}.
\end{proof}

 Combining  Proposition \ref{prop:SufficientAlthick} and \ref{prop:SufficientC}, we immediately deduce the following result.

\begin{corollary}\label{cor:OneImplicationThick}
Let $G$ be a linear connected Lie group with real-triangulable radical. If it is algebraically thick, then it contains a closed undistorted subgroup isomorphic to $\SOL_a$, for some $a>0$, or $\Osc$.
\end{corollary}
\begin{proof}
The only case not covered by Proposition \ref{prop:SufficientAlthick}
is when $R$ is a $C$-group or has $\R$-rank at least $2$, which is treated in Proposition  \ref{prop:SufficientC}. 
\end{proof}

 We now turn to the converse.
 
 \begin{lemma}\label{lem:NCpassesTo subgroup}
The class of solvable algebraically thin groups is stable under taking closed subgroups.
 \end{lemma}
 \begin{proof}
	 Recall that algebraically thin solvable Lie groups are precisely NC-groups of $\R$-rank at most $1$ by Proposition~\ref{prop:algthinDescrip}. 
	 The conclusion can be deduced from Lemma \ref{lem:NC}: indeed, if $G'<G$, then the exponential radical $E'$ of $G'$ is contained in $E$, and either $G'$ contains an element that contracts $E$, and therefore $E'$, or it does not. But since the $\R$-rank of $G$ is at most one, the last option implies that $E'=\{1\}$, and therefore that $G'$ has polynomial growth.     
 \end{proof}

 \begin{lemma}\label{lem:algebraicallythin=NC}
 If a connected Lie group is algebraically thin, then so are all its closed connected solvable subgroups.
 \end{lemma}
 \begin{proof}
This is a statement about Lie algebras. 
Let $\mathfrak{s}$ be the semisimple part of the Lie algebra $\mathfrak{g}$ of $G$, and let $\mathfrak{r}$ be its solvable radical.  Let $\mathfrak{n}$ be a maximal solvable subalgebra of $\mathfrak{s}$: it is of type NC. Let $\mathfrak{m}=\mathfrak{n}+ \mathfrak{r}$.
The condition $[\mathfrak{r},\mathfrak{s}]=0$ ensures that $\mathfrak{m}$ is the direct product of $\mathfrak{n}$ with $\mathfrak{r}$. Besides, by Hahn-Banach, property NC says that there exists $x$ in the Lie algebra such that $\omega(x)>0$ for all for weights $\omega$. Since weights of $\mathfrak{\mathfrak{n}\times \mathfrak{r}}$ are pairs 
$(\omega_n,\omega_r)$, where $\omega_n$ and $\omega_r$ are respectively weights of $\mathfrak{n}$ and $\mathfrak{r}$, we deduce that 
$\mathfrak{m}$ has property NC and has $\R$-rank at most 1. Now every maximal solvable subalgebra of $\mathfrak{g}$ is of this form. Hence combining this with Lemma \ref{lem:NCpassesTo subgroup} proves the lemma. 
 \end{proof}

\begin{proof}[Proof of Proposition \ref{prop:algThick}]
Corollary \ref{cor:OneImplicationThick} readily implies (i)$\implies$(ii), and (ii)$\implies$(iii) is clear. 
Hence we are reduced to proving (iii)$\implies$(i). This results from Lemma \ref{lem:algebraicallythin=NC} together with the fact that neither $\SOL_a$ nor $\Osc$ are algebraically thin.
\end{proof}

\section{Embeddings of Diestel--Leader graphs}\label{sec:DL}
 The goal of this section is to show that for all $a>0$, $\SOL_a$ (and some others groups) contain a quasi-isometrically embedded copy of the Diestel--Leader graph. This will be important when we establish lower bounds on Poincar\'e profiles in \S \ref{sec:Poincare n/logn}. Unfortunately, we are unable to prove it for $\Osc$, so its Poincaré profiles will be established by a direct (much harder) computation in \S \ref{sec:Poincare n/logn}.

\begin{theorem}\label{thm:embSgraph}
	The Diestel--Leader graph $\DL(2,2)$ quasi-isometrically embeds into \begin{itemize}
\item $T\times T$ where $T$ is any tree with minimal vertex degree $\geq 3$,
\item any finitely generated wreath product $H\wr K$ where $H$ is non-trivial and $K$ is infinite,
\item the Baumslag--Solitar group $\BS(m,n)$, whenever $\abs{m}\neq \abs{n}$, and
\item $\SOL_a$ for any $a>0$.
\end{itemize}
\end{theorem}
The first two items are certainly not new, and, as mentioned in the introduction, all are likely known to experts, but proofs are given for completeness.

Firstly, we recall the definition of the Diestel--Leader graph. Given a simplicial tree $T$, $v\in VT$ and $\xi\in\partial T$, for each vertex $w$ let $\gamma_w$ be the unique geodesic ray from $w$ to $\xi$. The \textbf{Busemann function} associated to the triple $(T,v,\xi)$ is defined by $b_{T,v,\xi}:VT\to\R$,
\[
 b_{T,v,\xi}(w)= d_T(v,\gamma_v\cap\gamma_w) - d_T(w,\gamma_v\cap\gamma_w).
\]

Let $\mathcal T_i=(T_i,v_i,\xi_i)$ for $i=1,2$, where each $T_i$ is a simplicial tree, $v_i\in VT_i$ and $\xi_i\in\partial T_i$. Let $h_i=b_{T_i,v_i,\xi_i}$ for $i=1,2$.
The vertex set of the \textbf{Diestel--Leader graph} $\DL(\mathcal T_1,\mathcal T_2)$ is
\[
 \setcon{(x,y)\in VT_1\times VT_2}{h_1(x)+h_2(y)=0}
\]
and two vertices $(x,y)$,$(x',y')$ span an edge if and only if $xx'$ and $yy'$ are edges in $ET_1$ and $ET_2$ respectively. As a shorthand we write $\DL(q_1,q_2)$ when each $T_i$ is a $(q_i+1)$-regular tree. The Diestel--Leader graph $\DL(q,q)$ is a Cayley graph of the lamplighter group $\Z_q\wr\Z$ \cite{Woess-LAMPLIGHTERS_HARMONIC_FUNCTIONS}.

We start with the standard fact that Diestel--Leader graphs are undistorted in the product of trees used to define them.

\begin{lemma}\label{lem:DLbiLip} The inclusion map $\iota:\DL(\mathcal T_1,\mathcal T_2)\to T_1\times T_2$ (defined on vertices) is a bi-Lipschitz embedding with respect to the shortest path metric on $\DL(\mathcal T_1,\mathcal T_2)$ and the $L^1$ product metric on $T_1\times T_2$.
\end{lemma}
\begin{proof}[Sketch of proof] It is clear that $\iota$ is $2$-Lipschitz. 
	For the converse, suppose we have $(s_1,s_2)$ and $(t_1,t_2) \in \DL(\cT_1,\cT_2) \subset T_1 \times T_2$.
	We find a path in $\DL(\cT_1,\cT_2)$ connecting these points by concatenating a path of length $d_{T_1}(s_1,t_1)$ from $(s_1,s_2)$ to some $(t_1,s_2')$, and a path of length at most $d_{T_1}(s_1,t_1)+d_{T_2}(s_2,t_2)$ from $(t_1,s_2')$ to $(t_1,t_2)$.  Hence, $d_{\DL(\cT_1,\cT_2)} \leq 2 d_{T_1\times T_2}$.
\end{proof}

\subsection{Busemann compatible embeddings}
We let $\HH_\R^2=\{(x,y), \; x\in \R, y>0\}$ be the hyperbolic half-plane with boundary $\partial\HH_\R^2 = (\R\times\set{0})\cup\set{\infty}$, and define a Busemann function on $\HH_\R^2$ by $b_{\HH_\R^2,(0,1),\infty}(x,y)=\log(y)$.

Recall that if $T$ is a tree, and $\iota:T\rightarrow \HH_\R^2$ is a bi-Lipschitz  embedding, then $\iota$ extends to a topological embedding (a homeomorphism onto its image) between the Gromov compactifications $\overline{\iota}:\overline{T}\rightarrow \overline{\HH_\R^2}$. 

\begin{definition}\label{defn:compatibletree}
	Let $\alpha>0$, and let $T$ be a tree. Let $(X,v',\infty)$ be either the triple $(\HH_\R^2,(0,1),\infty)$ or a triple of a tree $T'$, a vertex and a point labelled $\infty$ in the boundary of $T'$. Let $h_X$ be the Busemann function associated to the triple $(X,v',\infty)$. A map $\iota:T\rightarrow X$ is called \textbf{$\alpha$-Busemann-compatible}, if there is a vertex $v\in T$ and a point $\xi\in\partial T$ such that
\begin{itemize}
  \item[(i)] $\iota$ is a bi-Lipschitz embedding;
\item[(ii)]  $\overline{\iota}(\xi)=\infty$;
\item[(iii)] $h_X(\iota(z))=\alpha b_{T,v,\xi}(z)$, for all $z\in VT$.
\end{itemize}
\end{definition}

\begin{remark}\label{rem:TtoT'}
It is an easy observation that given a tree $(T',v',\xi')$ where every vertex has degree at least $3$, and any positive integer $k$, there exists a $k$-Busemann-compatible embedding of the $3$-regular tree $(\mathcal{T}_3,v,\xi)$ into $(T',v',\xi')$. The following proposition shows that a similar fact is true replacing $T'$ by $\HH_\R^2$.
\end{remark}

\begin{proposition}\label{prop:existenceofbctrees}
	For all $\alpha>\log(m)$, there exists an $\alpha$-Busemann-compatible embedding of the $(m+1)$-regular tree in $\HH_\R^2$.
\end{proposition}
\begin{proof}
Fix $m\in\N$ and $\alpha>\log(m)$. Set $t:=e^\alpha>m$. Let $A_{m,t}$ be the subset of $\R$ consisting of finite combinations of the form $\sum_{k\geq 0} a_it^k$, with $a_i\in \{0,\ldots,m-1\}$. 
Now for every $n\in \Z$, define $\Sigma_n$ to be 
\[\Sigma_n=\setcon{(at^n,t^n)}{a\in A_{m,t}}.\]
We now define a graph $T$ whose set of vertices is $\Sigma=\cup_{n\in \Z}\Sigma_n$, and whose edges relate pairs of vertices $(v,v')\in \Sigma_n\times \Sigma_{n+1}$, with $v=(at^n,t^n)$ and $v'=((a-a_0)t^{n},t^{n+1})$. This ensures that the distance between $v$ and $v'$ is bounded by a constant $K$ only depending on $t$ and $m$.  It follows by construction that $T$ is an $(m+1)$-regular tree, and that the restriction of $\alpha^{-1}b_{\HH_\R^2,(0,1),\infty}$ to $T$ coincides with the Busemann function based at $(0,1)\in T$, and pointing towards $\infty$. As already observed the inclusion $\iota:T\to\HH_\R^2$ is $K$-Lipschitz; we will now prove that the choice of $\alpha$ ensures that it is bi-Lipschitz. We start with the case of two points on $T$ at the same Busemann level, whose images in $\HH_\R^2$ are therefore of the form $v=(at^n,t^n)$ and $v'=(a't^n,t^n)$. On applying the hyperbolic isometry of the half-plane that fixes $0$ and $\infty$ and maps $H_n=b_{\HH_\R^2,(0,1),\infty}^{-1}(\alpha n)$ to $H_0$, we can assume that $n=0$. Then their distance in the tree is $2k+2$ where $k$ is the largest integer such that $a_k\neq a'_k$. On the other hand, one has
\[|a-a'|\geq t^k- \sum_{i=0}^{k-1}(m-1)t^i\geq t^k\left(1-\frac{m-1}{t-1}\right).\]
Since $t>m$, there exists $c>0$ only depending on $\alpha$ and $m$ such that $d_{\HH_\R^2}(v,v')\geq c k$, so we are done. The general case can easily be deduced from this one using that the distance between any points in $\Sigma_n$ and $\Sigma_{n+k}$ is at least $\alpha k$.
\end{proof}

\subsection{Horocyclic products}
Let $X,Y$ be spaces with associated Busemann functions $b_X,b_Y$ and let $\alpha>0$. The \textbf{($\alpha$-stretched) horocyclic product of $X$ and $Y$} is defined as
\[
 S_\alpha(X,Y) =\setcon{(z_1,z_2)\in X\times Y}{b_X(z_1)+\alpha b_Y(z_2)=0}
\]
and is equipped with the subspace metric from the $L^1$ product metric on $X\times Y$.

\begin{proposition}\label{prop:BusemanProducts}
Let $\alpha_1,\alpha_2>0$, and let each of $\mathcal X_1=(X_1,v_1,\infty_1)$ and $\mathcal X_2=(X_2,v_2,\infty_2)$ be either $(T,v,\xi)$ with $T$ a tree, $v\in VT$ and $\xi\in\partial T$ or $(\HH_\R^2,(0,1),\infty)$.

Suppose $\mathcal T_1=(T_1,v_1,\xi_1)$ and $\mathcal T_2=(T_2,v_2,\xi_2)$ are two trees of degree at least $3$, with distinguished vertices and points in their boundary, and admitting $\alpha_1$ and $\alpha_2$--Busemann-compatible embeddings into $X_1$ and $X_2$ respectively. Then the Diestel--Leader graph $\DL(\mathcal T_1,\mathcal T_2)$ admits a bi-Lipschitz embedding into the $\alpha$-stretched horocylic product $S_\alpha(X_1,X_2)$ of $X_1$ and $X_2$, where $\alpha=\alpha_1/\alpha_2$.
\end{proposition}
\begin{proof}
For $i=1,2$, let $\phi_i:T_i\to X_i$ be an $\alpha_i$--Busemann-compatible embedding, i.e., for all $z_i\in T_i$,
\[
 h_{X_i}(\phi_i(z_i))=\alpha_ib_{\mathcal T_i}(z_i).
\]

We immediately obtain a bi-Lipschitz embedding $\psi=(\phi_1,\phi_2)$ from $T_1\times T_2$ to $X_1 \times X_2$, which by Definition~\ref{defn:compatibletree}(iii) restricts to a (Lipschitz) embedding $\overline{\psi}:\DL(\cT_1,\cT_2)\to S_\alpha(X_1,X_2)$. To see this notice that for $(z_1,z_2)\in \DL(\mathcal T_1,\mathcal T_2)$ we have
\[
 0 = b_{\mathcal T_1}(z_1)+b_{\mathcal T_2}(z_2) = \alpha_1^{-1}h_{X_1}(\phi_1(z_1)) + \alpha_2^{-1}h_{X_2}(\phi_2(z_2)),
\]
so $h_{X_1}(\phi_1(z_1)) + \alpha h_{X_2}(\phi_2(z_2))=0$.

By Lemma \ref{lem:DLbiLip} the embedding of $\DL(\mathcal T_1,\mathcal T_2)$ into $T_1\times T_2$ is bi-Lipschitz, so 
\[
 \psi\circ\iota: \DL(\mathcal T_1,\mathcal T_2)\to X_1\times X_2
\]
is a composition of bi-Lipschitz embeddings and hence, is a bi-Lipschitz embedding. Since $\psi\circ\iota$ equals the composition of $\overline{\psi}$ with the natural embedding $j:S_\alpha(X_1,X_2)\to X_1\times X_2$ and $\overline{\psi}$ and $j$ are both Lipschitz, it follows that $\overline{\psi}$ must be a bi-Lipschitz embedding.
\end{proof}

With these results we can now complete the proof of Theorem \ref{thm:embSgraph}.

\begin{proof}[Proof of Theorem \ref{thm:embSgraph}]
	As is standard, by Lemma \ref{lem:DLbiLip}, $\DL(2,2)$ quasi-isometrically embeds into the product of two trivalent trees, and the trivalent tree isometrically embeds into any tree with minimum degree $\geq 3$. 
 	Next, $\DL(2,2)$ is a Cayley graph of the lamplighter $\Z_2\wr\Z$. Let $h\in H\setminus\{1_H\}$ and let $\gamma$ be a bi-infinite geodesic in a Cayley graph of $K$. The map
\[
 (f,z)\in\Z_2\wr\Z \mapsto (g,\gamma(z))\in H\wr K
\]
	where $g(k)=h$ if $k=\gamma(m)$ for some $m$ and $f(m)=1$, and otherwise $g(k)=1_H$, defines a quasi-isometric embedding of $\Z_2\wr\Z$ into $H\wr K$. Thus $H\wr K$ contains a quasi-isometrically embedded copy of $\DL(2,2)$.

	For $n\geq 2$, $\BS(1,n)$ is quasi-isometric to a horocyclic product of an $(n+1)$-valence tree with a copy of $\HH_\R^2$, so contains a quasi-isometrically embedded copy of $\DL(2,2)$ by Proposition \ref{prop:BusemanProducts}. 
	The groups $\BS(m,n)$ with $\abs{m},\abs{n}\geq 2$ and $\abs{m}\neq\abs{n}$ are all quasi-isometric \cite{Whyte-HigherBS}, and $\BS(2,4)=\langle a,t \mid t^{-1}a^2t=a^4\rangle$ contains $\BS(1,2)=\langle a^2,t\rangle$ as an (undistorted) subgroup.

 	Finally, for $a>0$, we consider 
	\begin{align*}
		\SOL_a = \R^2\rtimes_{(1,-a)}\R 
		& \cong \left\{ ((x,t),(y,t)) \in \left(\R\rtimes_{1}\R\right)\times\left(\R\rtimes_{-a}\R\right) \right\}
		\\ & \cong \left\{ ((x,-t),(y,t)) \in \left(\R\rtimes_{-1}\R\right)\times\left(\R\rtimes_{-a}\R\right) \right\}.
	\end{align*}
	Here, for $b\in \R$, $\R\rtimes_b \R$ indicates the semidirect product where the action is given by $x\cdot \psi(t)=e^{bt}x$;
	when $b\neq 0$ this group admits a left-invariant metric isometric to $\HH_\R^2$.
	Let us fix isometries $\iota_1:\R\rtimes_{-1}\R\to\HH_\R^2, \iota_2:\R\rtimes_{-a}\R\to\HH_\R^2$ so that for each $i=1,2$ we have $b_{\HH_\R^2,(0,1),\infty}(\iota_i(x,t))=\beta_i t$  where $\beta_1=1, \beta_2 = 1/a$. Equipped with a suitable left-invariant metric, $\SOL_a$ is isometric to the $\frac{\beta_1}{\beta_2}$-stretched horocyclic product of $(\HH_\R^2,(0,1),\infty)$ with itself.

	By Remark \ref{rem:TtoT'} and Proposition \ref{prop:existenceofbctrees}, for $\beta>0$ large enough the $3$-regular tree admits a $\beta\beta_i$--Busemann-compatible embedding in $\HH_\R^2$ for each $i=1,2$, hence by Proposition \ref{prop:BusemanProducts}, $\DL(2,2)$ quasi-isometrically embeds into $\SOL_a \cong S_\alpha(\HH_\R^2,\HH_\R^2)$ for $\alpha=\frac{\beta_1}{\beta_2}=\frac{\beta\beta_1}{\beta\beta_2}$.
\end{proof}

\begin{remark}
Given two locally compact groups $G_1,G_2$ with automorphisms $\alpha_1,\alpha_2$ contracting into compact sets, 
	the methods above adapt straightforwardly to show that $\DL(2,2)$ quasi-isometrically embeds into $(G_1\times G_2)\rtimes_{(\alpha_1,\alpha_2^{-1})}\Z$ (cf.~\cite[Definition 1.3]{CorTesIHES}). 
\end{remark}

\section{Poincar\'e profile calculations for analytically thick groups}\label{sec:Poincare n/logn}

Having established the required background on Lie groups and Diestel--Leader graphs, we now begin the main content of this paper.  
In this section we prove: $\DL(2,2)$ and $\Osc$ have $\Lambda^p(r)\simeq r/\log(r)$ for all $p\in[1,\infty]$. 
In both theorems the $p=\infty$ case follows immediately from \cite[Proposition 6.1]{HumeMackTess}, and so by the following proposition it suffices to prove a lower bound of $r/\log(r)$ on $\Lambda^1$.

\begin{proposition}
	\label{prop:think-profiles-general-upper}
	If $X$ is a graph with bounded degree with finite Assouad--Nagata dimension, then $\Lambda_X^p(r) \lesssim r/\log(r)$ for all $p \in [1,\infty)$. 
\end{proposition}

\begin{proof}
Note that having finite Assouad--Nagata dimension, $X$ has finite measurable dimension in the sense of \cite[Definition 9.1]{HumeMackTess} with function $\gamma(t)\lesssim e^t$. We conclude thanks to \cite[Proposition 9.5]{HumeMackTess}.
\end{proof}

\begin{corollary}\label{cor:general-upper-bound}
	Let $G$ be a connected Lie group or a Baumslag--Solitar group $\BS(m,n)$.  Then $\Lambda_G^p(r) \lesssim r/\log(r)$ for all $p\in [1,\infty)$, and moreover if $G$ is analytically thick, 
	\[\Lambda_X^p(r) \simeq r/\log(r) \quad \forall p\in[1,\infty].\]
\end{corollary}
\begin{proof}
Such a group $G$ has finite Assouad--Nagata dimension (see \cite{HP-13-ANdimension-nilppolyc} for the case of connected Lie groups). Moreover $G$ is large scale equivalent in the sense of \cite[Definition 5.3]{HumeMackTess} to a graph with bounded degree, hence the first statement follows from Proposition \ref{prop:think-profiles-general-upper}. 
	Thus if $G$ is analytically thick it has $\Lambda_G^p(r) \simeq r/\log(r)$ for $p\in[1,\infty)$ by \cite[Propositions 7.2]{HumeMackTess}.
For $p=\infty$, $G$ does not have polynomial growth (as those groups are analytically thin), so it must have exponential growth \cite{Gui-73-crois-poly,Jenkins_Growth}, thus the result follows from \cite[Proposition 6.1]{HumeMackTess}.	
\end{proof}

\subsection{Diestel--Leader graphs}\label{sec:DLprofile}
This section is dedicated to the proof of the following theorem.
\begin{theorem}\label{DLlowerbound}
	For all $p\in[1,\infty]$, the Diestel--Leader graph $X=\DL(2,2)$ satisfies
\[
 \Lambda^p_X(r) \simeq r/\log(r).
\]
\end{theorem}
\begin{proof}
	Fix a Busemann function $h$ on the $3$-regular tree $\mathcal{T}_3$.
	Consider copies $T_1, T_2$ of $\cT_3$ so that $V\DL(2,2) \subset VT_1\times VT_2$.
Suppose $k \in \N$ is given.
	Fix $o_1 \in VT_1$ with $h(o_1)=k$ and $o_2 \in VT_2$ with $h(o_2)=0$.
	Consider the induced subgraph $\Gamma_k$ of $\DL(2,2)$ with vertex set
\[
 V_k=\setcon{(x,y)}{\begin{array}{l} d(o_1,x)=k-h(x), \\ d(o_2,y)=-h(y), \\ 0 \leq d(o_1,x),d(o_2,y)\leq k\end{array}}.
\]
For $0\leq t\leq k$, let $V_k^t=\setcon{(x,y)\in V_k}{h(x)=t}$. We call a directed edge $(x,y)(x',y')$ in $\Gamma_k$ an \textbf{up} edge if $h(x')>h(x)$ and a \textbf{down} edge otherwise.

Given a pair of vertices $(x,y)$, $(x',y')$ in $\Gamma_k$ with $h(x)=t\geq h(x')=s$ we assign a family $\mathcal C_{(x,y),(x',y')}$ of $2^{k-t+s}$ paths of length $2k-t+s$ connecting them as follows
\begin{equation}\label{updown}
 c_{z,z'} = (x,y) \uparrow (o_1,z') \downarrow (z,o_2) \uparrow (x',y'),
\end{equation}
where $z'$ varies over the $2^{k-t}$ vertices in the second coordinate $T_2$ satisfying $k=d(o_2,z')=d(o_2,y)+d(y,z')$, and $z$ varies over the $2^s$ vertices in the first coordinate $T_1$ satisfying $k=d(o_1,z)=d(o_1,x')+d(x',z)$. Each path $c_{z,z'}$ is uniquely determined by the two vertices $z,z'$ and the length restriction. This forces the path to split into three parts as indicated in (\ref{updown}): the first and last consisting only of up edges and the second only down edges.

We split the remainder of the proof into three claims.

\textbf{Claim 1:} $\abs{\Gamma_k} = (k+1)2^{k}$.

For each $0\leq t \leq k$, there are $2^k$ pairs $(x,y)$ such that $h(x)=t$: $2^t$ different possibilities for $x$ and $2^{k-t}$ possible $y$. Thus there are $(k+1)2^k$ vertices in total.

\textbf{Claim 2:} Every edge in $E\Gamma_k$ is contained in at most $2^{2k-t+s}$ paths connecting a vertex $(x,y)\in V^t_k$ to a vertex $(x',y')\in V^s_k$.

Fix an up edge $e=(a,b)(a',b')$ so $0\leq h(a) \leq k-1$. For $t\geq s$, denote by $N_e(t,s)$ the number of times $e$ appears with either orientation in one of the chosen paths which starts at some $(x,y)$ where $h(x)=t$ and ends at some $(x',y')$ where $h(x')=s$.

If the edge $e$ appears in the first section of some path in $\mathcal C_{(x,y),(x',y')}$ then $h(a)\geq t$ and $y$ is the unique vertex satisfying $d(o_2,b)=d(o_2,y)+d(y,b)$ and $h(y)=-t$. Moreover, $x$ can be any of the $2^{h(a)-t}$ vertices satisfying $d(o_1,x)=d(o_1,a)+d(a,x)$. For any of these $2^{h(a)-t}$ choices of a pair $(x,y)$ and every choice of $(x',y')$, the edge $e$ appears in exactly $2^{-(h(a')-t)}$ proportion of the paths in $\mathcal C_{(x,y),(x',y')}$. All of this analysis is independent of the choice of $(x',y')$ so for each of the $2^{h(a)-t}2^k$ possible choices of suitable $(x,y),(x',y')$, $e$ appears in $2^{-(h(a')-t)}2^{k-t+s}$ of the paths in $\mathcal C_{(x,y),(x',y')}$.

Thus $e$ appears in the first section of some path at most $2^{2k-t+s-1}$ times. A similar analysis shows that if $h(a)<s$ then $e$ appears in the third section of some path at most $2^{2k-t+s-1}$ times, and if $h(a)\geq s$ then $e$ never appears in the third section.

	We are left to analyse the second section of the paths. If either $h(a)\geq t$ or $h(a)<s$ then the above analysis also holds and $e$ is used $2^{2k-t+s-1}$ times. Otherwise, either none or all of the paths in $\mathcal C_{(x,y),(x',y')}$ contain $e$. In the case where it is all of them, we have $2^{t-h(a')}$ possibilities for $y$, $2^{h(a)-s}$ possibilities for $x'$, $2^{k-t+s}$ choices of the pair $(z,z')$, $2^{k-t}$ choices of $x$ and $2^s$ choices of $y'$. Therefore, $e$ appears at most $2^{2k-t+s-1}$
times as a down edge. Combining these observations, we see that the total number of different paths containing $e$ is at most $2\cdot 2^{2k-t+s-1}$ as required.

\textbf{Claim 3:} $h^1(\Gamma_k) \succeq 1/k$.

	Let $f:V_k\to\R$ be a non-constant function with $\sum_{v\in V_k}f(v)=0$. For an edge $e\in E\Gamma_k$ with endpoints $v, w$, let $|\nabla f(e)| = |f(v)-f(w)|$.  Using the triangle inequality, we have
\begin{align*}
		\sum_{v,w\in V_k} \abs{f(v)-f(w)} 
		 & \leq 2\sum_{t\geq s} \sum_{v\in V_k^t,w\in V_k^s}\frac{1}{2^{k-t+s}} \left( \sum_{\gamma\in\mathcal{C}_{v,w}}\sum_{e\in\gamma} \abs{\nabla f(e)}\right)
		\\ & = 2\sum_{t\geq s}\frac{2^{2k-t+s}}{2^{k-t+s}} \Bigg(\sum_{\substack{v\in V_k^t,\\w\in V_k^s}}  \sum_{\gamma\in\mathcal{C}_{v,w}}\sum_{e\in\gamma} \frac{1}{2^{2k-t+s}}\abs{\nabla f(e)}\Bigg)
		\\ & \leq  2^{k+1}\sum_{t\geq s}\sum_{e\in E\Gamma_k} \abs{\nabla f(e)}.
\end{align*}

	We deduce that 
	\[ \sum_{v,w\in V_k} \abs{f(v)-f(w)} \leq  {2^{k}(k+1)(k+2)} \sum_{e\in E\Gamma_k} \abs{\nabla f(e)}.\]
Now, for every finite graph $\Gamma$ with maximal degree $d$,  and every function $f:V\Gamma\to \R$ we have
\[
 \sum_{e\in E\Gamma} \abs{\nabla f(e)} \leq \frac{d}{2} \sum_{v\in V\Gamma} |\nabla f(v)| = \frac{d}{2}\norm{\nabla f}_1,
 \]
which implies that
	if $\sum_{w\in V_k}f(w)=0$,
\begin{align*}
	\abs{V_k} \|f\|_1 
	& = \abs{V_k} \sum_{v\in V_k} \left| f(v)-\frac{1}{\abs{V_k}}\sum_{w\in V_k} f(w) \right|
	\\ & \leq  \sum_{v,w\in V_k} |f(v)-f(w)| \leq {\frac{d}{2}2^{k}(k+1)(k+2)}\norm{\nabla f}_1
	\\ & = {\frac{d}{2}(k+2)}\abs{V_k}\norm{\nabla f}_1.
\end{align*}
Thus $h^1(\Gamma_k)\succeq 1/k$ as required.

Since for each $k\geq 1$, $\abs{\Gamma_k} \leq \abs{\Gamma_{k+1}} \leq 4\abs{\Gamma_k}$, for every $r\geq 4$ there exists a $k$ such that $\frac{r}{4}\leq\abs{\Gamma_k}\leq r$. Moreover, $\abs{\Gamma_k}=(k+1)2^k$, so $\log_2(r) \geq k$.

Given Proposition~\ref{prop:think-profiles-general-upper}, the proof is now complete, since
\[
 \Lambda^1_X(r) \geq \abs{\Gamma_k}h^1(\Gamma_k) \succeq \frac{r}{\log(r)}.\qedhere
\]
\end{proof}

\begin{corollary}\label{wreathlb} For every non-trivial finitely generated group $H$ and every infinite finitely generated group $K$, the wreath product $G=H\wr K$ satisfies $\Lambda^p_G(r)\gtrsim r/\log(r)$ for all $p\in[1,\infty]$. If, in addition, $H$ is finite and $K$ is virtually cyclic then $\Lambda^p_G(r)\simeq r/\log(r)$ for all $p\in[1,\infty]$.
\end{corollary}
\begin{proof} For $p=\infty$, $G$ has exponential growth so $\Lambda^\infty_G(r)\simeq r/\log(r)$ by \cite[Proposition 6.1]{HumeMackTess}.

By Theorem \ref{DLlowerbound} and Theorem \ref{thm:embSgraph}
\[
	r/\log(r) \lesssim \Lambda^1_{\DL(2,2)}(r) \lesssim \Lambda^1_G(r).
\]
	Finally, $G$ has Assouad--Nagata dimension $1$ whenever $H$ is finite and $K$ is virtually cyclic, so by Proposition \ref{prop:think-profiles-general-upper}, for all $p\in[1,\infty)$
\[
 r/\log(r) \lesssim \Lambda^1_G(r)\lesssim \Lambda^p_G(r) \lesssim r/\log(r). \qedhere
\]
\end{proof}

\subsection{Poincar\'{e} profiles of $\Osc$}\label{sec:HeisRtimesZ}

Let $\Heis_3$ denote the real Heisenberg group, consider the action $\R\curvearrowright \Heis_3$ given by
\[
 \left(
 \begin{array}{ccc} 
 1 & a & c \\ 
 0 & 1 & b \\ 
 0 & 0 & 1
 \end{array}
 \right)\cdot \psi(k) =
 \left(
 \begin{array}{ccc} 
 1 & e^ka & c \\ 
 0 & 1 & e^{-k}b \\ 
 0 & 0 & 1
 \end{array}
 \right) , 
\]
and construct the corresponding semidirect product $\Osc=\Heis_3\rtimes_\psi\R$. For brevity, we will omit $\psi$ in what follows and we introduce the following shorthands for elements of $\Heis_3$ and $\Osc$ respectively
\[
 (a,b,c):=\left(
 \begin{array}{ccc} 
 1 & a & c \\ 
 0 & 1 & b \\ 
 0 & 0 & 1
 \end{array}
 \right)
 \quad (a,b,c;k):=\left(\left(
 \begin{array}{ccc} 
 1 & a & c \\ 
 0 & 1 & b \\ 
 0 & 0 & 1
 \end{array}
 \right),k\right).
\]
For example, the group operation in $\Osc$ is 
\[ (a,b,c;k)(a',b',c';k')=(e^{k'}a+a',e^{-k'}b+b',c+c'+e^{k'}ab';k+k').\]
In what follows we will work with the cocompact subgroup $G=\Heis_3\rtimes\Z$. To define Poincar\'e profiles on $G$, we use the word metric from the compact generating set $[-1,1]^3\times \{-1,0,1\}$,
and we use the following notion of gradient: 
 Given a function $f:X\to\R$ on a metric space $X$, and $a\geq 2$, we define $|\nabla_af|:X\to \R$ by
\[
 |\nabla_a f|(x)=\sup\setcon{|f(y)-f(y')|}{y,y'\in B(x,a)}.
\]
Full details about Poincar\'e profiles with respect to this notion of gradient of a function are presented in \cite[Sections 3 and 4]{HumeMackTess}.

Our goal is the following:

\begin{theorem}\label{thm:HeisExt}
	For all $p\in [1,\infty]$,
$
 \Lambda^p_{\Osc}(r)\simeq r/\log(r).
$
\end{theorem}
By Corollary~\ref{cor:general-upper-bound} it suffices to prove that $\Lambda^1_G(r)\gtrsim r/\log(r)$.

The proof has four main steps: first we define special families of sequences for each pair of points in chosen subsets of $\Heis_3$, second we show these sequences have ``small overlap'', third from them we construct coarse paths in $G$, and finally we get a lower bound on $\Lambda^1_G$ by controlling the change in functions by their gradient on these paths.

Let us demonstrate the approach with a simpler example that avoids some of the technicalities required. Consider $G=\R^4\rtimes_{(1,-1,1,-1)}\Z$ with the word metric from the compact generating set $[-1,1]^4\times\{-1,0,1\}$.  For each $t$ define  $H_t=[-e^t,e^t]^4\subseteq\R^4$. Given any pair $\underline{a}=(a_1,a_2,a_3,a_4)$, $\underline{b}=(b_1,b_2,b_3,b_4)$ we define a special sequence $S(\underline{a},\underline{b})$ as follows:
\[
 (a_1,a_2,a_3,a_4) \to (b_1,a_2,a_3,a_4) \to (b_1,b_2,a_3,a_4) \to (b_1,b_2,b_3,a_4) \to (b_1,b_2,b_3,b_4)
\]
By ``small overlaps'', we mean that if for some $1\leq n \leq 5$ we know $\underline{c}=(c_1,c_2,c_3,c_4)$ is the $n$th term in the sequence $S(\underline{a},\underline{b})$ then $a_i=c_i$ for $n \leq i \leq 4$ and $b_i=c_i$ for $1 \leq i <n$.  We interpret this as saying that the set of points in $H_t\times H_t$ whose special sequence contains a given $\underline{c}$ in the $n$th term is a ``copy'' of $H_t$.

Next define 
\begin{multline*}
	G_t=\bigcup_{k=-t}^t (\psi(k)H_k,k)=\big\{ (x,y,z,w;k)\in G : x,z\in[-e^{t+k},e^{t+k}],\ \\ y,w\in[-e^{t-k},e^{t-k}],\ -t\leq k \leq t \big\}.
\end{multline*}
For each pair $\underline{a_i}=(x_i,y_i,z_i,w_i;k_i)\in G_t$ for $i=1,2$ we define a discrete path $P(\underline{a_1},\underline{a_2})$ connecting them with adjacent terms at distance $\leq 2$ with respect to the word metric, and having at most $10t+5$ entries in $G_t$.
In the description below of $P(\underline{a_1},\underline{a_2})$, ``$\to_\Z$'' indicates that we move between the two points by applying $\psi^{\pm 1}$ the appropriate number of times. Each application of $\psi^{\pm 1}$ defines a new point on the path.  Observe that all points defined are of the form $(\underline{c}\cdot\psi(k);k)$ for some $-t \leq k \leq t, \underline{c}\in P(\underline{a_1},\underline{a_2})$.
\begin{eqnarray*}
 (x_1,y_1,z_1,w_1;k_1) & \to_{\Z} & (e^{-t-k_1}x_1,e^{t+k_1}y_1,e^{-t-k_1}z_1,e^{t+k_1}w_1;-t) \\
 & \to & (e^{-t-k_2}x_2,e^{t+k_1}y_1,e^{-t-k_1}z_1,e^{t+k_1}w_1;-t) \\
 & \to_\Z & (e^{t-k_2}x_2,e^{-t+k_1}y_1,e^{t-k_1}z_1,e^{-t+k_1}w_1;t) \\
 & \to & (e^{t-k_2}x_2,e^{-t+k_2}y_2,e^{t-k_1}z_1,e^{-t+k_1}w_1;t) \\
 & \to_\Z & (e^{-t-k_2}x_2,e^{t+k_2}y_2,e^{-t-k_1}z_1,e^{t+k_1}w_1;-t) \\
 & \to & (e^{-t-k_2}x_2,e^{t+k_2}y_2,e^{-t-k_2}z_2,e^{t+k_1}w_1;-t) \\
 & \to_\Z & (e^{t-k_2}x_2,e^{-t+k_2}y_2,e^{t-k_2}z_2,e^{-t+k_1}w_1;t) \\
 & \to & (e^{t-k_2}x_2,e^{-t+k_2}y_2,e^{t-k_2}z_2,e^{-t+k_2}w_2;t) \\
 & \to_\Z & (x_2,y_2,z_2,w_2;k_2).
\end{eqnarray*}

Given $f:G_t\to\R$, for any $x=(\underline{a};r), y=(\underline{b};s) \in G_t$, by the triangle inequality we have
\begin{eqnarray*}
	|{f(x)-f(y)}| \leq \sum_{k=-t}^t \sum_{\underline{c}\in S(\underline{a},\underline{b})} |\nabla_2 f|(\underline{c}\cdot\psi(k);k).
\end{eqnarray*}
Therefore,

\begin{eqnarray}\label{Eucnormexample}
	\mu(G_t)\|f-f_{G_t}\|_1 & \leq & \int_{G_t\times G_t} |{f(x)-f(y)}| d\mu(G_t\times G_t) \\ 
	& \leq & (2t+1)^2\int_{H_t\times H_t} \sum_{k=-t}^t \sum_{\underline{c}\in S(\underline{a},\underline{b})} |\nabla_2 f|({\underline{c}}\cdot\psi(k);k)d\mu(H_t\times H_t).\notag
\end{eqnarray}
Now we split $\int_{H_t\times H_t}$ into 8 integrations over the variables $x_i,y_i,w_i,z_i$. We also split into 5 terms coming from the five positions $\underline c$ could take in the sequence $S(\underline{a},\underline{b})$. For the $n$th of these we reorder the integration as follows:
\begin{equation}\label{intpart}
\int_{a_1,\ldots,a_{n-1},b_n,\ldots,b_4} \sum_{k=-t}^t \left(\int_{b_1,\ldots,b_{n-1},a_n,\ldots,a_4} \sum_{\underline{c}= S(\underline{a},\underline{b})_n} |\nabla_2 f|(\underline{c}\cdot\psi(k);k)\right)
\end{equation}
By the ``small overlap'' condition, the bracketed part of the above expression is simply $\int_{\underline{c}\in H_t} |\nabla_2 f|( \underline c\cdot\psi(k);k)$, so ($\ref{intpart}$) is bounded from above by
\[
	\mu(H_t)\int_{G_t} |\nabla_2 f|={\frac{1}{2t+1}}\mu(G_t)\int_{G_t} |\nabla_2 f|.
\]
Combining this with ($\ref{Eucnormexample}$) and cancelling $\mu(G_t)$ we deduce that
\[
 \|f-f_{G_t}\|_1 \leq 5(2t+1)\int_{G_t} |\nabla_2 f|.
\]
We deduce that $h^1(G_t) \geq \frac{1}{5(2t+1)} \geq \varepsilon/\log(\mu(G_t))$ for some $\varepsilon>0$ which is independent of $t$.  Since for any $r$ one can find $t$ with $\mu(G_t)$ comparable to $r$, thus $\Lambda^1_G(r) \gtrsim r/\log(r)$.
\medskip

Before proceeding with the lower bound on $\Osc$ we briefly mention some of the difficulties of generalising our approach to $\R^4\rtimes\Z$. The first, and most obvious, is that the special sequences in $\Heis_3$ are longer than those in $\R^4$ and the ``small overlap'' condition is more involved. It is crucial to our argument that the Haar measure ${d\mu_H}$ on $\Heis_3$ coincides with the Lebesgue measure $dxdydz$ on $\R^3$ allowing us to split integrals. However, in making this change we will have to apply three changes of variables, some of which have non-trivial Jacobians. These also need to be controlled.

\medskip
\noindent{\bf Notations and conventions.}
 In what follows we let $t$ be a positive integer. 
\begin{itemize}
\item Define a subset of $\Heis_3$ as follows:
\[
 H_t=\setcon{(a,b,c)}{-e^{t}\leq a < e^{t},\ -e^{t}\leq b < e^{t},\ 
-2e^{2t}< c \leq 2e^{2t}}\subset\Heis_3.
\]

\item Define a subset of $G=\Heis_3\rtimes \Z$ as follows:
\[
 G_t = \setcon{(a,b,c;k)}
{ \begin{array}{l}
-e^{t+k}\leq a < e^{t+k},\ -e^{t-k}\leq b < e^{t-k} \\
-2e^{2t}< c \leq 2e^{2t},\ -t\leq k\leq t \end{array}},
\]
i.e. $G_t= \bigcup_{k=-t}^t \left((H_t;0) \cdot (1_{\Heis_3};k)\right)=\bigcup_{k=-t}^t (H_t\cdot\psi(k);k)$.
\item We shall consider the following subset of $\Heis_3$: $$S_t=\setcon{A(s),B(s)}{-2e^t\leq s \leq 2e^t}\subseteq \Heis_3,$$ where $A(s)=(s,0,0)$ and $B(s)=(0,s,0)$.
\end{itemize}

\noindent \textbf{Step 1: Defining sequences in $H_t$.}
Let us fix $x'=(a_1,b_1,c_1),y'=(a_2,b_2,c_2)\in H_t$. We will define a sequence $x'=x'_0,\ldots,x'_{54}=y'$ such that $(x'_i)^{-1}x'_{i+1}\in S_t$ for all $i$.

Define $\overline{a_i}\in[0,\frac13 e^t)$ such that $a_i-\overline{a_i}\in \frac13e^t\Z$, and $\overline{b_i}\in[0,\frac13 e^t)$ such that $b_i-\overline{b_i}\in \frac13e^t\Z$; note $-e^t \leq b_i-\overline{b_i} \leq \frac{2}{3}e^t$. 

Define $\overline{c_1}\in(-\frac12e^{2t},0]$ such that $\overline{c_1}-c_1=\frac{l_1}{2}e^{2t}$ for some $l_1\in\{-4,-3,\ldots,3\}$. Finally, define $\overline{c_2}$ such that $c_2-\overline{c_2}=\frac{l_2}{2}e^{2t}$ for some $l_2\in\Z$ and
\begin{equation}\label{eq:c2}
 \frac12e^{2t}< \overline{c_2}-\overline{c_1}-\overline{a_1}(\overline{b_1}-b_1) - \overline{a_2}(b_2-\overline{b_1}) \leq e^{2t}.
\end{equation}
As we shall later see, $\frac{-2}{3}e^{2t} \leq \overline{c_2}\leq\frac{5}{3}e^{2t}$ so $l_2\in \{-7,-6,\ldots,5\}$.
Set $k= (\overline{c_2}-\overline{c_1}-\overline{a_1}(\overline{b_1}-b_1) - \overline{a_2}(b_2-\overline{b_1}) )^{1/2}$.

The first sixteen steps travel from $x'_0=(a_1,b_1,c_1)$ to $x'_{16}=(a_1,b_1,\overline{c_1})$. We have
\[
 x'_{16}=x'_0 C^{l_1}
\]
where $C$ is any cyclic conjugate of the commutator $A(\frac{1}{\sqrt 2}e^t)B(\frac{1}{\sqrt 2}e^t)A(-\frac{1}{\sqrt 2}e^t)B(-\frac{1}{\sqrt 2}e^t)$ and $\abs{l_1}\leq 4$. Now $C^{l_1}$ decomposes as a product of at most 16 elements of $S_t$. Splitting into four cases depending on the signs of $a_1,b_1$, at least one of these paths will remain inside $H_t$, for example, the commutator given above always stays inside $H_t$ when $a_1,b_1\leq 0$. If $x'_i=(a_1,b_1,\overline{c_1})$ occurs for the first time with $i< 16$ then we simply define all terms in the sequence from $x'_i$ to $x'_{16}$ to be equal to $(a_1,b_1,\overline{c_1})$.

The next steps are given by multiplying by suitable $A(s)$ or $B(s)$ in turn:
\begin{eqnarray*}
x'_{17}  & = & (\overline{a_1},b_1,\overline{c_1}) =x'_{16} A(\overline{a_1}-a_1) \\
	x'_{18}  & = & (\overline{a_1},\overline{b_1},\overline{c_1}+\overline{a_1}(\overline{b_1}-b_1))  = x'_{17} B(\overline{b_1}-b_1)\text{, etc.} \\
x'_{19}  & = & (\overline{a_2},\overline{b_1},\overline{c_1}+\overline{a_1}(\overline{b_1}-b_1)) \\
x'_{20}  & = & (\overline{a_2},\overline{b_2},\overline{c_1}+\overline{a_1}(\overline{b_1}-b_1)+\overline{a_2}(\overline{b_2}-\overline{b_1})) \\
x'_{21}  & = & (\overline{a_2}-k,\overline{b_2},\overline{c_1}+\overline{a_1}(\overline{b_1}-b_1)+\overline{a_2}(\overline{b_2}-\overline{b_1})) \\
x'_{22}  & = & (\overline{a_2}-k,\overline{b_2}-k,\overline{c_2}-\overline{a_2}(b_2-\overline{b_2})-k\overline{a_2}) \\
x'_{23}  & = & (\overline{a_2},\overline{b_2}-k,\overline{c_2}-\overline{a_2}(b_2-\overline{b_2})-k\overline{a_2}) \\
x'_{24}  & = & (\overline{a_2},\overline{b_2},\overline{c_2}-\overline{a_2}(b_2-\overline{b_2})) \\
x'_{25}  & = & (\overline{a_2},b_2,\overline{c_2}) \\
x'_{26}  & = & (a_2,b_2,\overline{c_2})
\end{eqnarray*}

It is clear from the definitions given above that the $a$ and $b$ coordinates remain within $H_t$. For the $c$ coordinate, we have: 
\begin{eqnarray*}
	-\frac{13}{18}\exp(2t)  \leq & \overline{c_1}+\overline{a_1}(\overline{b_1}-b_1) & \leq  \frac{1}{3}\exp(2t) \\
-\frac{5}{6}\exp(2t)  \leq & \overline{c_1}+\overline{a_1}(\overline{b_1}-b_1)+\overline{a_2}(\overline{b_2}-\overline{b_1}) & \leq  \frac{4}{9}\exp(2t) 
\end{eqnarray*}
Adding $k^2$ to both sides, and then adding either $-k\overline{a_2}$ or $\overline{a_2}(b_2-\overline{b_2})$ we get:
\begin{eqnarray*}
-\frac13\exp(2t)  \leq & \overline{c_2}-\overline{a_2}(b_2-\overline{b_2}) & \leq  \frac{13}{9}\exp(2t) \\
-\frac23\exp(2t)  \leq & \overline{c_2}-\overline{a_2}(b_2-\overline{b_2})-k\overline{a_2} & \leq \frac{13}{9}\exp(2t) \\
-\frac{2}{3}\exp(2t)  \leq & \overline{c_2} & \leq \frac{5}{3}\exp(2t)
\end{eqnarray*}
and therefore the sequence remains inside $H_t$.

For the final part we apply conjugates as in the first step, except that this time, we only have $|l_2|\leq 7$, meaning as many as 28 steps could be required. As before, we insist on using the full 28 steps (to simplify notation later) so if $x'_i=y$ occurs for the first time at some $i<54$ we simply define $x'_j=y$ whenever $i\leq j\leq 54$. Thus, to each pair of points $x',y'\in H_t$ we have assigned a sequence $x'_0,\ldots,x'_{54}$.  We define $g_i:H_t\times H_t \to H_t$ by $g_i(x',y')=x'_i$.

\noindent \textbf{Step 2: Showing that these sequences have ``small overlap".}
These sequences retain a considerable amount of information about their initial and terminal points. We want to show that for any fixed $v\in H_t$ the set of pairs $(x',y')\in H_t\times H_t$ such that $v$ lies on the sequence connecting $x'$ to $y'$ as defined above is ``small''. We now make this precise.

For all $\underline{v}=(v_1,v_2,v_3)\in\R^3$ and $\sigma=(\sigma_1,\sigma_2,\sigma_3)\in\set{1,2}^3$, we define
\[
 \Heis^2_{\underline{v},\sigma}=\setcon{((a_1,b_1,c_1),(a_2,b_2,c_2))\in(\Heis_3)^2}{a_{\sigma_1}=v_1,\ b_{\sigma_2}=v_2,\ c_{\sigma_3}=v_3}
\]
Recall that for $i\in\{0,\ldots, 54\}, x'\in H_t$ and $y'\in H_t$ the $i$th term in the sequence connecting $x'$ to $y'$ is $g_i(x',y')$.

\begin{lemma}\label{lem:memory}  For each $i$, there exists $\sigma^i=(\sigma^i_1,\sigma^i_2,\sigma^i_3)$ such that $g_i^{-1}(a',b',c')$ intersects each $\Heis^2_{\underline{v},\sigma^i}$ in at most $M=2^93^2$ points.
\end{lemma}
\begin{proof}
We will prove this by finding an appropriate $\sigma^i$ in each case and proving a bound on the intersection. As a shorthand, let us write $x'_i=(A_i,B_i,C_i)$.

\noindent\textbf{Case} $i\leq 16$ and $i\geq 26$: For $i\leq 16$, set $\sigma^i=(2,2,2)$, we know that $a'\in\set{a_1-\frac{1}{\sqrt 2}e^t,a_1,a_1+\frac{1}{\sqrt 2}e^t}$, $b'\in\set{b_1-\frac{1}{\sqrt 2}e^t,b_1,b_1+\frac{1}{\sqrt 2}e^t}$ and $c' = c_1+\epsilon a_1e^t+\frac{p}{2}e^{2t}$ for some $\epsilon\in\set{-1,0,1}$ and at most $8$ possible values of $p\in\Z$. Therefore the intersection with each $H_{\overline{v},\sigma^i}$ contains at most $3^3\cdot 8$ points, corresponding to the three choices of $a',b',\epsilon$ and eight choices of $p$ respectively. We argue similarly for stage 7, where $i\geq 28$, except that we must set $\sigma^i=(1,1,1)$.

\noindent\textbf{Case} $i=17$: Set $\sigma^i=(2,2,2)$. Given $(\overline{a_1},b_1,\overline{c_1})$ there are at most 6 possibilities for $a_1$ and 8 possibilities for $c_1$.

\noindent\textbf{Case} $i=18$: Set $\sigma^i=(2,2,2)$. Given $(\overline{a_1},\overline{b_1},\overline{c_1}+\overline{a_1}(\overline{b_1}-b_1))$ there are at most 6 possibilities for $a_1$, 6 for $b_1$ and, for each of these choices, 8 possibilities for $c_1$.

\noindent\textbf{Case} $i=19$: Set $\sigma^i=(1,2,2)$. There are at most $6$ possibilities for $a_2$, $6$ possibilities for $b_1$, and for each of these, $8$ possibilities for $c_1$.

\noindent\textbf{Case} $i=20$: Set $\sigma^i=(1,1,2)$. There are at most $6$ possibilities for $a_2$ and $6$ possibilities for $b_2$. From these we determine $\overline{c_1}$ exactly, and there are $8$ possibilities for $c_1$.

\noindent\textbf{Case} $i=21$: Set $\sigma^i=(1,1,2)$. We have
\[
 k^2 + k (b_2-\overline{b_2}) + A_{21}(b_2-\overline{b_2}) - \overline{c_2} + C_{21} = 0.
\]
	There are at most $6$ possible values of $b_2-\overline{b_2}$ and at most $8$ possible values of $\overline{c_2}$ (given that $c_2$ has been fixed). So solving the quadratic equation, there are at most $96$ possible values of $k$. For each one we determine $\overline{a_2}$ and then $\overline{c_1}$. Then there are at most $6$ possibilities for $a_2$ and $8$ for $c_1$, giving at most $96\cdot 6\cdot 8\leq M$ possibilities.  

\noindent\textbf{Case} $i=22$: Set $\sigma^i=(1,1,1)$. We have
\[
 C_{22} = \overline{c_1} + \overline{a_1}(\overline{b_1}-b_1) + (A_{22}+k)(B_{22}+k-\overline{b_1}) - A_{22}k
\]
which is a monic quadratic in $k$ where we are given all the coefficients. We solve this, giving at most two possibilities for $k$. Using $k$ we determine $\overline{a_2}$ and $\overline{b_2}$ giving six possible $a_2$ and $b_2$ in each case. Finally, we determine $\overline{c_2}$, giving eight possible values of $c_2$.

\noindent\textbf{Case} $i=23$: Set $\sigma^i=(1,1,1)$. We have
\begin{eqnarray*}
 k^2 & = & C_{23}+k\overline{a_2} -\overline{c_1} - \overline{a_1}(\overline{b_1}-b_1)-\overline{a_2}(\overline{b_2}-k+k-\overline{b_1}) \\ & = & C_{23}-\overline{c_1} - \overline{a_1}(\overline{b_1}-b_1)-A_{23}B_{23}+A_{23}\overline{b_1}.
\end{eqnarray*}
Thus we may calculate $k$ exactly, and use this to determine $\overline{b_2}$. There are then $6$ possibilities for $b_2$. For each one, we then calculate $\overline{c_2}$ using the original definition of $k^2$. There are at most $6$ possibilities for $a_2$, and 8 possibilities for $c_2$.

\noindent\textbf{Case} $i=24,25$: The same technique as $i=18,17$ respectively work, except with $\sigma^i=(1,1,1)$.
\end{proof}

\noindent\textbf{Step 3: Defining sequences in $G_t$.}

Fix $x=(a_1,b_1,c_1;r)$ and $y=(a_2,b_2,c_2;s)$ in $G_t$ and let $x'=(e^{-r}a_1,e^rb_1,c_1)$ and $y'=(e^{-s}a_2,e^sb_2,c_2)$ be the corresponding points in $H_t$.
Fix the sequence of points $x'=x'_{0},\ldots,x'_{54}=y'$ in $H_t$ constructed in Step 1.
We now construct a sequence $x=x_0,\ldots,x_m=y$ of points in $G_t$ such that $d_G(x_i,x_{i+1})\leq 2$ for all $i$.

Starting from $x_0=(a_1,b_1,c_1;r)$, define $x_i=(e^{-i}a_1,e^ib_1,c_1;r-i)$ for $0\leq i \leq r$ if $r \geq 0$ and $x_i=(e^{i}a_1,e^{-i}b_1,c_1;r+i)$ for $1\leq i \leq -r$ if $r<0$. Now $x_{|r|}=(x'_{0};0)$. We now define parts of the sequence $(x_j)$ going from $x_{k_i}=(x'_i;0)$ to $x_{k_{i+1}}=(x'_{i+1};0)$ for each $0\leq i < 54$; note that $k_0=|r|$. If $x'_i$ and $x'_{i+1}$ differ by some $A(s)$ then define
\[
x_{k_i+j}=\left\{\begin{array}{rl}
x_{k_i}\cdot (\underline{0};-j) & \textup{if }0\leq j \leq t, 
\\
x_{k_{i+1}}\cdot (\underline{0};-(2t+1)+j) & \textup{if }t+1\leq j \leq 2t+1.
\end{array}
\right.
\]
Note that here $x_{k_i+t+1}=x_{k_i+t}(A(se^{-t});0)$ with $|se^{-t}|\leq 2$.
If $x'_i$ and $x'_{i+1}$ differ by some $B(s)$ then define 
\[
x_{k_i+j}=\left\{\begin{array}{rl}
x_{k_i}\cdot (\underline{0};j) & \textup{if }0\leq j \leq t, 
\\
x_{k_{i+1}}\cdot (\underline{0};2t+1-j) & \textup{if }t+1\leq j \leq 2t+1.
\end{array}
\right.
\]
Note that here $x_{k_i+t+1}=x_{k_i+t}(B(se^{-t});0)$ with $|se^{-t}|\leq 2$.
Finally, to get from $(x'_{54};0)$ to $(a_2,b_2,c_2;s)$ we apply $(\underline{0};s)$ to get the last $\abs{s}$ steps of the sequence, so the entire sequence has length $\abs{r} + 54(2t+1) + \abs{s}$, which is at most a bounded multiple of $t$.

{\noindent\textbf{Step 4: Controlling functions by their gradients.}}
In this step we show the following bound.
Let $\mu_G$ be the Haar measure on $G=\Heis_3\rtimes\Z$; note that this restricts to the Haar measure $\mu_H$ of $\Heis_3$ on each copy $\Heis_3\times\{i\}$.  Recall that $\mu_H$ agrees with Lebesgue measure under the identification of $(a,b,c) \in \Heis_3$ with $(a,b,c) \in \R^3$.
\begin{lemma}\label{lem:memory2}
	There exists a constant $C$ such that for any measurable function $f:G_t\to\R$ we have
\begin{equation}\label{eq:memory}
 \int_{x,y\in G_t} \abs{f(x)-f(y)} d\mu_G^{2}(x,y) \leq Ct^2e^{4t} \int_{z\in G_t} \abs{\nabla_2 f(z)} d\mu_G(z).
\end{equation}

\end{lemma}
\begin{proof} 
Using our defining sequences, the left-hand integral in equation ($\ref{eq:memory}$) is bounded from above by
\begin{equation}\label{eq:intG}
\int_{x,y\in G_t}  \sum_{i=0}^{m(x,y)}  \abs{\nabla_2 f(x_i(x,y))}d\mu_G^{2}(x,y),
\end{equation}
where $x_i(x,y)$ is the $i$th term of the defining sequence from $x$ to $y$, and $m(x,y)$ is its last index.
	Define $G_t^k=\setcon{(a,b,c;r)\in G_t}{r=k}$. Let $H_t^k=\setcon{(a,b,c)}{(a,b,c;k)\in G_t}$ which is by definition equal to $H_t\cdot\psi(k)$. Once $k'$ is fixed $(a',b',c';k')$ is in the sequence $x_0,\ldots,x_m$ only if $(e^{-k'}a',e^{k'}b',c')=(a',b',c')\cdot \psi(-k')$ is in the corresponding sequence $x'_{0},\ldots,x'_{54}$, and if a single point appears more than once in the sequence $x_0,\ldots,x_m$ then we can shorten the sequence so that this does not happen. Since $d_G(x_i,x_{i+1})\leq 2$ for all $i$, the left hand expression in ($\ref{eq:intG}$) is bounded from above by
\[ \sum_{r_1, r_2=-t}^t \sum_{k'=-t}^t \sum_{j=0}^{54} \int_{x\in H_t^{r_1}}\int_{y\in H_t^{r_2}} \abs{\nabla_2 f(x'_j\cdot \psi(k');k')}d\mu_H(y)d\mu_H(x)
\]
	where for $x=(e^{r_1}a_1,e^{-r_1}b_1,c_1;r_1), y=(e^{r_2}a_2,e^{-r_2}b_2,c_2;r_2)$ we use the shorthand $x'_j$ to represent the point $g_j((a_1,b_1,c_1),(a_2,b_2,c_2))$, which we recall is the $j$th term of the defining sequence from $(a_1,b_1,c_1)$ to $(a_2,b_2,c_2)$.

Now fix $r_1, r_2$, $k'$ and $j$. Our next goal is to bound the integral \begin{equation}\label{bound1}
\int_{x\in H_t^{r_1}}\int_{y\in H_t^{r_2}} \abs{\nabla_2 f(x'_j\cdot\psi(k');k')}d\mu_H(y)d\mu_H(x)
\end{equation}
in terms of $\int_{z\in G^{k'}_t} \abs{\nabla_2 f(z)} d\mu_G$ using Lemma \ref{lem:memory}. Firstly, by Tonelli's theorem, ($\ref{bound1}$) equals
\begin{equation}\label{eq:splitint}
\int_{(\exp(r_{\sigma^j_1})a_{\sigma^j_1},\exp(-r_{\sigma^j_2})b_{\sigma^j_2},c_{\sigma^j_3})}\int_{(\exp(r_{\tau^j_1})a_{\tau^j_1},\exp(-r_{\tau^j_2})b_{\tau^j_2},c_{\tau^j_3})} \abs{\nabla_2 f(x'_j\cdot\psi(k');k')}
\end{equation}
where $\tau^j_i$ is chosen so that $\{\sigma^j_i,\tau^j_i\}=\{1,2\}$.

	Although the integrand is what we are looking for, the problem is that $x'_j \in H_t$ depends on the endpoints $x$ and $y$ in a complicated way. To work this out, we now perform a change of variables (in three steps) which will fix the variables of the first integral and replace the second three by the image of $(x'_j\cdot\psi(k');k')$ in $H^{k'}_t$, i.e.\ the second integral will now be with respect to $d\mu_H$ as desired.

	The first change of variables is the natural rescaling which fixes $c_1,c_2$, and maps $\exp(r_k)a_{k}\to a_{k}$ and $\exp(-r_{k})b_{k}\to b_{k}$ for $k=1,2$; since $\mu_H$ is Lebesgue measure with respect to these coordinates, this rescaling preserves the measure.

The second step fixes $a_{\sigma^j_1}, b_{\sigma^j_2}, c_{\sigma^j_3}$ and replaces $(a_{\tau^j_1}, b_{\tau^j_2}, c_{\tau^j_3})$ by $x'_j$.

Finally, we return $(a_{\sigma^j_1}, b_{\sigma^j_2}, c_{\sigma^j_3})$ to $(\exp(r_{\sigma^j_1})a_{\sigma^j_1},\exp(-r_{\sigma^j_2})b_{\sigma^j_2},c_{\sigma^j_3})$ and map $x'_j=(\alpha,\beta,\gamma)\to (\exp(k')\alpha,\exp(-k')\beta,\gamma)$.

In each case we treat each of $d_i-\overline{d_i}$ for $d\in\{a,b,c\}$ and $i=1,2$ as a constant. In reality, each takes one of a finite number of values, so we may split the domain of the integral dependent on those values so that they are truly constants.

\textbf{Claim:} There exist constants $0<K<L$ such that the determinant of the Jacobian $J_j$ corresponding to any such three-step change of variables is between $K\exp(r_{\sigma^j_1}-r_{\sigma^j_2})$ and $L\exp(r_{\sigma^j_1}-r_{\sigma^j_2})$.
\begin{proof}[Proof of Claim:] We denote the Jacobian matrix of the $i$th change of variables by $\det J^i_j$. The first change of variables clearly has Jacobian determinant $1$, and third transformation has Jacobian determinant $\exp(r_{\sigma^j_1})\exp(-r_{\sigma^j_2})$.

We must compute the finitely many Jacobian determinants corresponding to the change of variable 
\[(a_{\sigma^i_1},b_{\sigma^i_2},c_{\sigma^i_3},a_{\tau^i_1},b_{\tau^i_2},c_{\tau^i_3})\to (a_{\sigma^i_1},b_{\sigma^i_2},c_{\sigma^i_3},\alpha,\beta,\gamma).
\]
For $i\neq 21,22,23$ it is straightforward to determine that these Jacobian determinants are equal to $1$. 
	For example, in the case $i=20$, we fix $(a_1,b_1,c_2)$ and replace $(a_2,b_2,c_1)$ by $(\overline{a_2}, \overline{b_2}, \overline{c_1}+\overline{a_1}(\overline{b_1}-b_1)+\overline{a_2}(\overline{b_2}-\overline{b_1}))$.
	The Jacobian for this transformation is
	\[ J^2_{20} = \begin{pmatrix} 1&0&0 \\ 0&1&0 \\ \overline{b_2}-\overline{b_1}&\overline{a_2}&1 \end{pmatrix}, \] which has determinant $1$.

	For $i=21,22,23$, the important partial derivatives are
\[
	\frac{\partial k}{\partial a_2}=-\frac{b_2-\overline{b_1}}{2k},\quad \frac{\partial k}{\partial b_2}=-\frac{\overline{a_2}}{2k}, \quad \frac{\partial k}{\partial c_1}=\frac{-1}{2k}, \quad \frac{\partial k}{\partial c_2}=\frac{1}{2k}.
\]
Applying these, we get
\[
	J^2_{21} = \left(\begin{array}{ccc} 1 + \frac{b_2-\overline{b_1}}{2k} & \frac{\overline{a_2}}{2k} & \frac{1}{2k} \\ 0 & 1 & 0 \\ \overline{b_2}-\overline{b_1} & \overline{a_2} & 1 
	\end{array}\right),
	\]\vspace{2mm}
	\[J^2_{22} = \left(\begin{array}{ccc} 1 + \frac{b_2-\overline{b_1}}{2k} & \frac{\overline{a_2}}{2k} & -\frac{1}{2k} \\ \frac{b_2-\overline{b_1}}{2k} & 1 + \frac{\overline{a_2}}{2k} & -\frac{1}{2k} \\ -(b_2-\overline{b_2})-k+\frac{\overline{a_2}(b_2-\overline{b_1})}{2k} & \frac{\overline{a_2}^2}{2k} & 1-\frac{\overline{a_2}}{2k}
 \end{array}\right), \]\vspace{2mm}
 \[
	 J^2_{23} = \left(\begin{array}{ccc} 1 & 0 & 0 \\ + \frac{b_2-\overline{b_1}}{2k} & 1 + \frac{\overline{a_2}}{2k} & -\frac{1}{2k} \\ -(b_2-\overline{b_2})-k+\frac{\overline{a_2}(b_2-\overline{b_1})}{2k} & \frac{\overline{a_2}^2}{2k} & 1-\frac{\overline{a_2}}{2k} 
 \end{array}\right)
\]
which have determinants $1+ \frac{b_2-\overline{b_2}}{2k}$, $\frac12+ \frac{\overline{b_2}-\overline{b_1}}{2k}$ and $1$ respectively.

	Since $\frac{1}{\sqrt 2}\exp(t)\leq k\leq \exp(t)$, and $-\exp(t)\leq b_2-\overline{b_2}\leq \frac{2}{3} \exp(t)$, we have $0<1 - \frac{1}{\sqrt 2} \leq \det J^2_{21} \leq 1+\frac{\sqrt 2}{3}$.

Next, $|\overline{b_2}-\overline{b_1}|\leq \frac13\exp(t)$, so
$0 < \frac12 -\frac{1}{3\sqrt 2} \leq \det J^2_{22} \leq \frac12 + \frac{1}{3\sqrt 2}.$
\end{proof}

For fixed $(\exp(r_{\sigma^j_1})a_{\sigma^j_1},\exp(-r_{\sigma^j_2})b_{\sigma^j_2},c_{\sigma^j_3})$, the change of variables in the other three coordinates is (at most $M$)-to-one by Lemma \ref{lem:memory}. Thus, $(\ref{eq:splitint})$ is bounded from above by
\begin{align*}
	& \left(\int_{(\exp(r_{\sigma^j_1})a_{\sigma^j_1},\exp(-r_{\sigma^j_2})b_{\sigma^j_2},c_{\sigma^j_3})} d\mu_H\right)\left( M \int_{z'\in H^{k'}_t} \frac{1}{\det J_j}|\nabla_2 f(z';k')| d\mu_H\right)
	\\ & \leq 16e^{4t}\frac{M}{K} \sum_{k'=-t}^t\int_{z'\in H^{k'}_t} |\nabla_2 f(z';k')| d\mu_H,
\end{align*}
 where we used that $\int_{(\exp(r_{\sigma^j_1})a_{\sigma^j_1},\exp(-r_{\sigma^j_2})b_{\sigma^j_2},c_{\sigma^j_3})} d\mu_H = 16e^{4t}\exp(r_{\sigma^j_1}-r_{\sigma^j_2})$, and that the Claim gives $\exp(r_{\sigma^j_1}-r_{\sigma^j_2}) \frac{1}{\det J_j}\leq \frac1K$.

Thus, we can bound $\int_{x,y\in G_t} \abs{f(x)-f(y)} d\mu_G^{2}(x,y)$ by
\[
  \sum_{r_1, r_2=-t}^t \sum_{j=0}^{54} 16e^{4t}\frac{M}{K} \sum_{k'=-t}^t\int_{z'\in H^{k'}_t} |\nabla_2 f(z';k')| d\mu_H
\]
Finally, $\sum_{k'=-t}^t\int_{z'\in H^{k'}_t} |\nabla_2 f(z';k')| d\mu_H=\int_{z\in G_t} |\nabla_2 f(z)| d\mu_G$. Combining these, we see that there is constant $C$ such that
\[
 \int_{x,y\in G_t} \abs{f(x)-f(y)} d\mu_G^{2}(x,y)\leq Ct^2e^{4t} \int_{z\in G_t} |\nabla_2 f(z)| d\mu_G,
\]
as required.
\end{proof}

We are now ready to complete the proof.

\begin{proof}[Proof of Theorem \ref{thm:HeisExt}]
	By Corollary~\ref{cor:general-upper-bound} it suffices to prove that $\Lambda^1_{\Heis_3\rtimes\Z}(r)\gtrsim r/\log(r)$.
	
	Let $f:G_t\to\R$ be a non-constant function. Since $\mu_G(G_t)\simeq te^{4t}$, by Lemma \ref{lem:memory2}
\[
 \mu_G(G_t)\norm{f-f_{G_t}}_1 \lesssim \int_{x,y\in G_t} \abs{f(x)-f(y)}d\mu_G^2(x,y) \lesssim t\mu_G(G_t)\norm{\nabla_2 f}_1.
\]
Thus $\Lambda^1_G(\mu_G(G_t))\gtrsim \mu_G(G_t)/\log(\mu_G(G_t))$ as functions of $t$. Here we are using the Poincar\'e profile as defined in \cite{HumeMackTess} rather than the version for graphs stated in the introduction. Since $\mu_G(G_t)$ grows at most exponentially in $t$ we have $\Lambda^1_G(r)\gtrsim r/\log(r)$.
\end{proof}

\section{Capacity profiles}\label{sec:capacityProfile}

The main goal of   \S \ref{sec:capacityProfile} and \S \ref{sec:product-spaces} is to compute the Poincar\'e profiles of spaces such as $P\times \HH_\bbK^m$, where $P$ is a connected Lie group of polynomial growth, $m\geq 2$, and $\bbK\in \{\R,\C,\HH,\mathbb{O}\}.$   In our previous work \cite{HumeMackTess}, we were able to compute the Poincaré profiles of $\HH_\bbK^m$ and $P$ separately. Since Poincaré profiles do not behave especially well under direct product, we cannot simply apply these calculations to our problem. As usual, upper bounds and lower bounds on Poincaré profiles involve radically different ideas and strategies. In this section we will be entirely concerned by upper bounds (lower bounds will be obtained in  \S \ref{sec:product-spaces}).
Our strategy to obtain upper bounds relies on introducing new invariants: {\bf (weighted) capacity profiles}. Let us start by briefly explaining the problem and outlining its solution.

Recall that for the $L^p$-Poincar\'e constant of a finite graph $\Gamma$ we minimize over nonconstant functions $f:V\Gamma\to\R$ the ratio $\|\nabla f\|_p/\|f-f_\Gamma\|_p$, where $f_\Gamma$ is the average value of $f$ on $\Gamma$.  
To find functions on $\Gamma\subset X\times Y$ it is natural to pull back functions on the projected graphs in $X$ and $Y$, where the image subgraphs are weighted by the number of vertices in the fibre over each point. 
A problem arises in that it is difficult to relate the Poincar\'e constant of $\Gamma\subset X\times Y$ to a `weighted Poincar\'e constant' in a projection.  
However, all works much better if we restrict the functions $f$ we consider to those which satisfy $f\leq 0$ and $f\geq 1$ on substantial proportions of $\Gamma$, and just minimize $\|\nabla f\|_p$ among such functions.
This resulting `$L^p$-capacity' constant and profile are by construction at least as large as the $L^p$-Poincar\'e constants and profiles, and are amenable to finding good upper bounds.

We define the capacity profile in \S\ref{ssec:cap-defs} and compute it for trees in \S\ref{ssec:capactree}.
In \S\ref{ssec:cap-Ahl-reg-pf} we study it for Gromov hyperbolic spaces, getting new bounds for Poincar\'e profiles along the way, and use it to find product graph bounds in \S\ref{ssec:cap-weightedANdim}. 

\medskip
We use the following notation.
For quantities $A, B$ we write $A \vee B := \max\{A,B\}$ and $A\wedge B:=\min\{A,B\}$.
We write $A \preceq B$ if there exists $C>0$ with $A \leq CB$, and write $A \asymp B$ if $A \preceq B$ and $B \preceq A$.
For a graph $\Gamma$ and $f:\Gamma\to\R$ we have 
$\|f\|_p=\|f\|_{\Gamma,p}:=\left(\sum_{x\in V\Gamma}|f(x)|^p\right)^{1/p}$,
and
for $x\in V\Gamma$ we have 
$|\nabla f|(x) := \max\{|f(x)-f(x')| : xx'\in E\Gamma\}$,
which lets us define $\|\nabla f\|_p:=\|\nabla f\|_{\Gamma,p}$.
Note that if $\Gamma$ has degree bounded by $d$ then
$\|\nabla f\|_{\Gamma,p}\asymp_d \left(\sum_{xx'\in E\Gamma}|f(x)-f(x')|^p\right)^{1/p}$.

\subsection{Definitions and basic properties}\label{ssec:cap-defs}
In this section we will define the $L^p$-capacity profiles of weighted and unweighted graphs.
\begin{definition}
	\label{def:weighted-graph}
	A \textbf{weighted graph} is a (finite) graph $\Gamma$ with a non-zero measure $\mu=\mu_\Gamma$ on $V\Gamma$, i.e.\ a function $\mu:V\Gamma\to[0,\infty)$ extended to subsets $A \subset V\Gamma$ by $\mu(A)=\sum_{x\in A}\mu(x)$. We define $\|\mu\|_\infty := \max_{x\in V\Gamma} \mu(x)$.

	For any function $f:\Gamma\to\R$ we define $\|f\|_{\mu,p}:=\left( \sum_{x\in V\Gamma} |f(x)|^p \mu(x)\right)^{1/p}$.
\end{definition}

\begin{definition}\label{def:p-cap}
  Let $(\Gamma,\mu)$ be a weighted graph.
	For each $p\in[1,\infty)$, $\alpha \in (0,1/4)$, we define the $(p,\alpha)$\textbf{-capacity} of $\Gamma$ to be 
  \begin{multline*}
	  C^{p,\alpha}(\Gamma,\mu)=\inf \big\{  \mu(\Gamma)^{-1/p} \norm{\nabla f}_{\mu,p} \,:\, f:V\Gamma\to\R \\ \text{ and } \mu(\{f\leq 0\}),\mu(\{f\geq 1\}) \geq \alpha\mu(\Gamma) \big\},
  \end{multline*}
	where $\{f\leq 0\}$ is short for $\{x \in V\Gamma : f(x)\leq 0\}$.
\end{definition}

\begin{definition}
  \label{def:cap-profile}
  We let $X$ be a connected graph.
	For $k:\N\to [1,\infty)$ with $k(r)\leq r/10$, and for $\alpha \in  (0,1/4)$, we define 
\[
	\Xi^{p,\alpha,k}_X(r) = \sup \mu(\Gamma) C^{p,\alpha}(\Gamma,\mu),
\]
	where the supremum is taken over all subgraphs $\Gamma$ of $X$ with $|\Gamma|\leq r$ equipped with some weight function $\mu$ so that $\mu(\Gamma) \leq r$ and $\|\mu\|_\infty \leq k(\mu(\Gamma))$.

	If there exists a function $f:\N\to\R_{>0}$ so that for all sufficiently small $\alpha$ we have $\Xi^{p,\alpha,k}_X(r)\simeq_\alpha f(r)$, then we say that the $(L^p,k)$\textbf{-weighted capacity profile} $\Xi^{p,k}_X$ exists and write $\Xi^{p,k}_X(r)\simeq f(r)$.

	Similarly, we define $\Xi^{p,\alpha}_X(r) = \sup \abs{\Gamma}C^{p,\alpha}(\Gamma,\#)$ where $\Gamma$ is a subgraph of $X$ with $\#(\Gamma)=|\Gamma|\leq r$, weighted by the counting measure $\#$ on $\Gamma$.
	
	If there exists a function $f:\N\to\R_{>0}$ so that for all sufficiently small $\alpha$ we have $\Xi^{p,\alpha}_X(r)\simeq_\alpha f(r)$, then we say that the (unweighted) $L^p$\textbf{-capacity profile} $\Xi^p_X$ exists and write $\Xi^{p}_X\simeq f(r)$.
\end{definition}

We do not pursue here whether the weighted capacity profile is a quasi-isometric invariant of a graph. Unweighted capacity profiles are monotone regular invariants; the proof -- which follows exactly the same strategy as for Poincar\'e profiles -- is omitted from this paper as it is not needed.

\begin{remark}
	In the definition of $C^{p,\alpha}$, we may as well assume that $f:V\Gamma\to[0,1]$, since replacing $f$ by $(f \vee 0)\wedge 1$ (that is, $\min\{\max \{f,0\},1\}$) only decreases $\|\nabla f\|_{\mu,p}$. Under this assumption, $\norm{\nabla f}_p \leq \mu(\Gamma)^{\frac1p}$, so $C^{p,\alpha}(\Gamma,\mu)\leq 1$ for every weighted graph $(\Gamma,\mu)$ and every $p$.
\end{remark}
\begin{remark}\label{rmk:unif-weight-bounds}
  In all our examples below the (weighted) capacity profiles $\Xi^p_X$ exist.
	Moreover, we aim to find bounds for weighted profiles that are uniform in the following sense: for a given $p$, we find a function $f_p:[1,\infty)\times\N\to\R_{>0}$ so that for all sufficiently small $\alpha$,
	for all functions $k=k(r)$ as in the definition, 
	$\Xi^{p,\alpha,k}_X(r) \simeq_{\alpha} f_p(k(r),r)$,
	where the constant of $\simeq$ does not depend on $k$.
	We then write this bound in short as $\Xi^{p,k}_X(r) \simeq_\alpha f_p(k,r)$.
\end{remark}

\begin{remark}One may also define capacity profiles for $p=\infty$, but it follows immediately from the proof of \cite[Proposition 6.1]{HumeMackTess} that $\Xi^\infty_X\simeq \Lambda^\infty_X$. Similarly, arguing as in \cite[Proposition 7.2]{HumeMackTess} (in fact in the proof we may directly define $f(z)=g^{q/p}(z)$) we can deduce that whenever $1\leq p \leq q<\infty$ and the functions are defined, we have
\[
 \Xi^p_X(r)\lesssim_{p,q} \Xi^q_X(r).
\]
	In what follows we work only with $p<\infty$.
\end{remark}

Poincar\'e, capacity and weighted capacity profiles are related by the following two simple observations. 
Firstly, we compare Poincar\'e and capacity profiles.
\begin{lemma}
  \label{lem:compare-cap-poinc-profiles}
  	Let $X$ be a graph.
	For all $\alpha\in (0,\frac14)$, $\Lambda^p_X \lesssim_\alpha \Xi^{p,\alpha}_X$.  So when $\Xi^p_X$ is defined, we have $\Lambda^p_X \lesssim \Xi^p_X$.
\end{lemma}
\begin{proof}
  If $\Gamma$ is a subgraph of $X$, and we have $f:V\Gamma\to[0,1]$ with $|\{f\leq 0\}|, |\{f\geq 1\}| \geq \alpha |\Gamma|$, then
  $\|f-f_\Gamma\|_{p} \succeq |\Gamma|^{1/p}$:
	if the mean value $f_\Gamma$ satisfies $f_\Gamma\geq \frac{1}{2}$ then $|f-f_\Gamma| \geq \frac{1}{2}$ on $\{f\leq 0\}$ and so $\|f-f_\Gamma\|_p \geq \frac{\alpha^{1/p}}{2} |\Gamma|^{1/p}$, and if $f_\Gamma \leq \frac{1}{2}$ the same bound holds on considering $\{f \geq 1\}$.
  So we have, infimising over all non-constant $f:V\Gamma\to \R$,
  \[
   h^p(\Gamma)=\inf\left\{\frac{\norm{\nabla f}_p}{\norm{f-f_\Gamma}_p}\right\} \leq \frac{2}{\alpha^{1/p}} C^{p,\alpha}(\Gamma, \#).\qedhere
  \]
\end{proof}

The weighted and unweighted profiles are related by the following.
\begin{lemma}\label{lem:wt-prof-bound} 
	Let $X$ be a graph. For any $k:\N\to[1,\infty)$ with $k=k(r)\leq r/10$, we have
	$k \Xi^{p,\alpha}_X(r/k) \lesssim_\alpha \Xi^{p,\alpha,k}_X(r)$.
	So, when defined,
	$k \Xi^p_X(r/k) \lesssim_\alpha \Xi^{p,k}_X(r)$.
\end{lemma}
\begin{proof} 
  Let $\Gamma \subset X$ be a subgraph of size $\leq r/k$ so that $|\Gamma| C^{p,\alpha}(\Gamma, \#) \asymp \Xi^{p,\alpha}_X(r/k)$.
  Setting $\mu$ to be $k\#$ where $\#$ is the counting measure on $V\Gamma$, we see that $C^{p,\alpha}(\Gamma,\mu) = C^{p,\alpha}(\Gamma,\#)$, so 
  \[
	  \Xi^{p,\alpha,k}_X(r) \geq \mu(\Gamma) C^{p,\alpha}(\Gamma,\mu) = k |\Gamma| C^{p,\alpha}(\Gamma,\#) \asymp k \Xi^{p,\alpha}_X\left(\frac{r}{k}\right). \qedhere
  \]
\end{proof}

As simple as these bounds are, they prove to be sharp for trees and rank 1 symmetric spaces, as we now proceed to show.

\subsection{Weighted profiles of trees}\label{ssec:capactree}
Our first goal is to adapt the argument of \cite[\S 9]{HumeMackTess} to bound the weighted profiles of trees, perhaps the easiest example.

\begin{proposition}\label{prop:wt-prof-tree}
	For the $3$-regular tree $T$ and any weight function $k=k(r)\leq r/4$, $\Xi^{p,k}_T(r) \simeq k^{\frac1p} r^{1-\frac1p}$.
\end{proposition}
\begin{proof}
  Since $\Lambda^p_T(r)\simeq r^{1-\frac{1}{p}}$ by \cite[Theorem 9]{HumeMackTess}, 
  Lemmas~\ref{lem:compare-cap-poinc-profiles} and \ref{lem:wt-prof-bound} give
  \[
	  \Xi^{p,\alpha,k}_T(r) 
	  \gtrsim_\alpha k\, \Xi^{p,\alpha}_T\left(\frac{r}{k}\right) 
	  \gtrsim_\alpha k\, \Lambda^p_T\left(\frac{r}{k}\right)
	  \simeq k \left(\frac{r}{k}\right)^{1-\frac1p} = k^{\frac1p}r^{1-\frac1p}.
  \]

	To show the upper bound, suppose we have a subgraph $\Gamma \subset T$ with $|\Gamma|\leq r$ and a weight $\mu$ on $\Gamma$ with $\mu(\Gamma)\leq r$ and $\|\mu\|_\infty\leq k=k(\mu(\Gamma))$.
  	One can find a median vertex $v$ of $\Gamma$, i.e.\ if $C_1,C_2,C_3$ denote the connected components of $T\setminus\set{v}$ then for each $i$ we have $\mu(C_i\cap\Gamma)\leq\frac12\mu(\Gamma)$.
	Since $\mu(\{v\}\cap\Gamma) \leq k \leq \frac14\mu(\Gamma)$, there is some $i$ so that $\mu(C_i\cap\Gamma) \geq \frac14\mu(\Gamma)$.
	Let $v' \in C_i$ be the vertex adjacent to $v$ in $C_i$, and set $f$ to be the characteristic function $f=\chi_{C_i}$.
	For any $\alpha \leq \frac14$ we have $\mu\{f\leq 0\},\mu\{f\geq 1\} \geq \alpha|\Gamma|$.
	Thus, as $\| \nabla f \|_{\mu,p}^p \leq \mu(\{v,v'\}\cap \Gamma) \leq 2k$ we have $\Xi^{p,\alpha,k}_T(r) \lesssim_\alpha k^{\frac1p} r^{1-\frac1p}$.
	Hence $\Xi^{p,k}_T$ exists and $\Xi^{p,k}_T(r)\simeq k^{\frac1p}r^{1-\frac1p}$.
\end{proof}

\subsection{Weighted profiles of hyperbolic spaces}\label{ssec:cap-Ahl-reg-pf}

We now consider (Gromov) hyperbolic groups and spaces, with the main goal a general upper bound on weighted capacity profiles (Theorem~\ref{thm:wt-prof-hyp-upper-bound}), adapting the argument of \cite[Theorem 11]{HumeMackTess}.
Our argument here is stronger than that of \cite[Theorem 11]{HumeMackTess} even in the unweighted case ($k \equiv 1$), giving a stronger Poincar\'e profile upper bound as the equivariant conformal dimension is replaced by the usual (Ahlfors regular) conformal dimension, which a priori may be strictly smaller.  As the Poincar\'e profile is a quasi-isometric invariant of a graph \cite[Theorem 1]{HumeMackTess}, if $X$ is quasi-isometric to $G$ then an upper bound on $\Xi_X^{p,\alpha}$ gives an upper bound on $\Lambda_G^p$ by Lemma~\ref{lem:compare-cap-poinc-profiles}: $\Lambda_G^p \simeq \Lambda_X^p \lesssim \Xi_X^{p,\alpha}$.  We take advantage of this by working in a particularly nice graph model: Bourdon--Pajot's hyperbolic cone on the boundary at infinity.

Recall that a metric space $X$ is Ahlfors $Q$-regular if there is a measure on $X$ so that the measure of any ball of radius $r \in (0, \diam X)$ is $\asymp r^Q$. Starting with an Ahlfors $Q$-regular compact space $X$, Bourdon and Pajot construct a hyperbolic graph whose visual boundary is isometric to $X$: the hyperbolic cone of $X$.
Although the automorphism group of this graph may be trivial, it nevertheless has a crucial homogeneity property: the volume of any ball of radius $R$ is $\asymp e^{QR}$ (see Lemma~\ref{lem:caphyp1}). This property is a key ingredient in replacing the ``equivariance'' that was required for the proof of \cite[Theorem 1]{HumeMackTess}.

\subsubsection{Preliminaries on hyperbolic geometry and hyperbolic cones}\label{sec:hypintro}
Experts in hyperbolic geometry may skip to $\S\ref{sec:hypwightedcalc}$. Recall that given three points $p,x,y$ in a metric space $(X,d)$, the Gromov product of $x$ and $y$ at $p$ is given by 
\begin{equation}\label{eq:Gromprod}
 (x|y)_p:=\frac12\left(d(p,x)+d(p,y)-d(x,y)\right).
\end{equation}
Note that $0 \leq (x|y)_p \leq d(p,x) \wedge d(p,y)$ by the triangle inequality. A metric space $X$ is \textbf{$\delta$-hyperbolic} if, for any $p,x,y,z\in X$
\begin{equation}\label{eq:Gromprodhyp}
 (x|z)_p \geq (x|y)_p \wedge (y|z)_p-\delta.
\end{equation}
In a $\delta$-hyperbolic geodesic metric space $X$, given any geodesics $\gamma,\gamma'$ with common start point $p$ and end points $x$ and $y$ respectively, we have that for all $t\leq (x|y)_p$, $d(\gamma(t),\gamma'(t))\leq 2\delta$.

To each proper geodesic hyperbolic space $X$ there is an associated boundary at infinity $\bdry X$, which is a compact space, with a family of visual metrics that are pairwise `quasisymmetric'.  If $X$ admits a geometric group action then the visual metrics are Ahlfors regular.  The Gromov product can be extended to $\overline{X}=X\cup\bdry X$ by setting
	\[
	 (x\mid y)_p :=\sup\liminf_{i,j\to\infty}(x_i|y_j)_p
	\]
	where the supremum is taken over all sequences $(x_i)$ and $(y_j)$ in $X$ with $x=\lim x_i$ and $y=\lim y_j$. Moreover, $\liminf_{i,j\to\infty}(x_i|y_j)_p \geq (x| y)_p-2\delta$ for all such sequences. Finally, given $x,y,z\in \overline X$ and $p\in X$, we have
\begin{equation}\label{eq:Gromprodhypbdry}
 (x|z)_p \geq (x|y)_p \wedge (y|z)_p-2\delta.
\end{equation}

Suppose $(Z,\rho)$ is a compact Ahlfors $Q$-regular metric space with at least two elements, and rescale so that $\diam Z = 1/2$.
Following Bourdon--Pajot~\cite[\S 2.1]{BP-03-lp-besov}	we define a \textbf{hyperbolic cone} on $Z$ to be a graph $X$, with vertex set $VX = \bigsqcup_{t\in\N} X_t$ where each $X_t$ is a maximal $e^{-t}$-separated net in $Z$, and with an edge connecting $z \in X_t$ to $w \in X_u$ if and only if $|t-u|\leq 1$ and $B_Z(z,e^{-t}) \cap B_Z(w,e^{-u}) \neq \emptyset$.
Each $x \in X_t \subset X$ corresponds to a ball $B_Z(x,e^{-t})$ in $Z$.

	The graph $X$ is hyperbolic, and $Z=\bdry X$. Moreover,
	\begin{equation}\label{eq:BPbdrymetric}
	\rho(x,y) \asymp e^{-(x|y)_o}\asymp \diam(A \cup B) 
	\end{equation}
	holds for all $x,y\in VX$ corresponding to balls $A,B$ in $Z$~\cite[Proposition 2.1, Lemma 2.2, Corollary 2.4]{BP-03-lp-besov}, where $o$ is the vertex in $X_0 \subset X$. If $(Z,\rho)$ is the boundary of a proper visual hyperbolic space $Y$, then any hyperbolic cone $X$ of $Z$ is quasi-isometric to $Y$ \cite{BS-00-gro-hyp-embed}.
	
	\subsubsection{Calculating weighted profiles of hyperbolic spaces}\label{sec:hypwightedcalc}
The main theorem of this section \ref{ssec:cap-Ahl-reg-pf} is:

\begin{theorem}\label{prop:wt-prof-hyp-cone}
	For $X$ a hyperbolic cone on a compact Ahlfors $Q$-regular space $Z$ with $Q>0$, for any $\alpha>0$ small enough and any $k=k(r)\leq r/10$ and $p\geq 1$,
\begin{equation*}
	\Xi^{p,\alpha,k}_{X}(r) \lesssim \left\{
 \begin{array}{lcl}
k \left(\frac{r}{k}\right)^{1-\frac{1}{Q}} & \textup{if} & 1\leq p < Q
 \\
k \left(\frac{r}{k}\right)^{1-\frac{1}{Q}} \log^{\frac{1}{Q}}\left(\frac{r}{k}\right) & \textup{if} & p = Q
 \\
k \left(\frac{r}{k}\right)^{1-\frac{1}{p}} & \textup{if} & p>Q.
 \end{array}\right.
\end{equation*}
\end{theorem}

Using this theorem, we obtain the following analogue of \cite[Theorem 11]{HumeMackTess} for weighted capacity profiles.

\begin{theorem}\label{thm:wt-prof-hyp-upper-bound}
Let $G$ be a finitely generated hyperbolic group with conformal dimension $Q \geq 1$.
	Then for every $\varepsilon>0$, there exists a graph $X$ quasi-isometric to $G$, so that for any $k=k(r) \leq r/10$ and any $\alpha>0$ small enough,
\[
\Xi^{p,\alpha,k}_{X}(r) \lesssim \left\{
 \begin{array}{lcl}
   k \left(\frac{r}{k}\right)^{1-\frac{1}{Q}+\epsilon} & \textup{if} & p \leq Q+\epsilon
 \\
 k \left(\frac{r}{k}\right)^{1-\frac{1}{p}} & \textup{if} & p>Q+\epsilon.
 \end{array}\right.
\]
If the conformal dimension is attained (see discussion following Theorem~\ref{thmIntro:profilesDirectProductLie}), there exists a graph $X$ quasi-isometric to $G$ so that: 
\begin{equation*}
	\Xi^{p,\alpha,k}_{X}(r) \lesssim \left\{
 \begin{array}{lcl}
k \left(\frac{r}{k}\right)^{1-\frac{1}{Q}} & \textup{if} & 1\leq p < Q
 \\
k \left(\frac{r}{k}\right)^{1-\frac{1}{Q}} \log^{\frac{1}{Q}}\left(\frac{r}{k}\right) & \textup{if} & p = Q
 \\
k \left(\frac{r}{k}\right)^{1-\frac{1}{p}} & \textup{if} & p>Q.
 \end{array}\right.
\end{equation*}
\end{theorem}
\begin{proof}[{Proof of Theorem~\ref{thm:wt-prof-hyp-upper-bound}}]
If $\bdry G$ attains its conformal dimension of $Q$, let $(Z,\rho)$ be an Ahlfors $Q$-regular metric space, quasisymmetric to $\bdry G$; without loss of generality, $\diam Z = 1/2$.
	Let $X$ be a hyperbolic cone on $Z$ as above, then the needed bounds on $\Xi^{p,\alpha,k}_X$ follow from Theorem~\ref{prop:wt-prof-hyp-cone}.

	If the conformal dimension of $\bdry G$ is not attained, for any $\epsilon>0$ we can find $Q'>Q$ sufficiently close to $Q$, so that the bounds of Theorem~\ref{prop:wt-prof-hyp-cone} for a hyperbolic cone on an Ahlfors $Q'$-regular space quasisymmetric to $\bdry G$ satisfy the necessary estimates.
\end{proof}
We also get general bounds on Poincar\'e profiles of hyperbolic cones.
\begin{theorem}
	\label{thmIntro:Ahlreg}
	Let $Z$ be an Ahlfors regular compact metric space of conformal dimension $Q\geq 1$ and let $X$ be the hyperbolic cone over $Z$ in the sense of Bourdon--Pajot. Then
\[
 Q \geq \inf\setcon{p\geq 1}{\Lambda^p_X(r)\lesssim r^{1-1/p}}.
\] 
\end{theorem}
\begin{proof}
Given an Ahlfors regular space $Z$ with conformal dimension $Q \geq 1$, let $X$ be a hyperbolic cone over $Z$.
Suppose $Z'$ is an Ahlfors $Q'$-regular space quasisymmetric to $Z$, with hyperbolic cone $X'$.
	Then $X'$ and $X$ are quasi-isometric, so by Lemmas~\ref{lem:compare-cap-poinc-profiles} and \ref{lem:wt-prof-bound} we have
	\[
		\Lambda_X^p(r) \simeq \Lambda_{X'}^p(r) \lesssim \Xi_{X'}^{p,\alpha}(r) \lesssim \Xi^{p,\alpha,1}_{X'}(r).
	\]
	Thus by Theorem~\ref{prop:wt-prof-hyp-cone}, for any $p>Q'$ we have $\Lambda_X^p(r) \lesssim r^{1-1/p}$; since we can take $Q'$ arbitrarily close to $Q$ we are done.
\end{proof}

	The following proof of Theorem~\ref{prop:wt-prof-hyp-cone} adapts and extends the work of \cite[\S 12]{HumeMackTess}; in that paper we only considered spaces on which $G$ acts geometrically, but here we instead use the hyperbolic cone construction of Bourdon--Pajot (cf.\ $\S\ref{sec:hypintro}$).
	The idea is that, given a weighted subgraph of $X$, one can use a Lipschitz function on the boundary to get a good candidate function for the $p$-capacity.  The argument is mainly elementary though somewhat long due to details given; we suggest the reader skips the proofs of the lemmas on a first reading.

\begin{proof}[{Proof of Theorem~\ref{prop:wt-prof-hyp-cone}}]
	As in $\S\ref{sec:hypintro}$, we suppose $(Z,\rho)$ is a compact Ahlfors $Q$-regular metric space with $Q>0$, and rescale so that $\diam Z = 1/2$. Let $X$ be a hyperbolic cone over $Z$ with hyperbolicity constant $\delta$. 
	First we establish some geometric properties of hyperbolic cones.

	A graph $X$ is $C$-\textbf{visual} with respect to a point $x_0 \in X$ if for any $x\in X$, there is a $C$-quasi-geodesic ray (i.e., a $(C,C)$-quasi-isometric embedding) $\gamma:[0,\infty)\to X$ with $\gamma(0)=x_0$ and $x \in \gamma$ (cf.\ (1) in [\S 12, HMT]).

	\begin{lemma}\label{lem:caphyp-vis}
	There exists $C$ so that for any $x_0\in X$, $X$ is $C$-visual with respect to $x_0$.
	\end{lemma}
	\begin{proof} 
Firstly we prove the lemma for $x_0=o$. By definition, any $x \in VX$ corresponds to a point $z \in X_t\subset Z$.
		For each $s \in \N$ with $s\neq t$ choose $z_s \in X_s$ so that $z \in B_Z(z_s,e^{-s})$; and set $z_t = x$.
		Then $(z_t)$ describes a geodesic ray from $o$ in $X$ which contains $x$ in its image.  Denote this ray by $\gamma_x$.
		So $X$ is $1$-visual with respect to $o$.

	Now consider general $x_0$ and $x$.		Let $d=(x|x_0)_o$ and note that $d(\gamma_x(d),\gamma_{x_0}(d)) \leq 4\delta$ since by \eqref{eq:Gromprodhyp},
	\begin{align*}
		d-\frac{1}{2}d(\gamma_x(d),\gamma_{x_0}(d))
		& = (\gamma_x(d)|\gamma_{x_0}(d))_o
		\\& \geq (\gamma_x(d)|x)_o \wedge (x|x_0)_o \wedge (x_0|\gamma_{x_0}(d))_o -2\delta 
		= d-2\delta.
	\end{align*}

		There are now two cases.

	\textbf{Case 1:} Suppose $d(x,o)>d+2\delta$.
		We concatenate: the subgeodesic of $\gamma_{x_0}$ from $x_0$ to $\gamma_{x_0}(d)$; a geodesic (of length at most $4\delta$) from $\gamma_{x_0}(d)$ to $\gamma_{x}(d)$; and the subray of $\gamma_x$ from $\gamma_{x}(d)$ to $\gamma_x(\infty)$. 
		This is a $(1,8\delta)$-quasi-geodesic ray starting at $x_0$ and containing $x$, as we now show.
		As the ray is $1$-Lipschitz, it suffices to show that for $t,t' \geq 0$ with $d+t' \leq d(o,x_0)$, writing $y=\gamma_x(d+t)$ and $y'=\gamma_{x_0}(d+t')$, that $d(y,y') \geq t+t'-8\delta$.
		Applying \eqref{eq:Gromprodhyp}, we have
		\begin{align*}
			d=(x|x_0)_o & \geq  (x|y)_o \wedge (y|y')_o \wedge (y'|x_0)_o - 2\delta 
			\\ & = \left( (d+t)\wedge d(x,o)\right) \wedge \left( d+\frac{1}{2}(t+t'-d(y,y'))\right) \wedge (d+t') - 2\delta.
		\end{align*}
		From this, either $t\leq 2\delta$, or $t'\leq 2\delta$, or $d\geq d+\frac{1}{2}(t+t'-d(y,y'))-2\delta$, that is $d(y,y')\geq t+t'-4\delta$ as required.  If $t'\leq 2\delta$ then 
		\[
			d(y,y') \geq d(y,\gamma_x(d)) - d(\gamma_x(d),\gamma_{x_0}(d))-t' \geq t-6\delta \geq t+t'-8\delta,
		\]
		and similarly if $t\leq 2\delta$.

	\textbf{Case 2:} Suppose $d(x,o) \leq d+2\delta$.
		Recall by \eqref{eq:BPbdrymetric} that there exists $C_\rho\geq 1$ so that for all $x,y\in Z, \rho(x,y)\leq C_\rho e^{-(x|y)_o}$.
		Choose an integer $T \geq \log(5C_\rho)+2\delta+3 \geq 3$.
		Since $\diam(Z)=\frac{1}{2}$, and considering $x_0$ as a point in $Z$, we can find $y'\in Z$ with $\rho(x_0,y')\geq  \frac{1}{4}$.
		Choose $y\in X_T \subset VX$ so that $y'\in B_Z(y,e^{-T})$, then $\rho(x_0,y) \geq \frac{1}{4}-e^{-T} \geq \frac{1}{5}$.
		Now $e^{-(x_0|y)_o} \geq \frac{1}{C_\rho}\rho(x_0,y) \geq 1/5C_\rho$ so $(x_0|y)_o \leq \log(5C_\rho) \leq T-2\delta = d(o,y)-2\delta$.
		Thus $x_0, y$ satisfy the hypotheses of Case 1, and so there is a $(1,8\delta)$-quasi-geodesic ray $\beta$ from $x_0$ to $\gamma_{x_0}((x_0|y)_o)$, to $\gamma_y((x_0|y)_0)$, then along $\gamma_y$.

		Since $d=(x|x_0)_o \leq d(o,x_0)$ and
		$d(x,\gamma_{x_0}(d)) \leq d(x,\gamma_x(d))+d(\gamma_x(d),\gamma_{x_0}(d)) \leq 6\delta$, and $(x_0|y)_o \leq \log(5C_\rho)$,
		we have that $x$ is within $C':=6\delta + \log(5C_\rho)$ of the geodesic segment of $\gamma_{x_0}$ from $x_0$ to $\gamma_{x_0}((x_0|y)_o)$.  So adding in to $\beta$ a path of length $\leq 2C'$ to $x$ and back, we get our desired $(1,8\delta+2C')$-quasi-geodesic ray.
		\end{proof}

	The graph $X$ has ``volume entropy'' $Q$ in the following sense (cf.\ [(2), \S12, HMT]).
	\begin{lemma}
		\label{lem:caphyp1}
		There exists $C$ so that for every $R>0$ and $x_0\in X$, we have $\left|QR - \log |B(x_0,R)| \, \right| \leq C$.  
	\end{lemma}
	
		We first note the following estimate; see Figure~\ref{fig:caphyp-volcone}.
	\begin{figure}
		\def\svgwidth{.7\textwidth}
\begingroup%
  \makeatletter%
  \providecommand\color[2][]{%
    \errmessage{(Inkscape) Color is used for the text in Inkscape, but the package 'color.sty' is not loaded}%
    \renewcommand\color[2][]{}%
  }%
  \providecommand\transparent[1]{%
    \errmessage{(Inkscape) Transparency is used (non-zero) for the text in Inkscape, but the package 'transparent.sty' is not loaded}%
    \renewcommand\transparent[1]{}%
  }%
  \providecommand\rotatebox[2]{#2}%
  \newcommand*\fsize{\dimexpr\f@size pt\relax}%
  \newcommand*\lineheight[1]{\fontsize{\fsize}{#1\fsize}\selectfont}%
  \ifx\svgwidth\undefined%
    \setlength{\unitlength}{170.07874016bp}%
    \ifx\svgscale\undefined%
      \relax%
    \else%
      \setlength{\unitlength}{\unitlength * \real{\svgscale}}%
    \fi%
  \else%
    \setlength{\unitlength}{\svgwidth}%
  \fi%
  \global\let\svgwidth\undefined%
  \global\let\svgscale\undefined%
  \makeatother%
  \begin{picture}(1,0.43716466)%
    \lineheight{1}%
    \setlength\tabcolsep{0pt}%
    \put(0,0){\includegraphics[width=\unitlength,page=1]{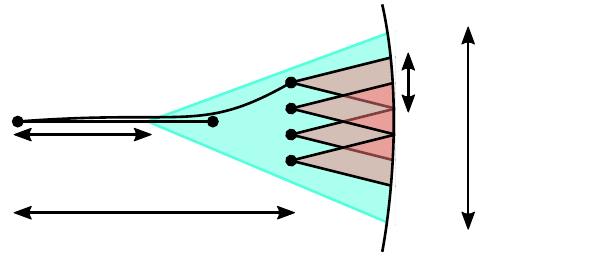}}%
    \put(0.25121122,0.0280086){\color[rgb]{0,0,0}\makebox(0,0)[lt]{\lineheight{1.25}\smash{\begin{tabular}[t]{l}$t$\end{tabular}}}}%
    \put(0.27700359,0.32775225){\color[rgb]{0,0,0}\makebox(0,0)[lt]{\lineheight{1.25}\smash{\begin{tabular}[t]{l}$y$\end{tabular}}}}%
    \put(0.00296278,0.24654731){\color[rgb]{0,0,0}\makebox(0,0)[lt]{\lineheight{1.25}\smash{\begin{tabular}[t]{l}$o$\end{tabular}}}}%
    \put(0.04880041,0.15282391){\color[rgb]{0,0,0}\makebox(0,0)[lt]{\lineheight{1.25}\smash{\begin{tabular}[t]{l}$d(o,y)-C$\end{tabular}}}}%
    \put(0.40888197,0.38355891){\color[rgb]{0,0,0}\makebox(0,0)[lt]{\lineheight{1.25}\smash{\begin{tabular}[t]{l}$z$\end{tabular}}}}%
    \put(0,0){\includegraphics[width=\unitlength,page=2]{cone-AR.pdf}}%
    \put(0.81095424,0.21220968){\color[rgb]{0,0,0}\makebox(0,0)[lt]{\lineheight{1.25}\smash{\begin{tabular}[t]{l}$e^{-d(o,y)+C}$\end{tabular}}}}%
    \put(0.70708741,0.28736996){\color[rgb]{0,0,0}\makebox(0,0)[lt]{\lineheight{1.25}\smash{\begin{tabular}[t]{l}$e^{-t}$\end{tabular}}}}%
  \end{picture}%
\endgroup%

		\caption{Counting points $z\in X_t$ with $(z|y)_o \succeq d(o,y)$}
		\label{fig:caphyp-volcone}
	\end{figure}
	\begin{lemma}\label{lem:caphyp-volcone}
		For any $C \geq 0$ there exists $C' \geq 1$ so that for any $y\in X$ and $t \in \N$ with $t \geq d(o,y)$,
		\[
			\left| \left\{ z \in X_t : (z|y)_o \geq d(o,y)-C \right\}\right|
			\asymp_{C'} e^{Q(t-d(o,y))}.
		\]
	\end{lemma}
	\begin{proof}
		Observe that $(z|y)_o \leq d(o,y)$ always.
		Thus by \eqref{eq:BPbdrymetric} we are counting $z\in X_t$ so that
		\begin{align*}
			\rho(z,y) \leq \diam( B_Z(z,e^{-t}) \cup B_Z(y,e^{-d(o,y)})) & \asymp e^{-(z|y)_o} 
			\asymp e^{-d(o,y)}.
		\end{align*}
		So we are counting a set of $e^{-t}$-separated points in some $B_Z(y, C''e^{-d(o,y)})$: by Ahlfors regularity there are $\preceq e^{-Qd(o,y)}/e^{-Qt}$ of them.
	\end{proof}
	
	\begin{proof}[Proof of Lemma~\ref{lem:caphyp1}]
		It is equivalent to prove that $|B(x_0,R)|\asymp \exp(QR)$.
		We assume $R\in\N$.
	
		By the construction of $X$, any geodesic from $o$ to $x_0$ consists of a sequence of centres of balls $z_t \in X_t$, where $t$ goes from $0$ to $n=d(o,x_0)$ such that $x_0 \in X_n$, so that each $B_Z(z_t, e^{-t}) \cap B_Z(z_{t+1},e^{-(t+1)}) \neq \emptyset$.

		Each $z \in B(x_0,R)$ has $0\leq (z|o)_{x_0}\leq d(x_0,z)\leq R$.
		We partition $B(x_0,R)$ into sets $V_0,V_1,\ldots,V_R$ such that $z \in V_i$ whenever $(z|o)_{x_0}=i$ or $i-\frac{1}{2}$.
		If $z \in V_i$ for some $0 \leq i \leq R$, then $z \in X_{t}$ for some $t$ with 
		\[ d(o,x_0)-i \leq t \leq R+d(o,x_0)-2i +1 \]
		where the first inequality follows from $d(o,x_0)-(z|o)_{x_0}=(z|x_0)_o\leq d(z,o)=t$ and the second from $t+2i -1 \leq d(o,z)+2(z|o)_{x_0} =d(x_0,z)+d(o,x_0)\leq R+d(o,x_0)$.
		Moreover,
		\begin{align*}
			d(o,x_0)-i & = d(o,z_{d(o,x_0)-i})
			\geq (z|z_{d(o,x_0)-i})_o
			\\& \geq  (z|x_0)_o \wedge (x_0|z_{d(o,x_0)-i})_o-\delta
			\\ & \geq ( d(o,x_0)-i )\wedge ( d(o,x_0)-i ) -\delta
			=d(o,x_0)-i-\delta.
		\end{align*}
		Therefore, for given values of $i$ and $t$, by Lemma~\ref{lem:caphyp-volcone} applied when $y=z_{d(o,x_0)-i}$, the number of options for $z$ is $\preceq e^{Q(t-d(o,x_0)+i)}$.  Hence 
		\[
			|B(x_o,R)| =\sum_{i=0}^R |V_i|\preceq \sum_{i=0}^R \sum_{t=d(o,x_0)-i}^{R+d(o,x_0)-2i} e^{Q(t-d(o,x_0)+i)} \preceq \sum_{i=0}^R e^{Q(R-i)} \preceq e^{QR}.
			\]

		On the other hand, by Ahlfors regularity there are $\succeq e^{QR}$ points of $X_{d(o,x_0)+R}$ in $B_Z(x_0, e^{-d(o,x_0)})$, so $|B(x_0,R)| \succeq e^{QR}$ also.
	\end{proof}
	
	Recall from \cite[Definition 12.1]{HumeMackTess} that $A \subseteq X$ is a \textbf{$C$-asymptotic shadow of $x_0 \in X$} if for every $x \in A$ there is a $C$-quasi-geodesic ray $\gamma_x:[0,\infty)\to X$ with $\gamma_x(0)=x_0, \gamma_x(r_x)=x$ for some $r_x$, and $\gamma_x[r_x,\infty) \subseteq A$. 
	In broad terms, the following lemma says that given any weighted subgraph $\Gamma\subset X$, we can find a point $x_0$ and two asymptotic shadows of $x_0$ that are far apart and both containing a substantial part of $\Gamma$.

	\begin{lemma}[{cf.\ \cite[(4), \S12]{HumeMackTess}}] 
		\label{lem:caphyp2}
		There exist (small) $\kappa>0$ and (large) $C>0, R_0 >0$ so that for any $R \geq R_0$ and subgraph $\Gamma \subset X$ weighted by $\mu$ which satisfies $\|\mu\|_\infty \leq \mu(\Gamma)/C$, there exists some $x_0 \in X$ and two $C$-asymptotic shadows $H^\pm \subset X\setminus B(x_0,R)$ of $x_0$ so that $\mu(H^+ \cap \Gamma), \mu(H^- \cap \Gamma) \geq \kappa \mu(\Gamma)-Ce^{QR}$ and so that for any $p^\pm \in H^\pm$ we have $(p^+|p^-)_{x_0} \leq -\log \kappa$. 
	\end{lemma}
	\begin{proof}
		We adapt the proof of \cite[Proposition 12.2 (4)]{HumeMackTess} to deal with the weight $\mu$ and the absence of a group action, and refer to \cite{HumeMackTess} for further details.
		
		By Bonk--Schramm \cite{BS-00-gro-hyp-embed} there exists a quasi-isometric embedding $\psi:X \to \HH_\R^n$ for some $n$.  Push forward $\mu$ to give a measure $\psi_*\mu$ on $\psi(X)$.

		The Helly's Theorem argument of \cite[Lemma 12.8]{HumeMackTess} applies verbatim to give the existence of $c>0$ and $x \in \HH_\R^n$ so that for any half-space $H$ of $\HH_\R^n$ containing $x$ we have $\psi_*\mu(H) \geq c \mu(\Gamma)$.
		By \cite[Lemma 12.10]{HumeMackTess} there is a point $x_0 \in X$ so that $d(\psi(x_0),x)\leq C_1=C_1(\psi,n)$.
		
		The proof of \cite[Lemma 12.9]{HumeMackTess} goes through nearly verbatim to find a constant $\alpha>0$, and a hyperplane $H \subset \HH_\R^n$ through $x$, so that
		$\psi_*\mu( H^\alpha \setminus\{x\}) \leq \frac{c}{2} \mu(\Gamma)$,
		where $H^\alpha$ is the union of all geodesics through $x$ making an angle of $\leq \alpha$ with $H$.
		Note that we have to remove $x$ from $H^\alpha$ to get the volume bound as the polar coordinates in the proof of \cite[Lemma 12.9]{HumeMackTess} degenerate at $x$.
		As $\psi$ is a quasi-isometry, $|\psi^{-1}(\psi(x))| \leq C_2$ for some $C_2$, so if we assume $C\geq 6C_2/c$ then
		$\psi_*\mu(H^\alpha) \leq \frac{c}{2} \mu(\Gamma) + C_2\mu(\Gamma)/C \leq \frac{2c}{3}\mu(\Gamma)$. 

		Let $V^\pm$ be the two components of $\HH_\R^n \setminus H^\alpha$; by the assumptions on $x$, $\psi_*\mu(V^\pm) \geq \frac{c}{3} \mu(\Gamma)$.
		Let $C_3$ be the visual constant of Lemma~\ref{lem:caphyp-vis}.
		By hyperbolicity and the Morse lemma, there exists $C_4=C_4(C_3,\delta)$ so that if $\gamma:[0,\infty)\to X$ is a $C_3$-quasi-geodesic ray with $\gamma(0)=x_0$, then for any $T \geq t \geq 0$ we have $d(\gamma(T),x_0) \geq d(\gamma(t),x_0)-C_4$ and $(\gamma(T)|\gamma(t))_{x_0} \geq d(\gamma(t),x_0)-C_4$.
		Let $\hat H^\pm := \psi^{-1}(V^\pm) \setminus B(x,R+C_4)$.
		Since $X$ is $C_3$-visual, for any $x\in \hat H^+$ there is a $C_3$-quasi-geodesic $\gamma_x$ with $\gamma_x(0)=x_0$ and $\gamma_x(r_x)=x$ for some $r_x$.  Let $H^+$ be the union of $\gamma_x([r_x,\infty))$ for all $x \in \hat H^+$, and likewise for $H^-$.

		By construction $H^\pm$ are $C_3$-asymptotic shadows of $x_0$ in $X\setminus B(x_0,R)$.  
		By the convexity of $V^-,V^+$ and hyperbolicity, there exists $R_0', C_5$ depending on $\alpha, \psi, C_1$ so that for $R \geq R_0'$ such $x^\pm \in \hat H^{\pm}$ must satisfy $(x^+|x^-)_{x_0} \leq C_5$.  If we fix $R_0 > R_0' \vee (2\delta+C_5)$ then for any $R \geq R_0$ we have that for any such $p^{\pm} \in \gamma_{x^\pm}([r_{\pm},\infty)) \subset H^\pm$ we have 
		\begin{align*}
			C_5 & \geq (x^+|x^-)_{x_0} \geq (x^+|p^+)_{x_0} \wedge (p^+|p^-)_{x_0} \wedge (p^-|x^-)_{x_0} -2\delta
			\\ & \geq  (R+C_4-C_4)  \wedge (p^+|p^-)_{x_0} \wedge (R+C_4-C_4)  -2\delta = (p^+|p^-)_{x_0} -2\delta,
		\end{align*}
		thus $(p^+|p^-)_{x_0} \leq C_5+2\delta$.

		Finally, by Lemma~\ref{lem:caphyp1}
		\[
			\mu(H^+) \geq\mu(\hat H^+) \geq \mu(\psi^{-1}(V^+))-\mu(B(x,R+C_4)) \geq \frac{c}{3}\mu(\Gamma) - C_6 e^{QR}
		\]
		for suitable $C_6$, and similarly for $\mu(H^-)$.
		Set $C= (6C_2/c)\vee C_3 \vee C_6$.
	\end{proof}

	If we are given a weighted subgraph $\Gamma \subset X$ and apply the preceeding lemma to find $x_0$ and $H^\pm$, then the following lemma shows that, roughly speaking, either all of $H^-$ or all of $H^+$ must be on the other side of $x_0$ from $o$.  
	\begin{lemma}
		\label{lem:caphyp3}	
			For $C'=-\log \kappa +\delta $ either 
			$\forall x \in H^-, (x| x_0)_o \geq d(o,x_0)-C'$, or
			$\forall x \in H^+, (x| x_0)_o \geq d(o,x_0)-C'$. 
	\end{lemma}
	\begin{proof}
		Indeed, if there exists $x^{\pm} \in H^\pm$ so that $(x^\pm | x_0)_o < d(o,x_0)-C'$
		then as we have $(x^+|x^-)_{x_0}  = d(o,x_0) - (x^+ | x_0)_o - (x^- | x_0)_o + (x^+|x^-)_{o}$ by definition of the Gromov product, by hyperbolicity we see that $(x^+|x^-)_{x_0}$ is at least
		\[
			d(o,x_0) -(x^+ | x_0)_o - (x^- | x_0)_o   + (x^+ | x_0)_o \wedge (x^- | x_0)_o - \delta > C'-\delta,
		\]
		contradicting Lemma \ref{lem:caphyp2}.
	\end{proof}
	
	Given this lemma, without loss of generality we suppose for all $x \in H^-, (x| x_0)_o \geq d(o,x_0)-C'$.

	\begin{lemma}
		\label{lem:caphyp4}
		There exists $\kappa'>0$ depending on $\delta, \kappa$ and the constant of \eqref{eq:BPbdrymetric} so that $\rho(\bdry H^-, \bdry H^+) \geq \kappa' e^{- d(x_0,o)}$.
	\end{lemma}
	\begin{proof}
		Let $C_\rho$ be the constant of \eqref{eq:BPbdrymetric}.
		If $a^\pm \in \bdry H^\pm$ satisfy $\kappa' e^{-d(x_0,o)} > \rho(a^+,a^-) \geq \frac{1}{C_\rho}e^{-(a^+|a^-)_o}$ then $(a^+|a^-)_o > d(x_0,o)-\log(C_\rho\kappa')$.
		There exist sequences $(x_i^\pm) \subset H^\pm$ so that $a^+=\lim x_i^+$ and $a^-=\lim x_i^-$, so as discussed in \S\ref{sec:hypintro} we can bound
		\begin{align*}
			(a^+|a^-)_o 
			& \leq 2\delta+\liminf_{i,j\to\infty}(x_i^+|x_j^+)_o 
			\\ & \leq 2\delta+\liminf_{i,j\to\infty}(x_i^+|x_j^+)_{x_0}+d(x_0,o)
			\leq 2\delta-\log \kappa+d(x_0,o)
		\end{align*}
		by Lemma~\ref{lem:caphyp2}, a contradiction for $\kappa'\leq e^{-2\delta}\kappa/C_\rho$.
	\end{proof}
	Recall from Lemma~\ref{lem:caphyp-vis} that for any $x\in X$ there is a $C$-quasi-geodesic ray $\gamma$ from $o$ through $x$.
	Given another point $x_0$, we now find a geodesic ray from $o$ through $x$ that doesn't go any closer to $x_0$ than it has to.
	\begin{lemma}
		\label{lem:caphyp4b}
		There exists $D>0$ so that given any $o, x_0, x \in X$, there exists a geodesic ray $\gamma$ from $o$ with $d(\gamma,x) \leq D$ and $\eta_x:=\gamma(\infty)$ satisfying 
	$|(x_0| \eta_x)_o - (x_0| x)_o|\leq D$.
	\end{lemma}
	\begin{proof}
		Given $x\in X$, by Lemma~\ref{lem:caphyp-vis} let $\alpha, \beta$ be $C$-quasi-geodesics from $o$ that contain $x_0, x$, respectively.
		
		By the Morse lemma there exists $C_1=C_1(C,\delta)$ so that $|d(x,o)-(x|\beta(\infty))_o|\leq C_1$.
		If $d(x,o) > (x_0|x)_o+2\delta+C_1$ then we can let $\gamma$ be a geodesic representative of $\beta$ since $d(x,\gamma)$ is bounded, and
		\begin{align*}
			(x_0|\beta(\infty))_o 
			& \geq (x_0|x)_o \wedge (x|\beta(\infty))_o - 2\delta 
			\\& \geq (x_0|x)_o \wedge (d(x,o)-C_1) -2\delta
			= (x_0|x)_o-2\delta,
		\end{align*}
		and 
		\begin{align*} 
			(x_0|x)_o 
			& \geq (x_0|\beta(\infty))_o \wedge (\beta(\infty)|x)_o -2\delta
			\\ & \geq (x_0|\beta(\infty))_o \wedge (d(x,o)-C_1) -2\delta
			= (x_0|\beta(\infty))_o-2\delta.
		\end{align*}

		Otherwise, $d(x,o) \leq (x_0|x)_o+2\delta+C_1 \leq d(x,o)+2\delta+C_1$. 
		Since $Z = \bdry X$ is Ahlfors $Q$-regular with $Q>0$, there exists $\eta_x \in \bdry X$ with $\rho(\eta_x,\alpha(\infty))$ comparable to $e^{-d(o,x)}$, and so there exists $C_2$ so that $|(\eta_x|\alpha(\infty))_o - d(o,x)| \leq C_2$.  Let $\gamma$ be a geodesic ray from $o$ to $\gamma(\infty)=\eta_x$; we want to bound $|(x_0|\eta_x)_o-d(x,o)|$.
		As with $x,\beta$ we have $|d(x_0,o)-(x_0|\alpha(\infty))_o|\leq C_1$, so
		\begin{align*}
			(x_0|\eta_x)_o 
			& \geq (x_0|\alpha(\infty))_o \wedge (\alpha(\infty)|\eta_x)_o  -2\delta
			\\ & \geq (d(o,x_0)-C_1) \wedge (d(o,x)-C_2) -2\delta,
		\end{align*}
		but $d(o,x_0)\geq (x|x_0)_o \geq d(x,o) - 2\delta-C_1$, so $(x_0|\eta_x)_o \geq d(o,x)-4\delta-2C_1-C_2$.
		On the other hand
		\begin{align*}
			d(o,x)+C_2 
			& \geq (\alpha(\infty)|\eta_x)_o
			\geq (\alpha(\infty)|x_0)_o \wedge (x_0|\eta_x)_o -2\delta
			\\ & \geq (d(o,x_0)-C_1) \wedge (x_0|\eta_x)_o -2\delta.
		\end{align*}
		So either $(x_0|\eta_x)_o \leq d(o,x)+C_2+2\delta$ and we are done, or $d(o,x_0)-C_1 \leq d(o,x)+C_2+2\delta$.
		But then $(x_0|\eta_x)_o \leq d(o,x_0) \leq d(x,o)+C_2+2\delta+C_1$.  So in summary $|(x_0|\eta_x)_o-d(o,x)| \leq 4\delta+2C_1+C_2$.

		It remains to bound $d(x,\gamma)$.  Let $p \in \gamma$ be the point with $d(o,p)=d(o,x)$.  Then
		\begin{align*}
			d(o,x)-\frac{1}{2}d(x,p)
			= (x|p)_o 
			\geq (x|x_0)_o \wedge (x_0|\eta_x)_o \wedge (\eta_x|p)_o -4\delta,
		\end{align*}
		and so as $(x|x_0)_o\geq d(x,o)-2\delta-C_1$, $(x_0|\eta_x)_o \geq d(x,o)-4\delta-2C_1-C_2$, and $(\eta_x|p)_o\geq d(p,o)-C_1=d(x,o)-C_1$ we are done.
		\end{proof}

 	We define a function $\psi:X\to\R$ by 
	\[
		\psi(x)=3\kappa'^{-1} \rho(\eta_x,\bdry H^-)e^{ d(x_0,o)} - 1,
	\]
	where each $\eta_x\in \bdry X$ is fixed by Lemma~\ref{lem:caphyp4b}.  (The bounds below work regardless of the choices of $\eta_x$.)
	Up until now, no constants have depended on the choice of $R\geq R_0$ in Lemma~\ref{lem:caphyp2}, and we now find a suitable choice of $R$ to ensure $\psi$ is well-behaved.  As a preliminary step we show that $\psi(x)\geq 1$ outside a cone-like `shadow' of $x_0$.

	\begin{lemma}
		\label{lem:caphyp5a}
		There exists $E$ so that for any $x\in \Gamma$ with $(x|x_0)_o\leq d(x_0,o)-E$ then $\psi(x)\geq 1$.
	\end{lemma}
	\begin{proof}
		If $(x|x_0)_o\leq d(x_0,o)-E$ then for any $y\in H^-$, Lemmas~\ref{lem:caphyp4b} and \ref{lem:caphyp3} give
		\begin{align*}
			d(x_0,o)-E+D
			& \geq (x_0|\eta_x)_o \geq (x_0|y)_o \wedge (y|\eta_x)_o -2\delta
			\\ & \geq (d(x_0,o)-C') \wedge (y|\eta_x)_o -2\delta
			= (\eta_x|y)_o-2\delta,
		\end{align*}
		where the last equality follows by assuming $E>D+C'+2\delta$.
		Therefore if $\xi\in\bdry H^-$, we have $(\eta_x|\xi)_o \leq d(x_0,o)-E+C_1$ for some constant $C_1$.
		Writing $C_\rho$ for the constant of \eqref{eq:BPbdrymetric}, we have
		\begin{equation}\label{eq:caphyp5b}
			\rho(\eta_x,\xi)e^{d(x_0,o)} 
			\geq \frac{1}{C_\rho}e^{E-C_1} \geq \frac{2\kappa'}{3},
		\end{equation}
		fixing a choice of $E>C_1+\log(2C_\rho \kappa'/3)$.
		As $\xi\in \bdry H^-$ was arbitrary, $\psi(x)\geq 1$.  
	\end{proof}
	\begin{lemma}
		\label{lem:caphyp5}
		There exists $R \geq R_0$ (independent of $\Gamma,\mu,x_0,H^\pm$) so that 
		for $x \in H^-, \psi(x)\leq 0$, and
		for $x \in H^+, \psi(x)\geq 1$.
	\end{lemma}
	\begin{proof}
		We require the following:

		\emph{Claim:} Suppose some $x\in X$ lies in a $C$-quasi-geodesic ray $\beta$ from $x_0$, and that we have $(x|x_0)_o \geq d(o,x_0)-F$ and $d(x_0,x)\geq R$.  Then there exists $C_1=C_1(C,F)$ and $R_1=R_1(C,F)$ so that if $R \geq R_1$ then $(\beta(\infty)|\eta_x)_o \geq d(o,x_0)+R-C_1$.

		To see this, let $\beta'$ be a geodesic from $o$ to $\beta(\infty)$.
		We first show that $d(x_0,\beta')\leq 2F+12\delta$.  Now,
		\[
			F \geq d(o,x_0)-(x|x_0)_o = (x|o)_{x_0}
			\gtrsim (x|\beta(\infty))_{x_0} \wedge (\beta(\infty)|o)_{x_0} -2\delta.
		\]
		As $x$ lies on the quasi-geodesic $\beta$ from $x_0$, the Morse lemma gives a constant $C_2$ so that $(x|\beta(\infty))_{x_0} \geq d(x,x_0)-C_2 \geq R-C_2$.  Thus if $R\geq R_1:=F+C_2+2\delta+1$, the above inequality gives
		$F \geq (\beta(\infty)|o)_{x_0} -2\delta$.
		Then for $q:= \beta'(d(x_0,o))$,
		\begin{align}\label{eq:caphyp5c}
			d(x_0,o)-\frac{1}{2}d(x_0,q)
			& = (x_0|q)_o
			\geq \liminf_{i\to\infty} (x_0|\beta'(i))_o \wedge (\beta'(i)|q)_o -2\delta.
		\end{align}
		As $q \in \beta'$, $(\beta'(i)|q)_o=d(o,q)=d(o,x_0)$ for large $i$.
		Also,
		\begin{align*}
			\liminf_{i\to\infty} (x_0|\beta'(i))_o 
			& = \liminf_{i\to\infty} \left(d(o,x_0)-(o|\beta'(i))_{x_0} \right)
			\\ & \geq d(o,x_0)-(o|\beta'(\infty))_{x_0} -2\delta
		\geq d(o,x_0)-F-4\delta,
		\end{align*}
		so \eqref{eq:caphyp5c} gives
		$d(x_0,\beta') \leq d(x_0,q) \leq 2F+12\delta$.

		Now as $x$ lies on a quasi-geodesic from $x_0$ to $\beta'(\infty)=\beta(\infty)$, and $x_0$ is a bounded distance from $\beta'$, the Morse lemma gives that $x$ is a bounded distance to $\beta'$.  Thus $|(\beta'(\infty)|x)_o -d(o,x)| \leq C_3$ for suitable $C_3$.
		By Lemma~\ref{lem:caphyp4b}, $x$ also lies a bounded distance from a geodesic from $o$ to $\eta_x$, thus again $|(\eta_x|x)_o-d(x,o)| \leq C_4$.
		So together we have
		\[
			(\beta(\infty)|\eta_x)_o \geq (\beta(\infty)|x)_{o} \wedge (x|\eta_x)_o - 2\delta
			\geq d(x,o)-C_3-C_4-2\delta.
		\]
		The claim follows from Lemma~\ref{lem:caphyp3} as
		\[
			d(o,x)=2(x|x_0)_o-d(o,x_0)+d(x,x_0) \geq d(o,x_0)-2C'+R,
		\]
		setting $C_1:=2C'+C_3+C_4+2\delta$.

		\smallskip
		We return to the proof of the lemma.
		Let $F = C' \vee E$ with $C'$ given by Lemma~\ref{lem:caphyp3} and $E$ by Lemma~\ref{lem:caphyp5a}, and fix the resulting $C_1, R_1$ from the claim above.

		For $x \in H^-$, 
		by the definition of asymptotic shadow, $x \in \beta$ for some $C$-quasi-geodesic $\beta$ from $x_0$ to $\beta(\infty) \in \bdry H^-$,  and $(x|x_0)_o \geq d(o,x_0)-F$ by Lemma~\ref{lem:caphyp3}.
		So the claim gives $(\beta(\infty)|\eta_x)_o \geq d(o,x_0)+R-C_1$, thus writing $C_\rho$ for the constant of \eqref{eq:BPbdrymetric}, 
		\[
			\rho(\eta_x,\bdry H^-) e^{d(x_0,o)}
			\leq C_\rho e^{-(\beta(\infty)|\eta_x)_o}
			\leq C_\rho e^{-R+C_1},
		\]
		so provided $R \geq C_1+\log(3C_\rho/\kappa')$ we have $\psi(x) \leq 0$.

		For $x\in H^+$, if
		$(x|x_0)_o \leq d(x_0,o)-E$ then by Lemma~\ref{lem:caphyp5a} we have $\psi(x)\geq 1$.
		So we assume $(x|x_0)_o \geq d(x_0,o) -E \geq d(x_0,o)-F$.
		As $H^+$ is an asymptotic shadow, $x \in \beta$ for a $C$-quasi-geodesic $\beta$ from $x_0$ to $\beta(\infty) \in \bdry H^+$.
		The claim again gives that
		\[
			\rho(\eta_x, \beta(\infty)) e^{d(x_0,o)}
			\leq C_\rho e^{-R+C_1} \leq 3^{-1}\kappa',
		\]
		where the last inequality uses $R \geq C_1+\log(3C_\rho/\kappa')$.
		Thus by Lemma~\ref{lem:caphyp4},
		$\rho(\eta_x,\bdry H^-) \geq \rho(\beta(\infty),\bdry H^-)-\rho(\beta(\infty),\eta_x) \geq (2/3)\kappa' e^{-d(x_0,o)}$ and $\psi(x)\geq 1$ follows.  Setting $R_0 = R_1 \vee \left( C_1 +\log(3C_\rho/\kappa') \right)$ we are done.
	\end{proof}

	We now set $\phi(x) = (\psi(x)\vee 0)\wedge 1$. By the above, $\phi(x) = 0$ on $H^-$ and $\phi(x) = 1$ on $H^+$, and by Lemma~\ref{lem:caphyp2} both $\mu(H^+\cap \Gamma)$ and $\mu(H^-\cap \Gamma)$ are $\geq \kappa \mu(\Gamma)-Ce^{QR} \geq \frac{\kappa}{2}\mu(\Gamma)$, assuming as we may that $\mu(\Gamma) \geq 2Ce^{QR}/\kappa$.  Let $\alpha = \kappa/2$.

	It remains to bound $\|\nabla \phi \|_\mu$.

	If $x$ has $(x|x_0)_o \leq d(x_0,o)-E-1$ where $E$ is the constant of Lemma~\ref{lem:caphyp5a}, then any neighbour $x'$ of $x$ has $(x'|x_0)_o \leq d(x_0,o)-E$, so by Lemma~\ref{lem:caphyp5a} $\phi(x)=\phi(x')=1$.  Thus $|\nabla \phi|(x)=0$.
	So the support of $|\nabla \phi|$ consists of $x$ with $(x | x_0)_0 \geq d(x_0,o)-E$, i.e.\ it is a subset of the cone-like set $V_{x_0}:= \{x\in X : (x|x_0)_o \geq d(x_0,o)-E\}$.

	We also have the bound $|\nabla \phi|(x) \preceq e^{- (d(x,o)-d(x_0,o))}$ as $\rho(\cdot , \bdry H^-)$ is $1$-Lipschitz on the boundary, and if $x$ and $x'$ are adjacent then $(\eta_x|\eta_{x'})_o \geq d(o,x)-C$ so $\rho(\eta_x,\eta_{x'})\preceq e^{-d(o,x)}$.
	Thus
	\begin{equation}\label{eq:caphyp}
		\|\nabla \phi\|_{\mu,p}^p \preceq \sum_{x \in \Gamma \cap V_{x_0}} e^{- (d(x,o)-d(x_0,o)) p} \mu(x)
	\end{equation}
	
	As in \cite[(12.13)]{HumeMackTess} we can optimise this bound: the right-hand side of \eqref{eq:caphyp} is maximized when the measure $\mu$ is all in $V_{x_0}$ with $d(x,o)$ as small as possible for $x$ in its support.
	For this reason, we choose $t$ minimal so that $V_{x_0}' := \{ x \in V_{x_0} : d(o,x) \leq d(x_0,o)-E+t\}$ has $k |V_{x_0}'| \geq \mu(\Gamma)$.
	By Ahlfors regularity (see Lemma~\ref{lem:caphyp-volcone}), we have $|V_{x_0}'| \asymp e^{Qt}$ so $k e^{Qt} \asymp \mu(\Gamma)$.
	Since $e^{-(d(x,o)-d(x_0,o))p}$ decreases as $d(x,o)$ increases, we have
	\begin{equation*}
		\|\nabla \phi\|_{\mu,p}^p 
		\preceq \sum_{x \in V_{x_0}'} k e^{-(d(x,o)-d(x_0,o)) p }
		 \preceq \sum_{i=0}^t k e^{Qi} e^{-ip}.
	\end{equation*}

	Case 1, $p > Q$: We have $\|\nabla \phi\|_{\mu,p}^p \preceq k$ so $C^{p,\alpha}(\Gamma) \preceq k^{1/p}\mu(\Gamma)^{-1/p}$.

	Case 2, $p<Q$: We have 
	\[
		\|\nabla \phi\|_{\mu,p}^p \preceq k e^{t(Q-p)} \asymp k (\mu(\Gamma)/k)^{(Q-p)/Q} = \mu(\Gamma) (\mu(\Gamma)/k)^{-p/Q},
	\]
	and so $C^{p,\alpha}(\Gamma) \preceq (\mu(\Gamma)/k)^{-1/Q}$.

	Case 3, $p=Q$: We have
	\[
		\|\nabla \phi\|_{\mu,p}^p \preceq k t \asymp k \log(\mu(\Gamma)/k)
	\]
	thus  $C^{p,\alpha}(\Gamma) \preceq \mu(\Gamma)^{-1/p} k^{1/p} \log^{1/p}(\mu(\Gamma)/k)$.

	In each case the bounds of Theorem~\ref{prop:wt-prof-hyp-cone} follow.
\end{proof}

\subsection{Weighted profiles and product bounds}\label{ssec:cap-weightedANdim}

Our motivation for bounding weighted profiles is that they give bounds on (unweighted) profiles of products of groups by projecting onto the factors.
As we later see, in the case that the $X$ factor is a hyperbolic group and the $Y$ factor a group of polynomial growth, the resulting upper bounds are sharp.

\begin{theorem}\label{thm:wt-prof-product-vnilp}
	Let $X$ and $Y$ be bounded degree graphs where $Y$ has finite Assouad--Nagata dimension $d$. Let $\kappa$ be the inverse growth function of $Y$, i.e.\ $\kappa(k)=\min\{t:\exists y\in Y \text{ with }|B(y,t)|>k\}$. Then for some $\alpha_0>0$, for all $\alpha \in (0, \alpha_0]$,
\[
\Lambda_{X\times Y}^p(r) \lesssim
	\Xi_{X\times Y}^{p,\alpha/2}(r) \lesssim \max_{m \leq r} \min_{1 \leq k \leq m/100d} \left( \frac{m}{\kappa(k)} + \Xi_X^{p,\alpha,k}(m) \right).
\]
\end{theorem}
\begin{proof}
  The first inequality follows from Lemma~\ref{lem:compare-cap-poinc-profiles}.

Let $d$ be the Assouad--Nagata dimension of $Y$. Suppose $\Gamma \subset X\times Y$ has $|\Gamma| = m \leq r$. Suppose $1 \leq k \leq m/100d$ is given.

Denote by $\delta_\Gamma$ the counting measure on $\Gamma$.
Let $\Gamma_Y = \pi_{Y}(\Gamma)$, $\mu_Y = (\pi_Y)_* \delta_\Gamma$ be the weighted projection of $\Gamma$ onto $Y$.

	Since $Y$ has Assouad--Nagata dimension at most $d$, there exists $c=c(Y)>0$ so that we can decompose $Y$ as $Y=V_0 \cup \cdots \cup V_d$ where each $V_i$ consists of a disjoint union of $c\kappa(k)$-separated sets each of diameter $\leq \kappa(k)/2$.  Without loss of generality,  $\mu_Y(V_0 \cap \Gamma_Y) \geq m/(d+1)$. We observe for future use that for each subset $A\subset Y$ with $\diam A \leq \kappa(k)/2$ we have $|A|\leq |N_{\kappa(k)/2}A| \leq k$, where $N_C(A)$ denotes the $C$-neighbourhood of $A$.

Consider $V_0 \cap \Gamma_Y$.  
	There are two cases: either (a) one of the diameter-$\kappa(k)/2$ subsets of $V_0$ meets $\Gamma_Y$ with weight $\geq m/4(d+1)$, or (b) condition (a) fails.

	In case (b), we can split $V_0 = V_0' \sqcup V_0''$ where we put the components of $V_0$ into $V_0'$ or $V_0''$ in such a way that $\mu(V_0' \cap \Gamma_Y)$ and $\mu(V_0'' \cap \Gamma_Y)$ are both $\geq \alpha m$ for $\alpha=1/4(d+1)$.  It suffices to prove the theorem for this fixed choice of $\alpha$.

	Define $f:Y\to[0,1]$ by $f(\cdot):=  1 \wedge \frac{1}{c\kappa(k)}d( \cdot, V_0')$.
	Let $F:\Gamma \ra [0,1]$ be the composition $F = f \circ \pi_{Y}$.

Since $0 \leq F \leq 1$, and
	$\delta_\Gamma\{F=0\} \geq \mu_Y (V_0'\cap \Gamma_Y) \geq \alpha m$,
	and $\delta_\Gamma\{F=1\} \geq \mu_Y (V_0''\cap \Gamma_Y) \geq \alpha m$,
 $F$ is a candidate for bounding $C^{p,\alpha}(\Gamma)$.
Since $f$ is $\frac{1}{c\kappa(k)}$-Lipschitz, we have $\|\nabla F\|_{\mu,p}^p \preceq \frac{1}{\kappa(k)^p} m$.
	Therefore  $C^{p,\alpha}(\Gamma) \preceq \frac{1}{\kappa(k)}$ and so $\delta_\Gamma(\Gamma)C^{p,\alpha}(\Gamma) \leq m/\kappa(k)$.

Now suppose we are in case (a).
	Let $U\subset V_0$ be the component set with $\mu_Y(U\cap \Gamma_Y) \geq m/4(d+1)$ and diameter $\leq \kappa(k)/2$.
Consider $\hat{U} = \pi_{Y}^{-1}(U) \subset X\times Y$ and its neighbourhood $\hat{U}' = N_{\kappa(k)/2} \hat{U}$.
	We define a weight function $\nu$ on $\Gamma$ by
	$\nu(\cdot) := 0 \vee (1-d(\hat{U},\cdot)2/\kappa(k))$, and note that $\nu$ is zero outside $\hat{U}'$.
	
	Let $\Gamma_X$ be the projection $\Gamma_X=\pi_X(\hat{U}' \cap \Gamma)$
	with weight $\mu_X = (\pi_X)_* \nu$.
	Observe that as $\delta_\Gamma(\Gamma\cap \hat{U})=\mu_Y(U\cap \Gamma_Y)\geq m/4(d+1)$, and $\nu=\delta_\Gamma$ on $\hat{U}$, we have $\mu_X(\Gamma_X) \geq m/4(d+1)$.
	On the other hand, as $\nu \leq 1$, we have $\mu_X(\Gamma_X) \leq \nu(\Gamma) \leq m$.
	Moreover, as the fibres of $(\pi_X)^{-1}(\cdot) \cap \hat{U}'$ have size at most $k$, $\|\mu_X\|_\infty \leq k$.

	For $\alpha>0$ fixed, 
	let $g: \Gamma_X \ra [0,1]$ be a function with $\mu_X\{g=0\}, \mu_X\{g=1\} \geq \alpha \mu_X(\Gamma_X) \geq \alpha m/4(d+1)$, and with
	$\|\nabla g\|_{\mu_\Gamma,p}^p \preceq m C^{p,\alpha}(\Gamma_X)^p \preceq m^{1-p} \Xi^{p,\alpha,k}_X(m)^p$.

Let $G : \Gamma \ra [0,1]$ be defined by the product
\[
	G(\cdot) =  g \circ \pi_X(\cdot) \left(\frac{2}{\alpha} \nu(\cdot) \wedge 1 \right)^{1+\frac1p}.
\]
We bound
	\[
		\delta_\Gamma\{ G=0 \} \geq \delta_\Gamma\{g\circ\pi_X=0\}
		\geq \nu\{g\circ\pi_X=0\}
		= \mu_X\{g=0\} \geq \alpha m.
	\]
	The bound for $\delta_\Gamma\{G=1\}$ is a little more delicate since $\mu_X\{g=1\}\geq \alpha m$ does not immediately give that $\nu\{ z: g\circ\pi_X(z)\nu(z) =1 \} \geq \alpha m$.
	However, we do get that $\nu\{G = 1\} \geq \frac{\alpha}{2} m$.
	Indeed, let $A = \left\{ g\circ\pi_X = 1\right\} \subset \Gamma$.
	Then 
	\begin{align*}
		\alpha m 
		& \leq \mu_X \left\{ g=1 \right\}
		= \nu (A)
		\\& = \nu \left( A\cap\left\{ \nu < \frac{\alpha}{2} \right\}\right) + \nu \left(A\cap\left\{ \nu \geq \frac{\alpha}{2} \right\}\right)
		\\& < \frac{\alpha}{2} \nu(\Gamma) + \nu \left\{ G = 1 \right\}
		\leq \frac{\alpha}{2} m + \nu \left\{ G = 1 \right\},
	\end{align*}
	so $\nu\left\{ G= 1\right\}\geq\frac{\alpha}{2}m$, and
	\[
		\delta_\Gamma\left\{ G = 1\right\}
		\geq \nu \left\{ G = 1\right\}
		\geq \frac{\alpha}{2} m.
	\]

We now bound $\|\nabla G\|_p^p$.
If $z \in \Gamma$, let 
\[ \|\nabla_X G\|(z) = \max\setcon{|G(z)-G(z')|}{zz'\in E\Gamma, \pi_Y(z)=\pi_Y(z')} \]
and similarly let
\[ \|\nabla_Y G\|(z) = \max\setcon{|G(z)-G(z')|}{zz'\in E\Gamma, \pi_X(z)=\pi_X(z')}. \]
Then $\|\nabla G\|=\| \nabla_X G\| \vee \|\nabla_Y G\|$ and so
$
	\|\nabla G\|_p^p 
\preceq \|\nabla_X G\|_p^p \vee \|\nabla_Y G\|_p^p .
$

If $zz'\in E\Gamma$ and $\pi_Y(z)=\pi_Y(z')$, then $\nu(z)=\nu(z')$ and so, using $\nu(z)^{p+1} \leq \nu(z)$, we can bound:
\begin{align*}
	\|\nabla_X G\|_p^p
	& = \sum_{z \in \Gamma} \max_{\substack{z' : zz'\in E\Gamma, \\ \pi_Y(z)=\pi_Y(z')}} |G(z)-G(z')|^p
	\\& \leq \frac{2^{p+1}}{\alpha^{p+1}} \sum_{z \in \Gamma} \max_{\substack{z' : zz'\in E\Gamma, \\ \pi_Y(z)=\pi_Y(z')}} |g\circ\pi_X(z)-g\circ\pi_X(z')|^p \nu(z)^{p+1}
	\\& \preceq_\alpha 
		\sum_{x \in \Gamma_X} \sum_{z \in \pi_X^{-1}(x)} \nu(z)
		\max_{x' \in \Gamma_X: x\sim x'} |g(x)-g(x')|^p 
	\\& =
		\sum_{x \in \Gamma_X} 
		|\nabla g|(x)^p \mu_X(x)
	\\&
	= \|\nabla g\|_{\mu_X,p}^p 
	\preceq m^{1-p} \, \Xi_X^{p,\alpha,k}(m)^p.
\end{align*}

To bound $\|\nabla_Y G\|_p$, we use Matousek's inequality: if $a,b \geq 0$ and $q\geq 1$ then
\[
	|a^q-b^q| \leq q |a-b| \, (a^{q-1}+b^{q-1}) \leq 2q |a-b| \, (a^{q-1} \vee b^{q-1}).
\]
So, using this with $q = 1+\frac1p$, and the fact that $\nu(\cdot)$ is $2/\kappa(k)$-Lipschitz, we have
\begin{align*}
	\|\nabla_Y G\|_p^p
	& = \sum_{x\in \Gamma_X} \sum_{z\in\pi_X^{-1}(x)}  \max_{\substack{z'\in\pi_X^{-1}(x):\\ zz'\in E\Gamma}} |G(z)-G(z')|^p
	\\ & \leq \frac{2^{p+1}}{\alpha^{p+1}} \sum_{x\in \Gamma_X} 
		\sum_{z\in\pi_X^{-1}(x)}  
		\max_{\substack{z'\in\pi_X^{-1}(x):\\ zz'\in E\Gamma}}
		|g(x)|^p |\nu(z)^{1+\frac1p}-\nu(z')^{1+\frac1p}|^p
	\\ & \preceq_{\alpha,p}
		\sum_{x\in \Gamma_X} 
		|g(x)|^p
		\sum_{z\in\pi_X^{-1}(x)}  
		\max_{\substack{z'\in\pi_X^{-1}(x):\\ zz'\in E\Gamma}}
		 |\nu(z)-\nu(z')|^p \, \left( \nu(z) \vee \nu(z') \right)
	\\ & \preceq \kappa(k)^{-p}
		\sum_{x\in \Gamma_X} |g(x)|^p \nu(\pi_X^{-1}(x))
		= \kappa(k)^{-p} \|g\|_{\mu_X,p}^p
	\\ &	\leq \kappa(k)^{-p} \|1\|_{\mu_X,p}^p
		\leq \kappa(k)^{-p} \mu_X(\Gamma_X)
		\leq \kappa(k)^{-p} m.
\end{align*}

So 
\[
  \|\nabla G\|_p^p \preceq m^{1-p} \Xi^{p,\alpha,k}_X(m)^p \vee m \kappa(k)^{-p}
\]
and thus
\begin{align*}
\mu(\Gamma)C^{p,\alpha/2}(\Gamma) 
	& \preceq m \left(\frac{m^{1-p} \Xi^{p,\alpha,k}_X(m)^p \vee m \kappa(k)^{-p}}{m}\right)^{1/p}
 \\ & \asymp \Xi^{p,\alpha,k}_X(m) \vee \frac{m}{\kappa(k)}.
\end{align*}
The proof is finished by varying $k$ to get the best estimate.
\end{proof}

\subsection{Questions on (weighted) capacitance profiles}

Using the same method as for Poincar\'e profiles \cite[Theorem 1]{HumeMackTess}, it is not difficult to prove that the unweighted capacitance profile is monotone under regular maps. We record a number of questions about this.

\begin{question} Does the $(L^p,k)$-weighted capacity profile $\Xi^{p,k}_X$ exist for every bounded degree graph $X$?
\end{question}

\begin{question} Is the $(L^p,k)$-weighted capacity profile (when it exists) monotone under regular maps?
\end{question}

\begin{question} Is there a bounded degree graph $X$ and a $p\in[1,\infty]$ such that $\Xi^{p}_X(r)\not\simeq \Lambda^p_X(r)$?
\end{question}

\section{Poincar\'e profiles of product spaces}
\label{sec:product-spaces}

In this section we prove upper bounds and  lower bounds for Poincar\'e profiles of certain product spaces. The upper bounds, obtained in \S \ref{sec:ubproducts}, are an application of the the results of \S \ref{sec:capacityProfile}. The lower bounds are proved in \S \ref{sec:lbproducts}, exploiting a general lower bound formula for direct product of spaces. Finally, in \S \ref{sec:beyondProfile},  we show our  non-embedding result Theorem~\ref{thmIntro:directProductCE}. Its proof combines Poincaré calculations with further arguments using techniques from \S \ref{sec:capacityProfile}. 

\subsection{Hyperbolic times polynomial growth: general estimates}\label{subsect:productsStatements}

In the particular case of the product of a hyperbolic group with a virtually nilpotent group, we find the following upper and lower bounds, which generalise the corresponding results for hyperbolic groups themselves in~\cite[Theorem 11]{HumeMackTess}.
Recall that a graph $Y$ has \textbf{polynomial growth of degree $d\geq 0$}, if there exist $C\geq 1$ such that for all $y\in Y$ and $r\geq 1$,
\[C^{-1}r^d\leq |B(y,r)|\leq Cr^d.\]
\begin{theorem}\label{thm:hyp-times-nilp-upper-bound}
	Let $G$ be a finitely generated non-elementary hyperbolic group with (Ahlfors regular) conformal dimension $Q$.
Let $Y$ be a graph of  polynomial growth of degree $d\geq 0$.
 Then for every $\varepsilon>0$,
\[
 \Lambda^p_{G\times Y}(r) \lesssim \left\{
 \begin{array}{lcl}
 r^{1-\frac{1}{Q+d}+\epsilon} & \textup{if} & p \leq Q
 \\
 r^{1-\frac{1}{p+d}} & \textup{if} & p>Q.
 \end{array}\right.
\]
If the conformal dimension of $G$ is attained (see discussion following Theorem~\ref{thmIntro:profilesDirectProductLie}), we have:
\[
 \Lambda^p_{G\times Y}(r) \lesssim \left\{
 \begin{array}{lcl}
 r^{1-\frac{1}{Q+d}} & \textup{if} & 1\leq p < Q
 \\
 r^{1-\frac{1}{Q+d}} \log^{\frac{1}{Q+d}}(r) & \textup{if} & p = Q
 \\
 r^{1-\frac{1}{p+d}} & \textup{if} & p>Q.
 \end{array}\right.
\]
\end{theorem}

These upper bounds are found using weighted projections and capacity estimates as in \S\ref{sec:capacityProfile}.

The following lower bound is stated in the generality used by \cite[Theorems 10.1, 11.1 and 11.3]{HumeMackTess}; as we do not work directly with the notion of $1$-Poincar\'e inequalities in this paper we refer to \cite{HumeMackTess} for definitions and further references.
\begin{theorem}\label{thm:hyp-x-nilp-lower}
	Let $X$ be a visual Gromov hyperbolic graph with a visual metric on its boundary that is Ahlfors $Q$-regular and admits a $1$-Poincar\'e inequality.
	
	Let $P$ be a connected Lie group (or finitely generated group) with polynomial growth of degree $d\geq 0$. Then
\[
 \Lambda^p_{X\times P}(r) \gtrsim \left\{
 \begin{array}{lcl}
 r^{1-\frac{1}{Q+d}} & \textup{if} & 1\leq p < Q
 \\
 r^{1-\frac{1}{Q+d}} \log^{\frac{1}{Q+d}}(r) & \textup{if} & p = Q
 \\
 r^{1-\frac{1}{p+d}} & \textup{if} & p>Q.
 \end{array}\right.
\] 
\end{theorem}
As we shall see, these lower bounds for $X\times N$ follow fairly easily from the lower bounds of the Poincar\'e profiles for $X$ and $P$.

Together, these bounds give sharp results when the hyperbolic group acts geometrically on a rank 1 symmetric space $\HH_\bbK^m$ or 
a Fuchsian building $I_{m,n}$, $m\geq 5$, $n\geq 3$,
 as studied by Bourdon and Bourdon--Pajot \cite{Bourdon,BP-03-lp-besov}. The isometry group of $I_{m,n}$ admits a uniform lattice $G_{m,n}$ generated by generalized reflections, a presentation\footnote{Note that this presentation is that of  a graph product of cyclic groups $\Z/n\Z$ indexed by an $m$-cycle. 
} of which is given by (see \cite{Bourdon}):
 \[G_{m,n}=\langle s_1, \ldots, s_m \,|\, s_1^n, s_2^n, \ldots, s_m^n, [s_1,s_2],[s_2,s_3],\ldots,[s_{m-1},s_m],[s_m,s_1] \rangle.\] 

In the case of a rank 1 symmetric space $\HH_\bbK^m$ the conformal dimension of the boundary is $Q=(m+1)\dim_R\bbK-2$, while for a group $G_{m,n}$ it is $Q=1+\log(n-1)/\arccosh((m-2)/2)$ \cite[Théorème 1.1.]{Bourdon}.
In both cases there is a metric on the boundary that is Ahlfors $Q$-regular and admits a $1$-Poincar\'e inequality, so the following immediate corollary to Theorems~\ref{thm:hyp-times-nilp-upper-bound} and \ref{thm:hyp-x-nilp-lower} generalises \cite[Theorem 12]{HumeMackTess}. 
\begin{corollary}
	\label{cor:product-general}
	Let $H$ be a finitely generated Gromov hyperbolic group that has its conformal dimension $Q \geq 1$ attained by a metric admitting a $1$-Poincar\'e inequality, and let $P$ be a discrete or connected Lie group with polynomial growth of degree $d\geq 0$.
	Then the group $G=H\times P$ has
\begin{equation*}
 \Lambda^p_{G}(r) \simeq \left\{
 \begin{array}{lcl}
 r^{1-\frac{1}{Q+d}} & \textup{if} & 1\leq p < Q
 \\
 r^{1-\frac{1}{Q+d}} \log^{\frac{1}{Q+d}}(r) & \textup{if} & p = Q
 \\
 r^{1-\frac{1}{p+d}} & \textup{if} & Q<p<\infty.
 \end{array}\right.
\end{equation*}
\end{corollary}

\subsection{Upper bounds for direct products}\label{sec:ubproducts} 
The upper bound on Poincar\'e profiles follows from the capacity profile and product space bounds obtained in \S\ref{sec:capacityProfile}.
\begin{proof}[Proof of Theorem~\ref{thm:hyp-times-nilp-upper-bound}]
	Suppose $Q$ is the conformal dimension of $G$, and $d$ the polynomial growth of $Y$.  As $Y$ has (exactly) polynomial growth it is doubling and so has finite Assouad--Nagata dimension, say $d'$.

	Given $1 \leq p \leq Q$ and $\epsilon>0$, choose $\epsilon'=\epsilon'(d,Q,\epsilon)>0$ as described below, and let $X$ be a graph quasi-isometric to $G$ with $\Xi^{p,\alpha,k}_X(r) \lesssim k(r/k)^{1-\frac{1}{Q}+\epsilon'}$ from Theorem~\ref{thm:wt-prof-hyp-upper-bound}.
	Theorem~\ref{thm:wt-prof-product-vnilp} gives
	\[
		\Lambda_{G\times Y}^p(r) \simeq \Lambda_{X\times Y}^p(r)
		\lesssim \max_{m\leq r} \min_{1 \leq k \leq m/100d'}
		\left( \frac{m}{k^{1/d}} + m^{1-\frac1Q+\epsilon'}k^{\frac1Q-\epsilon'}\right).
	\]
	This is optimised for $m=r$ and $k=r^{(1-\epsilon' Q)/(1-\epsilon' Q + \frac{Q}{d})}$, and gives a bound
	\[
		\lesssim r^{1-\frac{1-\epsilon' Q}{Q+d-\epsilon' Q d}} \leq r^{1-\frac{1}{Q+d}+\epsilon},
	\]
	where we can choose $\epsilon' = \epsilon'(d,Q,\epsilon)>0$ so that the last inequality holds.

	For $p>Q$, if $Q \geq 1$ choose $\epsilon'>0$ so that $p>Q+\epsilon'$ and let $X$ be a graph quasi-isometric to $G$ with $\Xi^{p,\alpha,k}_X(r) \lesssim k(r/k)^{1-\frac{1}{p}}$ from Theorem~\ref{thm:wt-prof-hyp-upper-bound}.  If $Q=0$ then $G$ is quasi-isometric to a $3$-regular tree $X$, so Proposition~\ref{prop:wt-prof-tree} again gives $\Xi^{p,\alpha,k}_X(r) \lesssim k^{\frac{1}{p}}r^{1-\frac{1}{p}}$.  In either case, Theorem~\ref{thm:wt-prof-product-vnilp} gives
	\[
		\Lambda_{G\times Y}^p(r)
		\lesssim \max_{m\leq r} \min_{1 \leq k \leq m/100d'}
		\left( \frac{m}{k^{1/d}} + m^{1-\frac1p}k^{\frac1p}\right),
	\]
	which is optimised for $m=r$ and $k=r^{1/(1+\frac{p}{d})}$, giving
	$\Lambda_{G\times Y}^p(r) \preceq r^{1-\frac{1}{p+d}}$.

	If the conformal dimension is attained then $Q\geq 1$ and the bounds for $1\leq p<Q$ and $p>Q$ follow in a similar way to that above.
	For $p=Q$, Theorems \ref{thm:wt-prof-hyp-upper-bound} and \ref{thm:wt-prof-product-vnilp} give
	\[
		\Lambda_{G\times Y}^p(r)
		\lesssim \min_{1 \leq k \leq m/100d'}
		\left( \frac{r}{k^{1/d}} + r^{1-\frac1Q}k^{\frac1Q}\log^{\frac1Q}(\tfrac{r}{k})\right),
	\]
	which is optimised for $k \asymp (r/\log r)^{d/(Q+d)}$,
	giving the desired bound of
	$\Lambda_{G\times Y}^p(r) \preceq r^{1-\frac{1}{Q+d}}\log^{\frac{1}{Q+d}} r$.
\end{proof}

\subsection{Lower bounds for direct products}\label{sec:lbproducts}

Let us now consider the easier lower bounds for products.
This follows from a generalisation of \cite[Theorem 3.2]{BenSchTim-12-separation-graphs} for Poincar\'{e} profiles.
\begin{proposition}\label{prop:product-lower-profile}
	For $X$ and $Y$ infinite graphs, 
	\[
		\Lambda^p_{X \times Y}(r) \gtrsim \max\setcon{\abs{A}\abs{B} (h^p(A)\wedge h^p(B))}{A\subseteq X,\, B\subseteq Y,\, \abs{A}\abs{B}\leq r}.
	\]
	If $\Lambda^p_X$, and likewise $\Lambda^p_Y$, satisfy the property that for any $r$ there exists $A \subset X$ with $|A|\asymp r$ and $\Lambda^p_X(r) \asymp |A|h^p(A)$, then the bound may be written as:
	\[
		\Lambda^p_{X \times Y}(r) \gtrsim \max_{1\leq k \leq r} \left( \frac{r}{k} \Lambda^p_X(k) \right) \wedge \left( k \Lambda^p_Y\left(\frac{r}{k}\right) \right) .
	\]
\end{proposition}
The result follows immediately from the following lemma concerning Poincar\'e constants of products of finite graphs.

\begin{lemma}\label{lem:product-lower-h}
	For every $p\in[1,\infty)$ there exists a constant $c_p$ such that for all finite graphs $A, B$, $h^p(A \times B) \geq c_p (h^p(A)\wedge h^p(B))$.
\end{lemma}
\begin{proof}
	For any $f:A\times B \ra \R$, the triangle and H\"older inequalities give
	\begin{align*}
		& \|f-f_{A\times B}\|^p
		\\ & = \sum_{(x,y)\in A\times B} \bigg|f(x,y)-\frac{1}{|A\times B|} \sum_{(x',y')\in A\times B} f(x',y')\bigg|^p
		\\ & \leq \frac{1}{|A\times B|^p} \sum_{(x,y)\in A\times B} \bigg( \sum_{(x',y')\in A\times B} |f(x,y)-  f(x',y')| \bigg)^p
		\\ & \leq \frac{1}{|A\times B|}\sum_{(x,y)\in A\times B} \sum_{(x',y')\in A\times B} |f(x,y)-f(x',y')|^p .
		\\ \intertext{Elementary inequalities and the definition of $h^p$ applied to fibres $\{x\}\times B$ and $A \times \{y'\}$ show that this is}
		& \preceq_{p} \frac{1}{|A||B|} \sum_{(x,y)\in A\times B} \sum_{(x',y')\in A\times B} 
		\Bigg( \left|f(x,y)-f_{\{x\} \times B} \right|^p+\\& \qquad
			\left|f_{\{x\} \times B} -f(x,y') \right|^p+ 
			\left|f(x,y')-f_{A \times \{y'\}} \right|^p+
			\left|f_{A \times \{y'\}} - f(x',y') \right|^p \Bigg)
			\\ & = 2 \sum_{x\in A} \sum_{y \in B} \left|f(x,y)-f_{\{x\}\times B}\right|^p
				+ 2 \sum_{y'\in B} \sum_{x \in A} \left|f(x,y')-f_{A\times \{y'\}}\right|^p
		\\ & \preceq \sum_{x \in A} h^p(B)^{-p} \sum_{y \in B} |\nabla f(x,y)|^p 
			+ \sum_{y' \in B} h^p(A)^{-p} \sum_{x \in A} |\nabla f(x,y)|^p
			\\ & \preceq \left( h^p(A) \wedge h^p(B) \right)^{-p} \left\| \nabla f \right\|_p^p. \qedhere 
	\end{align*}
\end{proof}

Proposition~\ref{prop:product-lower-profile} has the following consequence, when combined with the upper bound in Theorem~\ref{thm:hyp-times-nilp-upper-bound}.
\begin{corollary}\label{cor:treexpoly-lower}
	Let $T$ be the infinite trivalent tree (quasi-isometric to any non-abelian free group of finite rank), and let $P$ be a discrete or connected Lie group with polynomial growth of degree $d\geq 0$.  Then for $1\leq p < \infty$,
	\[
		\Lambda^p_{T \times P}(r) \simeq r^{1-\frac{1}{d+p}} .
	\]
\end{corollary}
\begin{proof}
Note that $P$ is quasi-isometric to some bounded degree graph $Y$ with the same degree of growth and whose Poincar\'e profile has same asymptotic behavior.  
	Since $T$ is quasi-isometric to the Cayley graph of a free group on two generators, which has conformal dimension $0$, Theorem~\ref{thm:hyp-times-nilp-upper-bound}  shows the upper bound on $\Lambda_{T \times P}$ for any $p\geq 1$.  The lower bound remains to be shown.

	For the tree $T$, \cite[Theorem 10.1]{HumeMackTess} and its proof show that $\Lambda^p_T(k) \simeq k^{1-1/p}$, and is attained by a ball of size $\asymp k$.
	For a group $P$ with polynomial growth of degree $d$, \cite[Theorem 7]{HumeMackTess} and its proof via \cite[Proposition 9.5]{HumeMackTess} show that $\Lambda^p_P(r/k) \simeq (r/k)^{1-1/d}$, and is attained by a subspace of size $\asymp r/k$. 
	So Proposition~\ref{prop:product-lower-profile} gives
	\[
		\Lambda^p_{T\times P}(r) \gtrsim
		\max_{1 \leq k \leq r} 
		\frac{r}{k}k^{1-1/p} \wedge k \left(\frac{r}{k}\right)^{1-1/d}
		= r \max_{1\leq k\leq r} k^{-1/p} \wedge k^{1/d}r^{-1/d},
	\]
	which is optimised for $k\asymp r^{1/(1+d/p)}$.
\end{proof}

Corollary \ref{cor:treexpoly-lower} is the last ingredient needed to complete the proof of Theorem \ref{thmIntro:BS}.

\begin{proof}[Proof of Theorem~\ref{thmIntro:BS}]
	If $\abs{m}=\abs{n}=1$, then $G=\BS(m,n)$ is commensurable to $\Z^2$, so $\Lambda^p_G(r)\simeq r^{1/2}$ for all $p\in[1,\infty)$ by \cite[Theorem 7]{HumeMackTess}. If $\abs{m}=\abs{n}>1$, then $G=\BS(m,n)$ is commensurable to $F_2 \times \Z$, so $\Lambda^p_G(r)\simeq r^{1-\frac{1}{p+1}}$ by Corollary \ref{cor:treexpoly-lower}. Finally, there is a regular map from $\DL(2,2)$ to $\BS(m,n)$ whenever $\abs{m}\neq\abs{n}$ by Theorem \ref{thm:embSgraph}, so $\Lambda^p_G(r)\simeq r/\log(r)$ for all $p\in[1,\infty)$.
By \cite[Theorem 1]{GalJanusz} there is a proper (and hence regular) map $G\to \Aut(T)\times\Aff(\R)$ where $T$ is a bounded valence tree. Since $\Aut(T)$ is quasi-isometric to a $3$-regular tree, and $\Aff(\R)$ has finite Assouad--Nagata dimension by \cite{HP-13-ANdimension-nilppolyc}, we have $\Lambda^p_G(r)\lesssim \Lambda^p_{\Aut(T)\times\Aff(\R)}(r)\lesssim r/\log(r)$ for all $p\in[1,\infty)$.
\end{proof}

We can now also show the lower bound for products of certain hyperbolic groups with groups of polynomial growth.

\begin{proof}[Proof of Theorem~\ref{thm:hyp-x-nilp-lower}]
	By \cite[Theorems 10.1, 11.1 and 11.3]{HumeMackTess} and their proofs, we have
\[
 \Lambda^p_{X}(r) \gtrsim \left\{
 \begin{array}{lcl}
 r^{1-\frac{1}{Q}} & \textup{if} & 1\leq p < Q
 \\
 r^{1-\frac{1}{Q}} \log^{\frac{1}{Q}}(r) & \textup{if} & p = Q
 \\
 r^{1-\frac{1}{p}} & \textup{if} & p>Q.
 \end{array}\right.,
\] 
and in each case the lower bound is attained by a set of size $\asymp r$.
So by Proposition~\ref{prop:product-lower-profile}
	for $p<Q$ we have
	\[
		\Lambda_{X \times N}(r) \gtrsim
		\max_{1 \leq k \leq r} 
		\frac{r}{k}k^{1-1/Q} \wedge k \left(\frac{r}{k}\right)^{1-1/d}
		\asymp r^{1-1/(Q+d)},
	\]
	on taking the optimal $k \asymp r^{1/(1+d/Q)}$.
	For $Q>p$, as in Corollary~\ref{cor:treexpoly-lower}, we have
	\[
		\Lambda_{X\times N}(r) \gtrsim 
		\max_{1 \leq k \leq r} 
		\frac{r}{k}k^{1-1/p} \wedge k \left(\frac{r}{k}\right)^{1-1/d}
		\asymp r^{1-1/(p+d)}.
	\]
For $p=Q$ we have
	\begin{align*}
		\Lambda^Q_{X\times N}(r) 
		& \gtrsim 
		\max_{1 \leq k \leq r} 
		\left( \frac{r}{k}k^{1-1/Q}\log^{1/Q}k\right) \wedge \left( k \left(\frac{r}{k}\right)^{1-1/d} \right) 
		\\ & = \max_{1 \leq k \leq r}\left( rk^{-1/Q}\log^{1/Q}k\right) \wedge \left( r^{1-1/d}k^{1/d}\right).
	\end{align*}
	The optimal value of $k$ is approximately
	$k \asymp r^{Q/(Q+d)} \log^{d/(Q+d)} r$,
	giving
	\[
		\Lambda^Q_{X\times N}(r)
		\gtrsim
		r^{1-1/d} \left(r^{Q/(Q+d)} \log^{d/(Q+d)}r\right)^{1/d}
		= r^{1-\frac{1}{Q+d}} \log^{\frac{1}{Q+d}} r. \qedhere
	\]
\end{proof}

We also have completed the proof of Theorem~\ref{thmIntro:profilesDirectProductLie}.
\begin{proof}[Proof of Theorem~\ref{thmIntro:profilesDirectProductLie}]
	The bounds follow from Corollary~\ref{cor:product-general} in the first two cases, and from Corollary~\ref{cor:treexpoly-lower} in the third case.
\end{proof}

\subsection{Beyond profiles}\label{sec:beyondProfile}
A more careful analysis of specific weighted subgraphs yields the following result.

\begin{theorem}\label{thm:hypQmonotone}
	Let $X$ be a hyperbolic graph whose boundary has a visual metric that is Ahlfors $Q$-regular for some $Q>1$ and admits a $1$-Poincar\'e inequality (hence has conformal dimension $Q$). 
	Let $Y$ be a graph with subexponential growth and let $H$ be a visual hyperbolic graph.  If there is a regular map $X\to H\times Y$, then the conformal dimension of $H$ is at least $Q$.
\end{theorem}
Recall that a hyperbolic space $H$ is \textbf{visual} if every point in $H$ is within a uniformly bounded distance from a geodesic ray from a base point of $H$.  Any discrete hyperbolic group is visual.
\begin{proof}
It suffices to assume that $X$ is visual by discarding any irrelevant points.

	Following the proof of \cite[Theorem 11.1]{HumeMackTess}, we can replace $X$ by hyperbolic cone graph $\Gamma$ (quasi-isometric to $X$) which has the following key property: 
	for each $t \in \N$, there is a subgraph $\Gamma_t$ in the ball of radius $t$ about the base point of $\Gamma$ which has $\asymp e^{Qt}$ vertices, and so that $h^1(\Gamma_t) \succeq e^{-t}$, by \cite[(11.2)]{HumeMackTess}.
	There is a subtlety here: \cite[(11.2)]{HumeMackTess} bounds $h^1_C(\Gamma_t)$ from below, where the gradient of $f:V\Gamma_t\to \R$ in the definition of $h^1_C$ is 
	$|\nabla_C f|(x)=\max\{|f(x')-f(x'')| : x',x''\in B(x,C)\cap V\Gamma_t\}$.
	Spaces with Poincar\'e inequalities are quasi-convex, so any two points at distance $C$ in $\Gamma_t$ can be connected by a path in $\Gamma_t$ of uniformly bounded length.  Thus the argument of \cite[Proposition 4.3]{HumeMackTess} gives that $h^1_C(\Gamma_t) \asymp h^1(\Gamma_t)$, essentially by the triangle inequality.

Suppose for a contradiction that there is a regular map $f: \Gamma \to H\times Y$ as in the statement of the theorem, where the conformal dimension of $H$ equals $Q' < Q$.  
	Choose $Q'' \in (Q', Q)$, then there is an Ahlfors $Q''$-regular space quasisymmetric to $\bdry H$, and let $H'$ be a hyperbolic cone on that space.  By \cite{BS-00-gro-hyp-embed}, $H'$ is quasi-isometric to $H$, so we get a regular map $f': \Gamma \to H' \times Y$.
	Now \cite[Corollary 5.10]{HumeMackTess} gives that $h^1(\Gamma_t) \preceq h^1_C(f'(\Gamma_t))$ for a constant $C$.
	Taking a suitable neighbourhood $\Gamma_t':=[f'(\Gamma_t)]_M$ for some $M \geq 1$ of the image $f'(\Gamma_t)$ (essentially to make it connected), as in \cite[\S 5.2]{HumeMackTess}, we have $h^1_C(f'(\Gamma_t)) \asymp h^1(\Gamma_t')$ again by the argument of \cite[Proposition 4.3]{HumeMackTess}.
	Lemma~\ref{lem:compare-cap-poinc-profiles} gives $h^1(\Gamma_t') \preceq C^{1,\alpha}(\Gamma_t',\#)$, where $\#$ is the counting measure, and $\alpha>0$ is any value small enough.

	We require the following lemma.
	\begin{lemma}\label{lem:wt-cap-projection}
	The projection $\pi:H' \times Y \to H'$ is monotone with respect to capacities: if $\Gamma' \subset V(H' \times Y)$ and $\pi_*\#$ is the push-forward measure on $\pi(\Gamma')$, 
		then $C^{p,\alpha}(\Gamma', \#) \leq C^{p,\alpha}(\pi(\Gamma'),\pi_*\#)$.
	\end{lemma}
	\begin{proof}
		For any function $h: V(\pi\Gamma') \to \R$,
		$\pi_*\# (\{ h \leq 0\}) = \#(\{ h\circ\pi \leq 0\})$
		and
		$\pi_*\# (\{ h \geq 1\}) = \#(\{ h\circ\pi \geq 1\})$.
		Meanwhile, $\| \nabla (h \circ \pi)\|_{\#,p}^p \leq \| \nabla h \|_{\pi_*\#, p}^p$ as for each $x \in V\Gamma'$, $|\nabla(h\circ\pi)|(x) \leq |\nabla h|(\pi(x))$.		
	\end{proof}
	Combining our results so far, we have for each $t$ that
	\[
		e^{-t} \preceq h^1(\Gamma_t) \preceq h^1(\Gamma_t') \preceq C^{1,\alpha}(\Gamma_t', \#) \preceq C^{1,\alpha}(\pi(\Gamma_t'),\pi_*\#).
	\]

	Since $\Gamma_t$ lies in a ball of radius $t$, $\Gamma_t'$ also lies in a ball of radius $\leq At$ for some constant $A \geq 1$ depending on $f'$.
	As $Y$ has subexponential growth, for each $\delta>0$ there exists $R_\delta$ so that every ball in $Y$ of radius $R\geq R_\delta$ contains at most $\exp(\delta R)$ points.
	We fix $\delta=\delta(Q,Q'',A)$ later in the proof, and assume $t \geq R_{\delta}$,
	so the fibres of $\pi|_{\Gamma_t'}:\Gamma_t' \to \pi(\Gamma_t')$ have at most $k:=\exp(\delta At)$ points in them.

	Therefore, for this value of $k$ Theorem~\ref{prop:wt-prof-hyp-cone} gives for $\alpha$ small enough that
	\begin{align*}
		C^{1,\alpha}(\pi(\Gamma_t'), \pi_*\#) 
		& \leq \pi_*\#(\pi(\Gamma_t'))^{-1} \Xi^{1,\alpha,k}_{H'}( \pi_*\#(\pi(\Gamma_t')))
		\\ & \preceq \exp(-Qt) k^{\frac{1}{Q''}} \exp(Qt)^{1-\frac{1}{Q''}}
		\\ & = \exp\left( -Qt+\frac{\delta A t}{Q''} +Qt - \frac{Q}{Q''}t \right).
	\end{align*}
	Since $Q''<Q$, we can choose $\delta>0$ so that $(Q-\delta A)/Q''>A'>1$ for some $A'$.
	But then as $e^{-t} \preceq e^{-A't}$ we have a contradiction.
\end{proof}

\begin{remark} Intuitively what is happening here is that $\cut(\Gamma_t) \asymp e^{(Q-1)t}$, while the weighted projection $\pi(\Gamma_t')$ in $H$ can be cut with weight $\preceq e^{(Q'-1)tQ/Q'}$, up to a subexponential factor, and cuts of $\pi(\Gamma_t')$ can be lifted back to cuts of $\Gamma_t$ (cf.\ Lemma~\ref{lem:wt-cap-projection}).

	This argument does not work when $Q=1$ because the cut sets of balls grow too slowly. For example, this argument cannot rule out a regular map $f:\HH_\R^2 \to T\times \Z^d$: considering balls $\Gamma_t=B(t)$ of volume $e^t$, using \cite[Theorem 11.3]{HumeMackTess} and Proposition~\ref{prop:wt-prof-tree} the argument gives
	\[
		te^{-t} \preceq h^1(\Gamma_t) \preceq C^{1,\alpha}(\pi(\Gamma_t')) \preceq e^{-t} k \preceq e^{-t} t^d,
	\]
	which is no contradiction.
\end{remark}

We are now able to prove the following non-embedding result, of particular interest for the products of rank 1 symmetric spaces $\HH^m_\bbK$ or Fuchsian buildings $G_{m,n}$ with nilpotent groups.
\begin{corollary}[{Theorem~\ref{thmIntro:directProductCE}}]
	\label{cor:prod-no-embed}
Assume $G_1=H_1\times P_1$ and  $G_2=H_2\times P_2$, where for $i=1, 2:$
\begin{itemize}
\item $H_i$ is a non-elementary finitely generated hyperbolic group of conformal dimension $Q_i \geq 0$, and
\item $P_i$ is a locally compact group with polynomial growth of degree $d_i \geq 0$.
\end{itemize}
 If there exists a regular map $G_1\to G_2$, then $d_1 \leq d_2$. Moreover, if $H_1$ has its conformal dimension $Q_1>1$ attained by a metric admitting a $1$-Poincar\'e inequality, then $Q_1\leq Q_2$.
\end{corollary}
\begin{proof}
	Any such $H_1$ contains a quasi-isometrically embedded $3$-regular tree, so by the monotonicity of $\Lambda^p$ and Corollary~\ref{cor:treexpoly-lower} we have $\Lambda^p_{G_1}(r) \gtrsim r^{1-\frac{1}{p+d_1}}$ for $p \geq 1$.
	Theorem~\ref{thm:hyp-times-nilp-upper-bound} applied to $G_2$ gives for $p > Q_2$ that $\Lambda^p_{G_2}(r) \lesssim r^{1-\frac{1}{p+d_2}}$, so the $d_1 \leq d_2$ conclusion follows.  
	The `moreover' statement is given by Theorem \ref{thm:hypQmonotone}.
\end{proof}

\section{Concluding steps}\label{sec:concludingsteps}

In this section we complete the proofs of most of the theorem stated in the introduction:
Theorems \ref{thmIntro:alg thick Implies an thick}, 
\ref{thmIntro:unimodularDicho},  \ref{thmIntro:geomDichoThin}, \ref{thmIntro:geomDichoThick} and Corollary  \ref{corIntro:polycyclic},.

\begin{proof}[Proof of Theorem  \ref{thmIntro:alg thick Implies an thick}]
Let $G$ be an algebraically thick connected Lie group. By Theorem \ref{thm:reduc}, we can assume that $G$ is linear and its radical is real-triangulable. Then by Proposition \ref{prop:algThick}, we deduce that it has a closed subgroup isomorphic to either $\SOL_a$, $a>0$ or $\Osc$.
By Proposition \ref{thm:embSgraph}, if it contains $\SOL_a$, then it contains a coarsely embedded copy of $\DL(2,2)$, and it is analytically thick by Theorem \ref{DLlowerbound}. If it contains $\Osc$, then it is analytically thick by Theorem \ref{thm:HeisExt}.
	We conclude by Corollary~\ref{cor:general-upper-bound}.
\end{proof}

\begin{proof}[Proof of Theorem \ref{thmIntro:unimodularDicho}]
By  Theorem  \ref{thmIntro:alg thick Implies an thick}, and since algebraic thinness is by definition the negation of algebraic thickness, it is enough to prove that if $G$ is a connected unimodular Lie group that is algebraically thin, then it is analytically thin. 
By Corollary \ref{cor:algthinGeneral}, $G$ either has polynomial growth, in which case this follows from \cite[Theorem 8.1.]{HumeMackTess}, or we may assume it is a direct product $R \times S$ of group $R$ with polynomial growth and a $\R$-rank $1$ simple Lie group $S$ with finite center. Picking a uniform lattice $\Gamma$ in $S$, $G$ is therefore quasi-isometric to $R \times \Gamma$, and we conclude by Theorem \ref{thm:hyp-times-nilp-upper-bound}.
\end{proof}

\begin{proof}[Proof of Corollary \ref{corIntro:polycyclic}]
	Recall that a polycyclic group has a finite index subgroup that embeds as a uniform lattice in a solvable unimodular connected Lie group $G$. If $\Gamma$ (or equivalently $G$) is algebraically thin, then by Proposition \ref{prop:algthinDescrip}, $G$ must have polynomial growth, and so is analytically thin \cite{HumeMackTess}.  If $\Gamma$ is algebraically thick the conclusion follows from Theorem~\ref{thmIntro:unimodularDicho}. 
\end{proof}

\begin{proof}[Proof of Theorem \ref{thmIntro:geomDichoThin}]
Assouad's embedding theorem gives (iv)$\implies$(iii), and (iii)$\implies$(ii) is obvious.
The separation profile is monotonous under regular maps, and $P\times\HH_\R^n$ is analytically thin by Theorem \ref{thm:hyp-times-nilp-upper-bound}, so we deduce that  (ii)$\implies$(i).
Finally, assume $G$ is algebraically thin, then by Corollary \ref{cor:algthinGeneral}, it is quasi-isometric to a direct product (with one factor possibly empty) $P\times S$, where $P$ has polynomial growth, and $S$ is simple of rank $1$ with finite center. Having finite center, $S$ is Gromov-hyperbolic, so by Bonk-Schramm's theorem \cite{BS-00-gro-hyp-embed}, it quasi-isometrically embeds into $\HH_\R^n$ for any large enough $n$. Hence (i)$\implies$(iv).  \end{proof}

\begin{proof}[Proof of Theorem \ref{thmIntro:geomDichoThick}]
(iv)$\implies$(iii)$\implies$(ii) are obvious. By Theorem \ref{thm:reduc}, we can assume that $G$ is linear and that its radical is real-triangulable. Then (i)$\implies$(iv) follows from Proposition \ref{prop:algThick}. Let us prove (ii)$\implies$(i). Since the separation profile is monotonous under regular maps, and is $\simeq n/\log n$ for  $\SOL_a$, $a>0$ and $\Osc$, we deduce that $G$ is analytically thick, and we conclude using Theorem \ref{thmIntro:unimodularDicho}.
\end{proof}

\section{Questions}\label{sec:qu}
This work raises many natural questions. We selected a few of them below.

\subsection{Thick/thin dichotomies}
A natural question that we left open is: what happens for non-unimodular connected Lie groups?  
We risk the following conjecture.
\begin{conjecture}\label{conj:thin/thickLie}
	Theorem \ref{thmIntro:unimodularDicho} holds in full generality (without assuming unimodular).
\end{conjecture}
As we mentioned in \S \ref{sectionIntro:nonunimodular}, an interesting example to look at is the semidirect product $G=\Heis_3\rtimes_{(1,0,1)}\R$. Such a group is non-unimodular, algebraically thin, and one can check by a Lie algebra computation that it is not a subgroup of a direct product of a Heintze group by a group of polynomial growth. Our conjecture would imply that $\Lambda_G^1(r)\lesssim r^{\alpha}$ for some $\alpha<1$.

\begin{question}\label{qu:algthickthin}
Is every finitely generated group either analytically thick or analytically thin?
\end{question}
We expect the answer to this question to be negative in general, more precisely, groups containing a very rapidly growing sequence of expanders, and certain elementary amenable lacunary hyperbolic groups (whose separation profiles were considered in \cite{HumSepExp} and \cite{HumeMack} respectively) are likely to provide counterexamples. However, there is at present no obvious candidate for a finitely presented counterexample. Despite this, there are likely to be many natural situations where this result does hold. For instance, a popular conjecture  \cite[Question 1.1]{BestQ} asserts that a group $G$ with finite classifying space and no $\Z^2$ subgroup either contains a subgroup isomorphic to $\BS(m,n)$ with $\abs{m}\neq\abs{n}$ (and therefore is analytically thick) or is hyperbolic (so analytically thin).

It follows from \cite{CozGour} that for every finitely generated solvable group with exponential growth, and every $\varepsilon>0$ one has $\Lambda^p_G(r_n)\gtrsim r_n^{1-\varepsilon}$ on an unbounded sequence $(r_n)$. In particular, a finitely generated solvable group is analytically thin if and only if it has polynomial growth.
In view of our result for polycyclic groups, a positive answer to the following question seems plausible.
\begin{question}\label{question:thin/thckSolvable} Does the analytically thick/thin dichotomy holds for the class of solvable finitely generated groups? In other words, does the bound $\Lambda^p_G(r)\gtrsim r/\log(r)$ hold for all finitely generated solvable groups of exponential growth? \end{question}

Here is another reasonable class of groups where to expect a positive answer:

\begin{question} 
Does the analytically thick/thin dichotomy holds for the class of linear finitely generated groups? 
 \end{question}

 \

\subsection{Groups admitting an embedding of $\DL(2,2)$}
A positive answer to the following question would in particular implies a positive answer to Question \ref{question:thin/thckSolvable}:
\begin{question} 
Does $\DL(2,2)$ regularly map to any finitely generated solvable group with exponential growth?
 \end{question} 
 We suspect the answer is negative, a possible counter-example being (a lattice in) the group $\Osc$.
Note that a positive answer for this specific example would provide an alternative (potentially less technical) proof of Theorem  \ref{thm:HeisExt}.

\subsection{Conformal dimension}

It is natural to wonder whether Theorem \ref{thmIntro:Ahlreg} is sharp in the following sense.
\begin{question} Let $Z$ be a connected non-discrete Ahlfors-regular metric space and let $X$ be a hyperbolic cone over $Z$. Is
\[\inf\setcon{p\geq 1}{\Lambda^p_{X}(r)\simeq r^{1-\frac{1}{p}}}\]
equal to the (Ahlfors regular) conformal dimension of $Z$?
\end{question}
As an example, take $Z$ to be the middle-thirds Sierpinski carpet, whose Ahlfors-regular conformal dimension $Q$ is currently unknown. Let $X$ be a hyperbolic cone over $Z$. By \cite{GladShum}, we know that $\Lambda^1_X(r)\not\simeq r^{1-1/Q}$, however the question above is still open in this case.

 Another natural example to consider would be the case of Heintze groups.  More generally, 
\begin{question} What are the Poincar\'e profiles of Heintze groups $E \rtimes \R$?
\end{question}
For example, what are the Poincar\'e profiles of $G = \R^2\rtimes_{(1,2)} \R$?  We have, see Figure~\ref{fig:heintze-ex}:
\begin{alignat*}{3}
	r^{1/2} \lesssim \ & \Lambda_G^p(r) && \lesssim r^{2/3} && \text{ if } 1 \leq p \leq 2 \\
		r^{1- \frac{1}{p}} \lesssim \ & \Lambda_G^p(r) && \lesssim r^{2/3} && \text{ if } 2 \leq p < 3 \\
	r^{2/3} \lesssim \ & \Lambda_G^p(r) && \lesssim r^{2/3}\log^{1/3}(r) \quad && \text{ if } p=3 \\
	 & \Lambda_G^p(r) && \simeq r^{1- \frac{1}{p}} && \text{ if } 3<p<\infty;
\end{alignat*}
here the lower bounds come from the embedded $\R^2$ and embedded binary trees, and the upper bounds follow as in Theorem~\ref{thm:hyp-times-nilp-upper-bound} from $\bdry G$ attaining its conformal dimension $3$, but the exact asymptotics remain unclear.
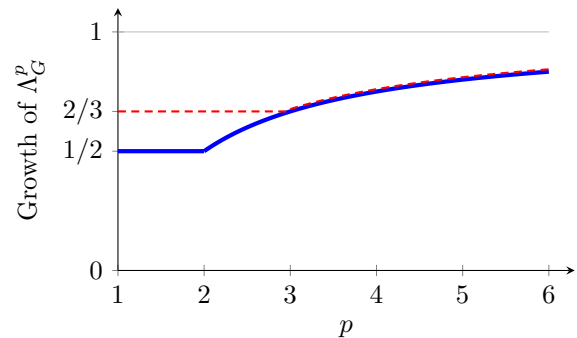
\begin{figure}\label{fig:heintze-ex}
\begin{tikzpicture}
 \begin{axis}[
	axis lines = left,
    xmin = 1, xmax = 6.3,
    ymin = 0, ymax = 1.1,
    xtick distance = 1,
	ytick={0,0.5,0.6666667,1},
	 yticklabels={$0$,$1/2$,$2/3$,$1$},
    width = 0.6\textwidth,
    height = 0.4\textwidth,
    xlabel = {$p$},
    ylabel = {Growth of $\Lambda_G^p$},]
 
\addplot[
    domain = 1:6,
    samples = 200,
    smooth,
    thin,
    lightgray,
] {1};
\addplot[
    domain = 1:2,
    samples = 200,
    smooth,
    ultra thick,
    blue,
] {1/2};
\addplot[
    domain = 2:6,
    samples = 200,
    smooth,
    ultra thick,
    blue,
] {1-(1/x)};
\addplot[
    domain = 1:3,
    samples = 200,
    smooth,
    thick,
	densely dashed,
    red,
] {2/3};
\addplot[
    domain = 3:6,
    samples = 200,
    smooth,
    thick,
	densely dashed,
    red,
] {1-(1/x)+0.01};

\end{axis}\end{tikzpicture}
	\caption{Growth rates of $\Lambda_G^p$: upper bound red dashed, lower bound thick blue}
\end{figure}

\def\cprime{$'$}

\end{document}